\documentclass{amsart}
\usepackage{amsthm}
\RequirePackage{amssymb, amsfonts, amsmath, latexsym, verbatim, xspace, setspace, float, mathtools,verbatim, enumerate, physics, bbm}
\usepackage[toc,page]{appendix}
 \usepackage[all]{xy}
\usepackage{float, stmaryrd}
\usepackage{tikz-cd}
\usepackage[outline]{contour}
\contourlength{2pt}
\usetikzlibrary{decorations.pathmorphing}
\tikzset{snake it/.style={decorate, decoration=snake}}
\usetikzlibrary{shapes, arrows, calc, intersections, matrix,decorations.pathreplacing}

\usepackage{mathrsfs}

\usepackage{hyperref}

\newtheorem{theorem}{Theorem}[section]
\newtheorem{lemma}[theorem]{Lemma}

\newtheorem{prop}[theorem]{Proposition}

\theoremstyle{definition}
\newtheorem{definition}[theorem]{Definition}
\newtheorem{example}[theorem]{Example}

\theoremstyle{remark}
\newtheorem{remark}[theorem]{Remark}

\numberwithin{equation}{section}

\newcommand{\A}{\mathcal{A}}
\newcommand{\B}{\mathcal{B}}
\newcommand{\C}{\mathcal{C}}
\newcommand{\E}{\mathcal{E}}

\newcommand\restr[2]{{
  \left.\kern-\nulldelimiterspace 
  {#1} 
  \vphantom{\big|} 
  \right|_{#2} 
  }}

\newcommand{\midarrow}{\tikz \draw[->] (0,0) -- +(.1,0);}

\tikzset{
    ncbar angle/.initial=90,
    ncbar/.style={
        to path=(\tikztostart)
        -- ($(\tikztostart)!#1!\pgfkeysvalueof{/tikz/ncbar angle}:(\tikztotarget)$)
        -- ($(\tikztotarget)!($(\tikztostart)!#1!\pgfkeysvalueof{/tikz/ncbar angle}:(\tikztotarget)$)!\pgfkeysvalueof{/tikz/ncbar angle}:(\tikztostart)$)
        -- (\tikztotarget)
    },
    ncbar/.default=0.5cm,
}
\tikzset{square left brace/.style={ncbar=0.5cm}}
\tikzset{square right brace/.style={ncbar=-0.5cm}}

\tikzset{round left paren/.style={ncbar=0.5cm,out=120,in=-120}}
\tikzset{round right paren/.style={ncbar=0.5cm,out=60,in=-60}}

\DeclareFontFamily{U}{wncy}{}
    \DeclareFontShape{U}{wncy}{m}{n}{<->wncyr10}{}
    \DeclareSymbolFont{mcy}{U}{wncy}{m}{n}
    \DeclareMathSymbol{\Sh}{\mathord}{mcy}{"58}

\usepackage{geometry}
\allowdisplaybreaks

\usepackage{ifthen}
\newcommand{\ptrans}[3]
{
  \ifthenelse{\equal{#2}{} \and \equal{#3}{}}
  	{  \ifthenelse{\equal{#1}{}}{P}{P(#1)}	 }
	{  \ifthenelse{\equal{#1}{}}{P|^{#3}_{#2}}{P(#1)|^{#3}_{#2}}}
}

\newcommand\wt[1]{\widetilde{#1}}

\renewcommand\={:=}  	
\newcommand\bt{\mathbf{t}}
\newcommand\iso{\cong}
\renewcommand\o{\circ}
\newcommand{\NGA}{\nabla\mathrm{GA}}
\newcommand{\DGA}{\mathrm{DGA}}
\DeclareMathOperator{\End}{End}
\DeclareMathOperator{\Hom}{Hom}
\DeclareMathOperator{\Mat}{Mat}
\DeclareMathOperator{\Vect}{Vect}

\begin{document}

\begin{abstract}
We introduce a new variant of Hochschild's two-sided bar construction for the setting of curved differential graded algebras.  One can geometrically think of the classical bar complex as elements from the algebra positioned along different points in the closed interval $[0,1]$.  In this paper, we start with a curved differential graded algebra and define a new ``zigzag algebra'' that, informally, consists of algebra elements arranged on a zigzag of intervals going back and forth between $0$ and $1$.  We focus on two curved differential graded algebras: the de Rham algebra of differential forms with values in the endomorphism bundle associated to a vector bundle with connection, and its induced zigzag algebra.  We define a curved version of Chen's iterated integral that incorporates parallel transport and maps this zigzag algebra of bundle-valued forms to bundle-valued forms on the path space.  This iterated integral is proven to be a homotopy equivalence of curved differential graded algebras, and for real-valued forms it factors through the usual Chen iterated integral.
\end{abstract}

 \pagestyle{headings}
\title{Modeling bundle-valued forms on the path space with a curved iterated integral}
\author[C. Glass]{Cheyne J. Glass}
\author[C. Redden]{Corbett Redden}

\maketitle 

\section{Introduction}\label{section introduction}

Consider the classical situation where $\A$ is an arbitrary commutative associative differential graded algebra.  Associated to $\A$ is its 2-sided bar complex, which we denote by $CH^I(\A)$ and refer to as the ``interval Hochschild complex.''   This 2-sided bar complex is also a commutative dga and provides an algebraic version of the path space of $\A$.  If $M$ is a smooth manifold, and $\A=\Omega(M)$ is the de Rham algebra of differential forms, then Chen's iterated integral \cite{C} provides an explicit homotopy equivalence of $\DGA$s
\[ It: CH^I\big(\Omega(M)\big) \overset{\simeq}\longrightarrow  \Omega(PM). \]
This powerful and widely used tool allows differential forms on the infinite-dimensional path space $PM$ to be modeled by tensor products of ordinary forms on $M$.  Variants of the construction, such as the  Hochschild complex and the cyclic bar complex, allow one to similarly model structures on the free loop space $LM$.  For example, when $(E,\nabla^E)\to M$ is a vector bundle with connection, Getzler--Jones--Petrack \cite{GJP} and Tradler--Wilson--Zeinalian \cite{TWZ} use this method to construct the Bismut--Chern character forms in $\Omega_{S^1}(LM)$, refining the ordinary $\Omega(M)$-valued Chern character.  

While the details of  \cite{GJP, TWZ} are not used in this paper, they provide one of the motivations for our work.  The Chern--Weil forms in $\Omega(M)$ are obtained by taking traces of powers of the curvature element $R \in \Omega^2(M,\E)$, where $\E \= \End (E) = \Hom(E,E)$ is the endomorphism bundle over $M$.  Much of the work in those papers, especially \cite{TWZ}, occurs {\it pre-trace} while working with tensor powers of bundle-valued forms $\Omega(M,\E)$.  It is tempting, but incorrect, to think that Chen's iterated integral would give a similar equivalence of dgas from $CH^I\big(\Omega_{\nabla}(M,\E)\big)$ to $\Omega_{\wt{\nabla}}(PM, \E_0)$, where $(\E_0, \wt{\nabla}) \= ev_0^*(\E,\nabla) \to PM$ is the pullback of $\E$ along paths' initial points.  In this paper, we provide appropriate modifications and transform that naive hope into a theorem.

The most obvious subtlety is that $\Omega(M)$ is a commutative differential graded algebra, but the collection of bundle-valued forms $\Omega_{\nabla}(M,\E)$ is usually not for two  reasons.  First, the graded algebra structure on $\Omega_{\nabla}(M,\E)$ is no longer commutative when $\dim E > 1$.  Second, if the connection $\nabla$ on $\E$ is not flat, then $\nabla^2 \neq 0$, so $\Omega_{\nabla}(M,\E)$ is not a cochain complex.
For this reason, we work in the more general context of (not necessarily commutative) {\it curved dgas}, where the differential is not required to square to zero.

Section \ref{section curved dga} is a brief introduction to the category of curved dgas, which we denote by $\NGA$.  One important note is that our category $\NGA$ only contains morphisms that commute with the differential, as opposed to the more general morphisms in sources such as \cite{Block, BD, P}.  In particular, two non-isomorphic connections on $E$  give rise to distinct $\Omega_{\nabla^i}(M, \E)$  that are not isomorphic in $\NGA$.  Section \ref{section cohomology} introduces homotopy equivalence and cohomology in $\NGA$.  We reuse, symbol for symbol, the standard definition of chain homotopy from cochain complexes to define homotopy equivalences in $\NGA$.  A priori, the cohomology of $\mathcal{A} \in \NGA$ is not defined since  $\nabla^2\neq 0$.  We account for this failure and construct what we call the {\it curved cohomology} algebra $H^\bullet_{cur}(\mathcal{A})$; it is equivalent to the cohomology of the maximal cochain subcomplex of $\mathcal{A}$.  As expected, a homotopy equivalence between two objects in $\NGA$ induces an isomorphism in curved cohomology.  At this point, we now have the language to say whether something is equivalent to $\Omega_{\wt{\nabla}} (PM, \E_0)$ in $\NGA$.

One might expect to simply work with $CH^I(\Omega_{\nabla}(M,\E))$, since the usual 2-sided bar construction does make sense for non-commutative DGAs.  However, the commutativity of $\Omega(M)$ is required for Chen's iterated integral to be a morphism of algebras.  This can be seen in the following formal calculation, where all $\pm$ signs are suppressed and where $\Omega$ denotes either $\Omega(M)$ or $\Omega(M, \Mat_{n\times n}(\mathbb{R}))$.  See Subsection \ref{subsec: define It} for a more precise discussion.  Let $\alpha, \beta \in \Omega$ with induced $\underline{\alpha} = 1\otimes \alpha \otimes 1$, $ \underline{\beta}=1\otimes\beta\otimes 1 \in CH^I(\Omega)$.  
The image of $\underline{\alpha}$ under the usual Chen iterated integral map is given by \label{intro calculation}
\[ 
\begin{tikzcd}[row sep=0em]
1& \alpha & 1 \\
    {}  \arrow[r, mapsfrom, maps to, start anchor={center}, end anchor={center}]& {}\arrow[r, mapsfrom, maps to, start anchor={center}, end anchor={center}]& {}
\end{tikzcd} 
\quad \longmapsto \quad It(\underline{\alpha})= \int_0^1 \iota_{\bt} \alpha(t) \, dt.
 \]
When $\odot$ is the usual shuffle product on $CH^I(\Omega)$, the product
\[ \underline{\alpha} \odot \underline{\beta} \quad = \quad 
\begin{tikzcd}[row sep=0em, column sep=tiny]
1& \alpha & \beta &1 \\
    {}  \arrow[r, mapsfrom, maps to, start anchor={center}, end anchor={center}]& {}\arrow[r, mapsfrom, maps to, start anchor={center}, end anchor={center}]& {} \arrow[r, mapsfrom, maps to, start anchor={center}, end anchor={center}]&{}
\end{tikzcd} \; + \;
\begin{tikzcd}[row sep=0em, column sep=tiny]
1& \beta & \alpha &1 \\
    {}  \arrow[r, mapsfrom, maps to, start anchor={center}, end anchor={center}]& {}\arrow[r, mapsfrom, maps to, start anchor={center}, end anchor={center}]& {} \arrow[r, mapsfrom, maps to, start anchor={center}, end anchor={center}]&{}
\end{tikzcd} \]
is mapped by the iterated integral to
\[
It(\underline{\alpha}\odot \underline{\beta}) = \int_0^1 \int_{t_1}^1 \iota_{\bt} \alpha(t_1) \wedge \iota_{\bt} \beta(t_2) \, dt_2 \,dt_1 \, + \,
\int_0^1 \int_{t_1}^1 \iota_{\bt} \beta(t_1) \wedge \iota_{\bt} \alpha(t_2) \, dt_2 \,dt_1 \, .
\]
On the other hand, 
\begin{align*}
It(\underline{\alpha}) \wedge It(\underline{\beta}) &= \int_0^1 \int_0^1 \iota_{\bt} \alpha(t_1) \wedge \iota_{\bt} \beta(t_2) \, dt_2\, dt_1 \\
& = \int_0^1 \int_{t_1}^1 \iota_{\bt} \alpha(t_1) \wedge \iota_{\bt} \beta(t_2) \, dt_2\, dt_1 \,+ \,
\int_0^1 \int_0^{t_1} \iota_{\bt} \alpha(t_1) \wedge \iota_{\bt} \beta(t_2) \, dt_2\, dt_1 \,.
\end{align*}
If $\alpha$ and $\beta$ commute, then $It(\underline{\alpha}\odot \underline{\beta}) = It(\underline{\alpha}) \wedge It(\underline{\beta})$ because the terms $\int_0^1\int_{t_1}^1 \iota_{\bt} \beta(t_1) \wedge \iota_{\bt} \alpha(t_2) \, dt_2 \,dt_1$ and $\int_0^1 \int_0^{t_1} \iota_{\bt}\alpha(t_1) \wedge \iota_{\bt}\beta(t_2) \, dt_2 \, dt_1$ are equal through a change of variables.  However, $\alpha$ and $\beta$ do not commute when they are matrix-valued, and no change of variables will alter the order of multiplication.  For this reason, we create a structure that separates the interval's time ordering from the multiplication ordering.

We call this structure the {\it zigzag algebra} and devote Section \ref{section zigzag} to its construction.  Definition \ref{def CH^ZZ} and Theorem \ref{theorem zigzag is NGA} combine to produce a functor $ZZ:\NGA \to \NGA$;  given any $\A \in \NGA$, there is a well-defined $ZZ(\A) \in \NGA$.  The product structure on $ZZ(\A)$ is an enhancement of the familiar shuffle product, and there is a natural homotopy equivalence $\A \simeq ZZ(\A)$.  The full definition contains several details and subtleties, but the geometric idea is simple.  One can view monomials in the classical 2-sided bar complex $CH^I(\A)$ as elements of $\A$ ordered along the interval $I=[0,1]$.  In the zigzag complex, the interval has been replaced by a sequence of connected intervals that unfold like a staircase.  (In this paper, zigzag is used merely to describe the appearance of monomials, and its use should not be confused with the zigzags of morphisms that appear when localizing a category.)

Section \ref{section iterated integral} focuses on the specific case of $\A = \Omega_{\nabla}(M,\E)$.  We define the desired iterated integral map
\begin{equation}\label{eq:introIt} It: ZZ\big(  \Omega_{\nabla}(M,\E) \big)  \longrightarrow \Omega_{\wt{\nabla}}(PM, \E_0), 
\end{equation}
prove that it is a morphism in $\NGA$, and prove that is a homotopy equivalence.  The parallel transport of the connection plays an important role in the map, and we first prove that a few key formulas from $\Omega(M)$ still hold (appropriately modified) for $\Omega_{\nabla}(M,\E)$.  We attempt to work with bundles in a coordinate-free manner whenever possible; one advantage is that the specific construction of \eqref{eq:introIt} is practically forced on us as the only sensible option.  

Finally, Section \ref{section finale} summarizes the main results and shows how they recover or reprove previously known classical results in the case where $\A \in \NGA$ arises from a commutative dga.  Theorem \ref{theorem zigzag is NGA}, Theorem \ref{thm It is NGA equivalence}, and Proposition \ref{prop Col It commute equiv}, respectively, combine to give the following summary of our paper's results.

\begin{theorem}\
\noindent\begin{enumerate}[(1)]
\item The zigzag construction defines a functor $ZZ:\NGA \to \NGA$, and there are natural homotopy equivalences $\A \simeq ZZ(\A)$.
\item Given any vector bundle with connection $(E,\nabla) \to M$, there is a well-defined iterated integral map \eqref{eq:introIt} that is a morphism of $\NGA$s and sits in the following commutative diagram of homotopy equivalences.
\[ \xymatrix{
\Omega_{\nabla}(M,\End(E)) \ar[rr]_{\simeq}^{ev_0^*} \ar[dr]^{\simeq} &&\Omega_{\wt{\nabla}}(PM, ev_0^* \End(E))  \\
&ZZ\big( \Omega_{\nabla}(M, \End(E)) \big) \ar[ur]^{\simeq}_{It}
}\]
\item Let $\A \in \DGA$ be a commutative dga with $CH^I(\A) \in \DGA \hookrightarrow \NGA$.  Then, the column-collapse map $Col: ZZ(\A) \to CH^I(\A)$ is a homotopy equivalence  in $\DGA \hookrightarrow \NGA$. 
\end{enumerate}
\end{theorem}
Note that even in the case of a non-commutative curved dga, $\mathcal{A} = \Omega_{\nabla}(M, \E)$, one could have used the usual 2-sided bar construction with a small change to its differential, and then our iterated integral $It: CH^{I}(\mathcal{A}) \to ev_0^* \mathcal{A}$ (which takes parallel transport into account) would commute with the differentials. However, in this case there is no hope for any algebra structures being preserved, and furthermore the map $Col: ZZ(\mathcal{A}) \to CH^I(\mathcal{A})$ would not commute with the differentials. In this sense the results of Section \ref{section finale} are indeed special to the case where $\A$ is a commutative differential graded algebra.

There are several natural questions and possible directions for future work.  In this paper, we only considered a structure that is analogous to the 2-sided bar complex and models $\Omega_{\wt{\nabla}}(PM, \E_0)$.  The next step that inspired most of this paper is to model $\Omega(LM, \E_0)$ via a structure analogous to the Hochschild complex. In a different direction is modeling $\Omega(P^2M, \mathcal{G})$, the gerbe-valued forms on the $2$-path space of a manifold using a 2-dimensional analog of the zigzag model described in this paper. For abelian ($S^1$-banded) gerbes this was accomplished in \cite{TWZ} by, informally speaking, using a variant of the Hochschild complex where the interval $[0,1]$ is replaced by a square $[0,1]\times[0,1]$.  Similar higher-dimensional analogs of the Hochschild complex were described in \cite{MR2721877}.  The first step in extending the work in the current paper to such a 2d non-abelian analog would be applying Lemma 4.2 from \cite{Glass} to obtain the appropriate version of this paper's Proposition \ref{prop d hol}. A suggested 2d analog of the zigzag model was first described in the PhD thesis \cite{MillerThesis}.  From there, the goal would be to proceed to modeling gerbe-valued forms on 2d analogs of the loop space.  Finally, in all of these future directions, there is a chance to place the correct model structure on the relevant categories ($\NGA$ in this paper) so that the algebraic models for forms on the path, loop, square, and sphere spaces are merely explicit cofibrant replacement constructions that one can use to present the derived hom in curved settings.

It should be noted that a first version of this project appeared on the arXiv as \cite{1505.03192}, however the improvements in this version were so significant it seemed appropriate to treat it as a completely separate paper. 

\textbf{Acknowledgements.}
C.G. would like to thank the Max-Planck-Institut f{\" u}r Mathematik in Bonn, for their support during his visit, as well as St. Joseph's College for their support with a Faculty Summer Research Grant.  The authors would also like to thank Thomas Tradler for inspiring the use of zigzags, as well as Jorge Florez, Joey Hirsch, and Scott Wilson for their helpful comments regarding curved cohomology. 

\section{Curved DGAs}\label{section curved dga}
The Chen iterated integral is a map of differential graded algebras, and the majority of this paper will occur in the more generalized setting of curved differential graded algebras. Following \cite{BD}, we introduce \emph{curved dgas} in this section and provide relevant examples.

\begin{definition}\label{def curved dga}
A \emph{\underline{curved dga}} is a triple, $\left(  \mathcal{A}^{\bullet}, \nabla, R \right)$, where $\mathcal{A}^{\bullet}$ is an $\mathbb{N}$-graded, unital, associative algebra over a field $\mathbbm{k}$, $R \in \mathcal{A}^2$, and $\nabla$ is a degree-one map 
\[ \nabla: \mathcal{A}^{\bullet} \to \mathcal{A}^{\bullet + 1} \]
satisfying 
\begin{enumerate}[(i)]
\item $\nabla(ab) = \left( \nabla a \right) b + (-1)^{\abs{a}} a \left( \nabla b \right) $,
\item $\nabla^2 (a) = [ R, a]$,
\item $\nabla(R) = 0$.
\end{enumerate}
\end{definition}

\begin{remark}
Throughout this paper we follow the standard $\pm$ sign conventions for graded algebras.  For homogeneous elements $a$ and $b$, an operator $D$ is a derivation if it satisfies the ``Leibniz rule'' $D(ab) = (Da)b + (-1)^{|a|} (Db)$, and the commutator bracket is defined $[a,b] \=  a b - (-1)^{|a|  |b|} b  a$.  ``Commutative'' is used in the graded sense; i.e. $a$ and $b$ commute if $ab = (-1)^{|a| |b|} ba$, which is equivalent to $[a,b]=0$.
\end{remark}

We start with two elementary examples of curved dgas: that of differential graded algebras and then an algebraic example. The algebraic example will be used later to illustrate that perturbing the differential need not be an equivalence of curved dgas. 
\begin{example}[DGA as a special case] Any dga $(\A, D)$ is naturally a curved dga with $R=0$.  Conversely, a curved dga $\left(  \mathcal{A}^{\bullet}, \nabla, R \right)$ is also an ordinary dga if $R = 0$.
\end{example}

\begin{example}\label{EX: cdga fdvs}
Let $V$ be a finite dimensional vector space and $\mathcal{A}^{\bullet} = T^{\bullet}(V)$  the free tensor algebra on $V$.  For any vector $v \in V = T^1(V)$, consider the operator $\nabla\= [v, -]$.  This gives the structure of a curved dga $(T^{\bullet}(V), [v, -], v \otimes v)$.  
\end{example}

The next two examples are central to this paper, whereby we associate a curved dga to a vector bundle with connection, and we take this opportunity to introduce notation.

\begin{example}[Vector bundle-valued forms]\label{EX: Omega nabla}
Let $E \to M$ be a vector bundle with connection $\nabla^E$.  The operator $\left(\nabla^E\right)^2 : \Omega^{\bullet}(M, E) \to \Omega^{\bullet+2}(M, E)$ is equal to $(R \wedge -)$, where $R \in \Omega^{2}(M, \End(E))$ is the curvature of $\nabla^E$.  We denote the endomorphism bundle by $\E \= \End(E) \to M$.  The connection $\nabla^E$ naturally induces a connection on $\E$, which we denote by $\nabla$, and these connections are related by  $\nabla = ad_{\nabla^E} = [\nabla^E , -]$.  The operator $\nabla: \Omega^\bullet(M,\E) \to \Omega^{\bullet+1}(M,\E)$ is a derivation, in that it satisfies the graded Leibniz rule, but it need not square to zero.  Instead, ${\nabla}^2 = [R, - ]$, where $R\in \Omega^2(M,\E)$ is the original curvature element of $\nabla^E$.  Finally, $\Omega(M, \E)$ has an associative algebra structure given by the wedge product on forms and the natural composition/product on $\End(E)$; the discussion surrounding \eqref{eq:wedgeproduct} contains additional details on this product.  The triple $\left( \Omega^\bullet(M,\E), {\nabla}, R \right)$ is therefore a curved dga, and we use the notation $\Omega_{\nabla}(M, \E)$ to refer to the entire curved dga structure.
\end{example}

\begin{remark}
Note that $\Omega_{\nabla^E}(M, E)$ is not a curved dga for two reasons.  First, it is not an algebra since there is (a priori) no product structure on $E$.  Second, the square of the covariant derivative is given by $R$, which does not live in $\Omega(M,E)$.  It turns out that $\Omega_{\nabla^E}(M, E)$ is in fact a curved dg \emph{module} over $\Omega_{\nabla}(M, \E)$, but this point will not feature anywhere else in the paper.
\end{remark}

\begin{example}[Forms on the path space with pull-back connection]\label{EX: Omega(PM)}
Let $PM = C^\infty([0,1], M)$ be the space of smooth paths in $M$, and let $ev_0: PM \to M$ be the evaluation map at $0$, which sends a path $\gamma \in PM$ to its basepoint $\gamma(0) \in M$.  Let $(E,\nabla^E)\to M$ with induced $(\E,\nabla)\to M$ be bundles with connection as described in Example \ref{EX: Omega nabla}.  The pullback $ev_0^*(\E, \nabla) \to PM$, which we denote by $(\E_0, \wt{\nabla})\to PM$, naturally gives rise to a curved dga $\Omega_{\wt{\nabla}}(PM, \E_0)$ by the same process  explained in Example \ref{EX: Omega nabla}.  This is due to the natural isomorphism $ev_0^*\left( \End(E) \right) \iso \End(ev_0^*E)$.  The square of $\wt{\nabla}$ is given by $\wt{\nabla}^2 = [ev_0^*R, -]$, though we will sometimes slightly abuse notation and refer to the pullback $ev_0^*R \in \Omega^2(PM, \E_0)$ by $R$ when there is little risk of confusion.
\end{example}

We find the following example instructive, not only in having an explicit example of a curved dga of the type from Example \ref{EX: Omega nabla}, but we also come back to this example in the next section to observe how the connection and curvature are ``seen'' by cohomology for curved dgas.

\begin{example}[Explicit $\E$ example]\label{EX: explicit cdga}
Consider the trivial bundle over $M = \mathbb{R}^2$, $E = \mathbb{R}^2 \times \mathbb{R}^2$.  The endomorphisms of this bundle are given, fiber-wise, by $2\times 2$ real-valued matrices.  We then have $\Omega^{\bullet}(M, \E)=\Omega^{\bullet}(\mathbb{R}^2, \Mat_{2\times2}(\mathbb{R})).$  For our differential we use $\nabla = d + [A, - ]$, where $A$ is the matrix
\[A_p \= \begin{pmatrix}0 & 1 \\ -1 & 0  \end{pmatrix} dx + \begin{pmatrix}0 & 1 \\ 1 & 0 \end{pmatrix}dy,\]
constant with respect to a choice of $p \in \mathbb{R}^2$.  The curvature of this connection is $R = 2 \begin{pmatrix}1 & 0 \\ 0 & -1 \end{pmatrix} dx \wedge dy$.
\end{example}

To close this section, we describe below the categorical setting for much of this paper.
\begin{definition}\label{Def cdga morphism}
A \emph{\underline{morphism of curved dga's}}, $ \left(  \A^{\bullet}, \nabla_{\A}, R_{\A} \right) \to \left(  \B^{\bullet}, \nabla_{\B}, R_{\B} \right)$ is a degree-preserving map of graded algebras $f: \A \to \B$, such that 
\[\left( \nabla_{\B} \circ f\right)(a) =\left( f  \circ \nabla_{\A} \right) (a), \quad R_{\B} = f(R_{\A}).\]
\end{definition}

\begin{definition}
The category $\NGA$ of \emph{\underline{curved differential graded algebras}} is that with objects as described in Definition \ref{def curved dga} and morphisms as described in Definition \ref{Def cdga morphism}. It contains the full subcategory of \emph{\underline{differential graded algebras}}, $\DGA \hookrightarrow \NGA$, which consists of objects of the form $\left(  \mathcal{A}^{\bullet}, d, 0 \right)$.
\end{definition}

\begin{remark}\label{rem:defn of nga}
In \cite[Section 2.7]{BD} there is a more general notion of morphisms between curved dgas, but for the purposes of this paper (see Examples \ref{ex:calculation1} and \ref{ex:calculation2}) we require a stricter notion.  It does turn out, however, that all general morphisms factor through our strict morphisms, followed by a map that is the identity on the graded algebra but perturbs the differential. 
\end{remark}

The assignment of a curved dga to each vector bundle with connection is functorial, and we introduce the appropriate domain category that makes this statement precise.

\begin{definition}
Define $\Vect^{\nabla}$ to be the category whose objects are bundles with connection and whose morphisms are as follows.  Given two objects, $(E_0 \to M_0, \nabla_0)$ and $(E_1 \to M_1, \nabla_1)$, a \emph{morphism of bundles with connection} will be a pair $(f_{10}, \alpha_{10})$, where $f_{10}: M_0 \to M_1$ is a smooth map on the base spaces, and $\alpha_{10}: f_{10}^*\left( E_1 \right) \to E_0$ is a connection-preserving map of bundles over $M_0$.  To be precise, $\alpha_{10}$ satisfies $\nabla_0 \circ \alpha_{10} = \alpha_{10} \circ \left( f_{10}^* \nabla_1  \right)$ as a map of graded vector spaces $\Omega^{\bullet}\left(M_0, f_{10}^*\left( E_1\right)\right) \to \Omega^{\bullet+1}\left(M_0,E_0\right)$.  Given two morphisms, 
\[ (E_0 \to M_0, \nabla_0) \xrightarrow{(f_{10}, \alpha_{10}) } (E_1 \to M_1, \nabla_1) \xrightarrow{(f_{21}, \alpha_{21} )} (E_2 \to M_2, \nabla_2),\]
define their composition by 
\[ (E_0 \to M_0, \nabla_0) \xrightarrow{(f_{21}, \alpha_{21}) \circ (f_{10}, \alpha_{10}) \= ( f_{21} \circ f_{10},\,   \alpha_{10}\circ f^*_{10} \alpha_{21} )} (E_2 \to M_2, \nabla_2).\]
The identity morphism will be the trivial pair of identity maps. 
\end{definition}

\begin{prop}\label{prop:BundleDGAMaps}
The assignment from Example \ref{EX: Omega nabla} that maps each vector bundle with connection $(E \to M, \nabla) \in \Vect^{\nabla}$ to the curved dga $\Omega^{\bullet}_{\nabla}(M, \E) \in \NGA$, and each morphism of bundles with connection $(f_{10}, \alpha_{10}) :(E_0 \to M_0, \nabla_0) \to (E_1 \to M_1, \nabla_1)$ to the morphism of curved dgas 
\begin{align*}
\Omega^{\bullet}_{\nabla_1}(M_1, \E_1) &\to \Omega^{\bullet}_{\nabla_0}(M_0, \E_0)\\
\omega &\mapsto   \left(\alpha_{10}\right)_*\left( f_{10}^* \omega \right),
\end{align*}
defines a functor $\Omega_{\nabla}: \left(\Vect^{\nabla}\right)^{op} \to \NGA$.
\end{prop}

As a brief application, we can reformulate the content of Examples \ref{EX: Omega nabla} and \ref{EX: Omega(PM)}. The morphism in $\Vect^{\nabla}$
\[ (\E \to M, \nabla) \xleftarrow{(ev_0, id)} (\E_0 \to PM,\widetilde{\nabla} )\]
induces a morphism in $\NGA$
\[  \Omega_{\nabla}(M, \E) \xrightarrow{\Omega_{\nabla}\left( ev_0, id \right) } \Omega_{\widetilde{\nabla}}(PM, \E_0), \]
which we recognize as the evaluation-pull-back of forms $ev_0^*$.

It is worth noting that, if we consider the full subcategories of vector bundles with flat connections $\Vect_{\mathrm{flat}}^{\nabla} \hookrightarrow \Vect^{\nabla}$ and differential graded algebras $\DGA \hookrightarrow \NGA$, we have the following diagram of functors, which commutes up to canonical isomorphism.
\begin{equation}\label{eq:FlatVectIntoVect}
\begin{tikzcd}
\Vect^{\nabla} \arrow[r, "{\Omega_{\nabla}}"]  & \NGA \\
\Vect_{\mathrm{flat}}^{\nabla} \arrow[u, hook] \arrow[r, "{\Omega_{\nabla}}"]   & \DGA  \arrow[u, hook]
\end{tikzcd}
\end{equation}

\section{Homotopy and Cohomology in $\NGA$}\label{section cohomology}
In pursuit of an appropriate notion of weak-equivalence in our category of curved differential graded algebras $\NGA$, we introduce a (curved) cohomology functor that lands in the category of graded algebras.  Since our differential does not square to zero, we can either adopt a variant on the concept of the \emph{kernel} and proceed as usual, or we can consider the maximal sub-dga.  We will define both and show they are equivalent.

\begin{definition}
Given a curved dga $(\mathcal{A}^{\bullet}, \nabla^{\bullet}, R)$, define its \emph{\underline{curved cohomology}} to be the graded abelian groups
\[ H^P_{cur}(\mathcal{A}) \= cur(\nabla^p) / im(\nabla^{p-1}), \]
where $ cur(\nabla^p) \= \{ a \in \mathcal{A}^p \ \vert \ \nabla{a} = [ R, \eta] \  \text{for some} \  \eta \in \A^{p-1} \}.$
\end{definition}

\begin{definition}
Given a curved dga, $(\mathcal{A}^{\bullet}, \nabla^{\bullet}, R)$, define its \emph{\underline{maximal sub-dga}}  $\widetilde{(\A^{\bullet}}, \nabla^{\bullet})$ by 
\[\widetilde{ \A^{\bullet}}\= \{ a \in \mathcal{A}^{\bullet} \ \vert \ \nabla^2(a) = 0 \}. \] 
\end{definition}

\begin{remark}It is easy to see these are both well-defined.  The image $im(\nabla^{p-1})$ is a subgroup of the curved kernel $cur(\nabla^p)$; if $\nabla (\eta) \in im(\nabla^{p-1})$ then $\nabla(\nabla(\eta)) = \nabla^2(\eta) = [R,\eta] \in cur(\nabla^p)$.  Both groups are abelian, therefore the quotient $H^p_{cur}(\A)$ is a well-defined abelian group.  That $\wt{\A^\bullet}$ is an algebra, i.e. it's closed under products, follows from the general formula $\nabla^2(ab) = (\nabla^2 a)b + a (\nabla^2b)$. The two $(\nabla a)(\nabla b)$ terms cancel for degree reasons. 
\end{remark}

\begin{prop}\label{cdga and maximal sub dga same cohomology}
 There is an isomorphism between the curved cohomology $H^p_{cur}(\A)$ and the cohomology of the maximal sub-dga $H^p(\widetilde{ \A})$.
\begin{proof}
The map $H^p(\widetilde{ \A}) \to H^p_{cur}(\mathcal{A})$  induced by the inclusion $ker(\nabla) \hookrightarrow cur(\nabla)$ is an isomorphism.  It is straightforward to check that the map $H^p_{cur}(\mathcal{A}) \to H^p(\widetilde{ \A})$ given by $[\omega] \mapsto [\omega - \nabla(\eta)]$, where $\nabla(\omega) = [R, \eta]$, is an inverse.
\end{proof}
\end{prop}

Note that the algebra structure on $\A$ does not immediately induce an algebra structure on the groups $H^\bullet_{cur}(\A)$, since $cur(\nabla)$ is not closed under products.  However, $H^\bullet(\wt{\A})$ is naturally an algebra since $\wt{\A} \in \DGA$.  The groups $H^\bullet_{cur}(\A)$ therefore inherit an algebra structure via the isomorphism $H^\bullet_{cur}(\A) \iso H^\bullet(\wt{\A})$ from Proposition \ref{cdga and maximal sub dga same cohomology}.  

\begin{prop}\label{PROP cur coh is cdga}
The curved cohomology of a curved dga is a graded algebra. 
\end{prop}

\begin{prop}\label{prop: curved coh is a functor}
The curved cohomology provides a functor from the category of curved dga's  to the category of graded algebras.
\begin{proof}
By the previous Proposition \ref{PROP cur coh is cdga}, we have an assignment of objects.  To show that this assignment respects morphisms, we consider two curved dgas and a morphism between them: 
\[ (\mathcal{A}, \nabla, R) \xrightarrow{g} (\mathcal{B},\Theta, F).\] 
By the definition of $g$ being a morphism of curved dgas, it is evident that $g$ preserves both the curved kernel and the image of the curved differentials.
\end{proof}
\end{prop}

At this point we should note that our curved cohomology is reduced to the usual cohomology when a curved dga is actually an ordinary dga.  More precisely, $\DGA \hookrightarrow \NGA$ is a full subcategory, and the cohomology functor for dgas factors through the curved cohomology functor for curved dgas.  Furthermore, this is compatible with the inclusion of flat vector bundle connections from  \eqref{eq:FlatVectIntoVect}.

\begin{prop}\label{Prop flat comm diagram}
Using the above notation, the following is a commutative diagram of functors.
\begin{equation}
\begin{tikzcd}
\Vect^{\nabla} \arrow[r, "{\Omega_{\nabla}}"]  & \NGA \arrow[r, "{H^{\bullet}_{cur}}"]  & \mathrm{Alg}^{\bullet}_{\mathbbm{k}} \\
\Vect_{\mathrm{flat}}^{\nabla} \arrow[u, hook] \arrow[r, "{\Omega_{\nabla}}"]   & \DGA \arrow[ru, "{H^{\bullet}}"']  \arrow[u, hook]& 
\end{tikzcd}
\end{equation}
\end{prop}

Perturbing the differential can lead to non-isomorphic curved cohomology groups, as demonstrated in the following examples.  This illustrates one of the reasons we use our stricter notion of morphism, as discussed in Remark \ref{rem:defn of nga}.

\begin{example}\label{ex:calculation1}
Recall the curved dga, $\mathscr{A} = (T^{\bullet}(V), [v, -], v \otimes v)$, from Example \ref{EX: cdga fdvs}.  Next consider the curved dga, $\mathscr{B} = (T^{\bullet}(V),0, 0)$, which is obtained by perturbing the differential.  While $\mathscr{A}$ and $\mathscr{B}$ are considered isomorphic in other contexts, they are not even quasi-isomorphic in our category $\NGA$.  It is clear that $H^{\bullet}_{cur}(\mathscr{B}) = T^{\bullet}(V)$, and $H^{0}_{cur}(\mathscr{A}) = T^{0}(V)$.  But the higher cohomology groups are not isomorphic if $\dim V>1$, as  $\dim H^{k}_{cur}(\mathscr{A}) <  \dim H^{k}_{cur}(\mathscr{B}) = \dim T^k(V)$ for $k >0$.  To see this, let $w\in V$ be linearly independent from $v$.  For $k=1$, note that $[v,w] \neq 0$, and hence $cur(\mathscr{A}^1)$ is a proper subspace of $T^1(V)$.  For $k\geq 2$, note that $cur(\mathscr{A}^k) \subseteq T^k(V)$, and hence $\dim H^{k}_{cur}(\mathscr{A}) \leq \dim T^k(V)$.   Then, $[v, w^{\otimes (k-1)}] \neq 0$, so $\dim im([v,-]) > 0$, and thus $\dim H^{k}_{cur}(\mathscr{A}) < \dim T^k(V)$. \end{example}

Proposition \ref{Prop flat comm diagram} says that our curved cohomology, at the very least, includes all  the usual information that ordinary cohomology gives about flat vector bundles and their induced differential graded algebras.  The next example goes one step further to demonstrate that the theory can detect curvature of a connection, even in the case of a product bundle over a contractible base.

\begin{example}\label{ex:calculation2}
Consider the curved dgas $\mathscr{A} = \left( \Omega^{\bullet}(\mathbb{R}^2, \Mat_{2\times2}(\mathbb{R})), \nabla = d_{DR}, 0\right)$ and\\
$\mathscr{B} = \left( \Omega^{\bullet}(\mathbb{R}^2, \Mat_{2\times2}(\mathbb{R})), \nabla = d_{DR} + [A, -], R\right)$, where $A$ and $R$ are defined in Example \ref{EX: explicit cdga}.  In this case, $\mathscr{A}$ is an ordinary dga and isomorphic to $\Omega(\mathbb{R}^2)\otimes \Mat_{2\times 2}(\mathbb{R})$, so $H^k(\mathscr{A})=0$ for $k>0$.  On the other hand, one can show that $H^1(\mathscr{B})\neq 0$ by verifying, with a little bit of work, that 
 \[\omega \= \begin{bmatrix} 0&0\\1&0\end{bmatrix} (dx + dy)\]
is curved-closed but not exact.
\end{example}

Chain homotopies induce quasi-isomorphisms in the usual context of differential graded algebras.  Under the exact same formula to define homotopy, this fact  generalizes to curved dgas.

\begin{definition}\label{defn:ChainHtpy}
Two $\NGA$ morphisms $f,g: (\mathcal{A}, \nabla, R) \to (\mathcal{B},\Theta, F)$ are said to be \emph{(chain) homotopic}  $f\simeq g$ if there exists a degree $-1$ map, $h: \mathcal{A}^n \to \mathcal{B}^{n-1}$, such that 
\[f - g = \Theta \circ h + h \circ \nabla. \]
Moreover, a map of curved dgas $f: \A \to \B$ is a \emph{homotopy equivalence} if there exists a map of curved dgas $g: \B \to \A$ such that $g \circ f \simeq id_{\A}$ and $f \circ g \simeq id_{\B}$. 
\end{definition}

All three of the following statements can be proved, symbol for symbol, by the exact same arguments used for cochain complexes.  The proofs are straightforward enough that most textbooks leave them as homework exercises.  None of the properties rely on $d^2=0$, though the second property does use that morphisms commute with the differential, something we required in Definition \ref{Def cdga morphism} for $\NGA$ morphisms.  

\begin{prop}\label{prop:HomotopyEquivalences} Let $\A, \B, \C \in \NGA$.
\noindent\begin{enumerate}[(1)]
 \item Homotopy equivalence $\simeq$ is an equivalence relation on $\NGA(\A,\B)$.
 \item Homotopies in $\NGA$ are closed under composition; i.e. if $f \simeq f' \colon \A \to \B$ and $g\simeq g' \colon \B \to \C$, then  $g \o f \simeq g' \o f'$.
 \item Given $\NGA$ morphisms  $f \colon \A \to \B$ and $g \colon \B\to \C$, if two of the three morphisms in $\{f, g, g \o f\}$ are homotopy equivalences, then so is the third.
\end{enumerate}
\begin{proof}
The first two statements follow from the exact same proofs for cochain complexes, which can be found in the course notes \cite[Lecture 10]{QuickNotes}.  We quickly prove the third property for the case used later. Assume that $f$ and $k \=gf$ are homotopy equivalences with homotopy inverses $\wt{f}$ and $\wt{k}$.  Then $\wt{g} \= f\wt{k}$ is a homotopy inverse for $g$, since $f\wt{k}g \simeq f\wt{k}gf \wt{f} \simeq f \wt{f}\simeq id_{\B}$ and $g f\wt{k} \simeq id_{\C}$.
\end{proof}
\end{prop}

We now justify our use of language by proving that homotopy equivalences between curved dgas induce isomorphisms in their curved cohomology.

\begin{prop}\label{prop:AlgChainHomotopy} Let $f, g$ be $\NGA$ morphisms $\A \to \B$.
\noindent\begin{enumerate}[(1)] 
 \item If $f \simeq g$, the two homorphisms $H^\bullet_{cur}(\A) \to H^\bullet_{cur}(\B)$ induced by $f$ and $g$ are equal.
 \item A homotopy equivalence $f :\A \xrightarrow{\simeq} \B$ induces an isomorphism $H^\bullet_{cur}(\A) \xrightarrow{\iso} H^\bullet_{cur}(\B)$.
\end{enumerate}
\begin{proof}
Let $f,g: (\mathcal{A}, \nabla, R) \to (\mathcal{B},\Theta, F)$ be two maps of curved dgas, and let $h: \mathcal{A}^n \to \mathcal{B}^{n-1}$ be a chain homotopy in the sense that $f - g = \Theta \circ h + h \circ \nabla$.  Then considering a class $[\alpha] \in H^p_{cur} \left( \mathcal{A}\right)$, with $\nabla \alpha = [ R, \eta] = \nabla^2 \eta$ for some $\eta \in \mathcal{A}^{p-1}$, we show the difference of the induced maps on cohomology is zero.\begin{align*}
(f-g)[ \alpha] &= [(f -g) (\alpha)] = [\left(  \Theta \circ h + h \circ \nabla \right) (\alpha) ]\\
&= [ \Theta \left(h \alpha \right) + h \left( \nabla \alpha \right) ] = [ \Theta \left(h \alpha  \right) + h \left( \nabla^2 \eta \right) ]
\intertext{applying $f - g - \Theta \circ h = h \circ \nabla$,}
&=  [ \Theta \left(h \alpha  \right) + \left(f - g - \Theta \circ h \right)  \left( \nabla^2 \eta \right)  ]\\
&=  [ \Theta \left(h \alpha  \right) + f\circ \nabla (\nabla \eta)  - g\circ \nabla (\nabla \eta) - \Theta \circ h \left( \nabla^2 \eta \right)  ]
\intertext{and now using that $f$ and $g$ are commute with the curved differentials,}
&=  [ \Theta \left(h \alpha  \right) + \Theta \circ f (\nabla \eta)  - \Theta \circ  g (\nabla \eta) - \Theta \circ h \left( \nabla^2 \eta \right)  ] \\
&= [ \Theta \left(h \alpha +  f (\nabla \eta)  -   g (\nabla \eta) -  h \left( \nabla^2 \eta \right)  \right)]=0.
\end{align*}
The second statement then follows immediately from the first.
\end{proof}
\end{prop}

We will combine the material from this section to prove in Proposition \ref{prop:PMhtpyM} that, given a vector bundle with connection $(E, \nabla) \to M$, the evaluation map $ev_0: PM \to M$ induces a homotopy equivalence of curved dgas  $ev_0^*: \Omega_{\nabla}(M, \E) \xrightarrow{\simeq} \Omega_{\wt{ \nabla}} \left(PM, \E_0\right)$. Therefore, by Proposition \ref{prop:AlgChainHomotopy} above, $ev_0^*$ induces an isomorphism on curved cohomology.

\section{Zigzag Hochschild}\label{section zigzag}

In this chapter, we define the complex $ZZ(\A)$ along with a product, making it into a curved dga.  It should be noted here that the product structure is in some sense the entire motivation for the use of zigzags.  In fact, one could employ a traditional bar complex construction and use a similar iterated integral instead of the one we define if we only wanted a map that respects the differentials.  However, in this approach, the product structure given by wedging forms on the path space would not be reflected in the algebraic model (see page \pageref{intro calculation} for some more clarification on this point).  Thus, we introduce an alteration of this bar complex that uses zigzags, along with an iterated integral map, to produce a morphism in our category of curved differential graded algebras.

Before providing the official Definition \ref{def CH^ZZ}, we will expand upon and explain the individual components that comprise the definition of $ZZ(\A)$.  The monomials will be elements $\underline{x}_{k,n}  \in  (\A \otimes(\A^{\otimes n}\otimes \A)^{\otimes k})[n]$, and such an element is represented visually by the following diagram.
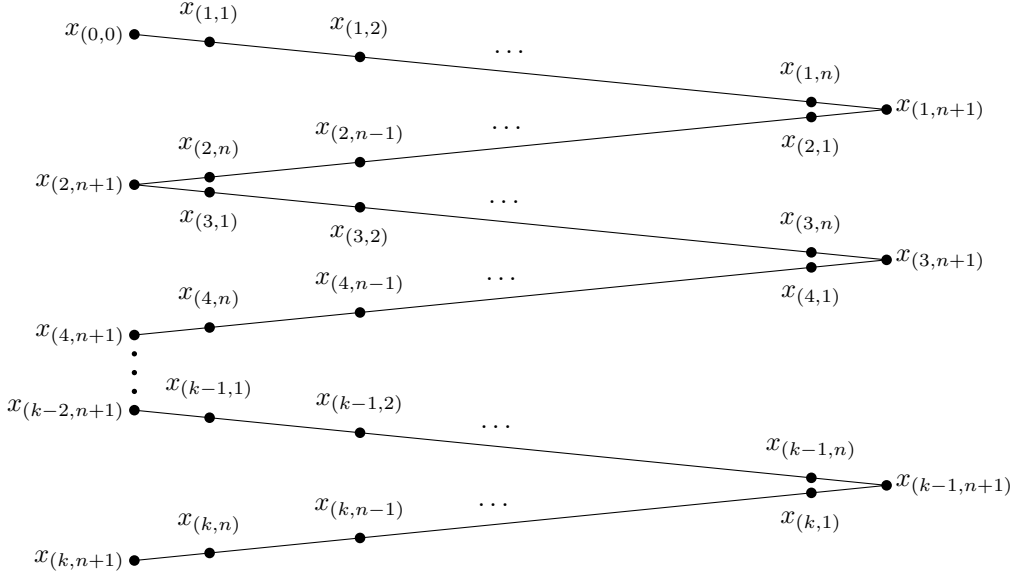
\begin{equation*}\label{general zigzag}
\begin{tikzpicture}
\draw[name path=zigzagtop] (0,10) to (10,9) to (0,8) to (10,7) to (0,6);
 \fill (0,5.75) circle (1pt);
 \fill (0,5.5) circle (1pt);
 \fill (0,5.25) circle (1pt);
\draw[name path=zigzagbot] (0,5) to (10,4) to (0,3); 
\fill (0,10) circle (2pt) node[left] {$x_{(0,0)}$};
\fill (0,8) circle (2pt) node[left] {$x_{(2,n+1)}$};
\fill (0,6) circle (2pt) node[left] {$x_{(4,n+1)}$};
\fill (0,5) circle (2pt) node[left] {$x_{(k-2,n+1)}$};
\fill (0,3) circle (2pt) node[left] {$x_{(k,n+1)}$};
\fill (10,9) circle (2pt) node[right] {$x_{(1,n+1)}$};
\fill (10,7) circle (2pt) node[right] {$x_{(3,n+1)}$};
\fill (10,4) circle (2pt) node[right] {$x_{(k-1,n+1)}$};
\draw[opacity=0,name path=t1] (1, 10) to (1,2.9);
\fill [name intersections={of=t1 and zigzagtop, by={a1,b1, c1, d1}}]
        (a1) circle (2pt)
        (b1) circle (2pt)
        (c1) circle (2pt)
        (d1) circle (2pt);
\node[label=above:{$x_{(1,1)}$}] at (a1) {};
\node[label=above:$x_{(2,n)}$] at (b1) {};
\node[label=below:$x_{(3,1)}$] at (c1) {};
\node[label=above:$x_{(4,n)}$] at (d1) {};

\fill [name intersections={of=t1 and zigzagbot, by={e1,f1}}]
        (e1) circle (2pt)
        (f1) circle (2pt);
\node[label=above:{$x_{(k-1,1)}$}] at (e1) {};
\node[label=above:$x_{(k,n)}$] at (f1) {};

\draw[opacity=0,name path=t2] (3, 10) to (3,2.9);
\fill [name intersections={of=t2 and zigzagtop, by={a2,b2, c2, d2}}]
        (a2) circle (2pt)
        (b2) circle (2pt)
        (c2) circle (2pt)
        (d2) circle (2pt);
\node[label=above:{$x_{(1,2)}$}] at (a2) {};
\node[label=above:$x_{(2,n-1)}$] at (b2) {};
\node[label=below:$x_{(3,2)}$] at (c2) {};
\node[label=above:$x_{(4,n-1)}$] at (d2) {};

\fill [name intersections={of=t2 and zigzagbot, by={e2,f2}}]
        (e2) circle (2pt)
        (f2) circle (2pt);
\node[label=above:{$x_{(k-1,2)}$}] at (e2) {};
\node[label=above:$x_{(k,n-1)}$] at (f2) {};



\draw[opacity=0, name path=t4] (9, 10) to (9,2.9);
\fill [name intersections={of=t4 and zigzagtop, by={a4,b4, c4, d4}}]
        (a4) circle (2pt)
        (b4) circle (2pt)
        (c4) circle (2pt)
        (d4) circle (2pt);
\node[label=above:{$x_{(1,n)}$}] at (a4) {};
\node[label=below:$x_{(2,1)}$] at (b4) {};
\node[label=above:$x_{(3,n)}$] at (c4) {};
\node[label=below:$x_{(4,1)}$] at (d4) {};

\fill [name intersections={of=t4 and zigzagbot, by={e4,f4}}]
        (e4) circle (2pt)
        (f4) circle (2pt);
\node[label=above:{$x_{(k-1,n)}$}] at (e4) {};
\node[label=below:$x_{(k,1)}$] at (f4) {};

\draw[opacity=0, name path=dotL] (5, 10) to (4.75,2.9);
\path [name intersections={of=dotL and zigzagtop, by={dotL1, dotL2, dotL3, dotL4}}];

\foreach \n in {1,...,4} {
\node[label=above: $\ldots$] at (dotL\n) {};}

\path [name intersections={of=dotL and zigzagbot, by={dotL5, dotL6}}];
\node[label=above: $\ldots$] at (dotL5) {};
\node[label=above: $\ldots$] at (dotL6) {};
\end{tikzpicture}
\end{equation*}
Here, the interval for the bar complex has been replaced with $\frac{1}{2} k$-many zigzags going back and forth over the same interval.  For this reason $k$ must be  even.  An odd-indexed trip over the interval (left-to-right) is referred to as a ``zig,'' and an even-indexed trip over the interval (right-to-left) is called a ``zag''.  The elements $x_{(i,j)} \in \A$ are placed at the points where the zigzags cross the $(n+2)$-many columns\footnote{These columns represent the ``time-slots''  $0=t_0\le t_1 \le \ldots \le t_n \le t_{n+1}=1$ over which we will eventually integrate when considering the iterated integral for differential forms.}.  For example, when $n=3$ and $k=4$ the monomial $\underline{x}_{k,n}$ would be represented by the following diagram.
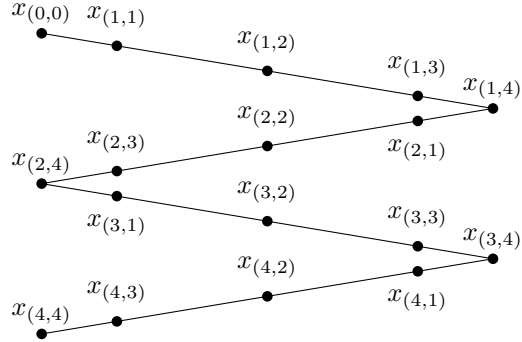
\abovedisplayskip=5pt
 \belowdisplayskip=5pt
\begin{equation*} \label{simple monomial pic}
\begin{tikzpicture}
\draw[name path=zigzagtop] (0,10) to (6,9) to (0,8) to (6,7) to (0,6); 
\fill (0,10) circle (2pt) node[above] {$x_{(0,0)}$};
\fill (0,8) circle (2pt) node[above] {$x_{(2,4)}$};
\fill (0,6) circle (2pt) node[above] {$x_{(4,4)}$};
\fill (6,9) circle (2pt) node[above] {$x_{(1,4)}$};
\fill (6,7) circle (2pt) node[above] {$x_{(3,4)}$};
\draw[opacity=0,name path=t1] (1, 10) to (1,6);
\fill [name intersections={of=t1 and zigzagtop, by={a1,b1, c1, d1}}]
        (a1) circle (2pt)
        (b1) circle (2pt)
        (c1) circle (2pt)
        (d1) circle (2pt);
\node[label=above:{$x_{(1,1)}$}] at (a1) {};
\node[label=above:$x_{(2,3)}$] at (b1) {};
\node[label=below:$x_{(3,1)}$] at (c1) {};
\node[label=above:$x_{(4,3)}$] at (d1) {};

\draw[opacity=0,name path=t2] (3, 10) to (3,6);
\fill [name intersections={of=t2 and zigzagtop, by={a2,b2, c2, d2}}]
        (a2) circle (2pt)
        (b2) circle (2pt)
        (c2) circle (2pt)
        (d2) circle (2pt);
\node[label=above:{$x_{(1,2)}$}] at (a2) {};
\node[label=above:$x_{(2,2)}$] at (b2) {};
\node[label=above:$x_{(3,2)}$] at (c2) {};
\node[label=above:$x_{(4,2)}$] at (d2) {};
\draw[opacity=0, name path=t3] (5, 10) to (5,6);
\fill [name intersections={of=t3 and zigzagtop, by={a3,b3, c3, d3}}]
        (a3) circle (2pt)
        (b3) circle (2pt)
        (c3) circle (2pt)
        (d3) circle (2pt);
\node[label=above:{$x_{(1,3)}$}] at (a3) {};
\node[label=below:$x_{(2,1)}$] at (b3) {};
\node[label=above:$x_{(3,3)}$] at (c3) {};
\node[label=below:$x_{(4,1)}$] at (d3) {};
\end{tikzpicture}
\end{equation*}

The grading on a monomial $\underline{x}_{k,n}$ is defined by taking the usual grading for tensor products  (i.e. summing the degrees of each tensor factor), and then shifting negatively by the number of interior columns $n$.  This degree shift will give a factor of $(-1)^n$ on the components of the differential, which will agree with the $\pm$ sign appearing in our eventual Stokes formula \eqref{eq:BundleStokes}.  The total space of the zigzag algebra is then defined, as a graded vector space, by a quotient of
\[ \bigoplus\limits_{\substack{n, k-1 \ge 0,\\ k \text{ even}}}   (\A \otimes(\A^{\otimes n}\otimes \A)^{\otimes k})[n].\] 
We will denote the grading of $\underline{x}_{k,n}\in ZZ(\A)$ by $|\underline{x}_{k,n}| - n$, where $|\underline{x_{k,n}}| \= |x_{(0,0)}| + |x_{(1,1)}| + \ldots + |x_{(k,n+1)}|$ is the total degree of $\underline{x}_{k,n}$ before shifting.  

The differential $D_z$ is comprised of three components: $\nabla_z$, $b_z$, and $c_z$.  The first two, $\nabla_z$ and $b_z$, are analogous to the differentials in the ordinary 2-sided bar complex, while  $c_z$  involves the curvature $R$ of the curved dga.  We proceed to describe each of these in more detail \label{describe Dz} and apply them to an arbitrary monomial
 \begin{align*}
 \underline{x}_{k,n} = x_{(0,0)}&\otimes (x_{(1,1)} \otimes \ldots \otimes x_{(1,n)} \otimes x_{(1,n+1)}) \otimes\\
  \ldots &\otimes (x_{(i,1)} \otimes \ldots \otimes x_{(i,n)} \otimes x_{(i,n+1)})\otimes \\
  \ldots &\otimes (x_{(k,1)} \otimes \ldots \otimes x_{(k,n)} \otimes x_{(k,n+1)}).
 \end{align*}
\begin{enumerate}
\item The differential $\nabla_z$ is defined by extending $\nabla$ from an operator on $\A$ to a derivation on the tensor algebra of $\A$.  In other words, $\nabla_z(\underline{\omega}_{k,n})$ applies the differential coming from $\A$ to each tensor factor of $\omega_{(i,p)}$ by the Leibniz rule. This is visualized by \label{dandbpics}
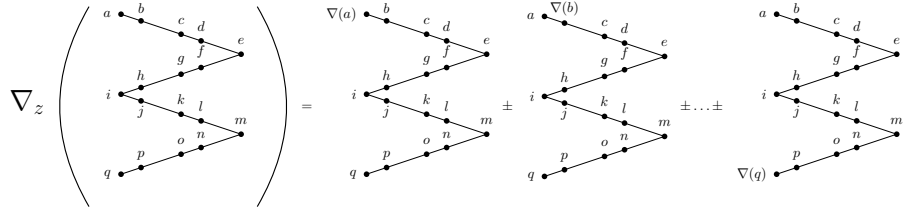
\begin{equation*}\resizebox{12cm}{!}{
 {\Huge $\nabla_z$} {\begin{tikzpicture}[ baseline={([yshift=-.5ex]current bounding box.center)},vertex/.style={anchor=base,circle,fill=black!25,minimum size=18pt,inner sep=2pt}]
  \draw [thick] (0,0) to [round left paren ] (0,5);
  \end{tikzpicture}}
{ \begin{tikzpicture}[ baseline={([yshift=-.5ex]current bounding box.center)},vertex/.style={anchor=base,circle,fill=black!25,minimum size=18pt,inner sep=2pt}]
  \draw[name path=zigzag1] (0,10) to (3,9) to (0,8) to (3,7) to (0,6);
   
   \draw[opacity=0,name path=t1] (0, 10) to (0,5);
   \fill [name intersections={of=t1 and zigzag1, by={a, i,q}}]
        (a) circle (2pt)
        (i) circle (2pt)
        (q) circle (2pt);
\node[label=left:{$a$}] at (a) {};
\node[label=left:{$i$}] at (i) {};
\node[label=left:{$q$}] at (q) {};

\draw[opacity=0,name path=t2] (0.5, 10) to (0.5,6);
   \fill [name intersections={of=t2 and zigzag1, by={b,h, j, p}}]
        \foreach \x in {b, h, p}
        {(\x) circle (2pt)}
        (j) circle (2pt); 
    \foreach \x in {b, h, p}
        {\node[above=2pt] at (\x) {$\x$};}
        \node[below] at (j) {$j$};

\draw[opacity=0,name path=t3] (1.5, 10) to (1.5,6);
   \fill [name intersections={of=t3 and zigzag1, by={c,g, k, o}}]
        \foreach \x in {c,g, k, o}
        {(\x) circle (2pt)};
        \foreach \x in {c,g, k, o}
        {\node[label=above:{$\x$}] at (\x) {};}
        
\draw[opacity=0,name path=t4] (2, 10) to (2,6);
   \fill [name intersections={of=t4 and zigzag1, by={d,f, l, n}}]
        \foreach \x in {d,f, l, n}
        {(\x) circle (2pt)};
        \foreach \x in {d,f, l, n}
        {\node[label=above:{$\x$}] at (\x) {};}
        
\draw[opacity=0,name path=t5] (3, 10) to (3,6);
   \fill [name intersections={of=t5 and zigzag1, by={e, m}}]
        \foreach \x in {e,m}
        {(\x) circle (2pt)};
        \foreach \x in {e,m}
        {\node[label=above:{$\x$}] at (\x) {};}
\end{tikzpicture}}
{\begin{tikzpicture}[ baseline={([yshift=-.5ex]current bounding box.center)},vertex/.style={anchor=base,circle,fill=black!25,minimum size=18pt,inner sep=2pt}]
  \draw [thick] (0,0) to [round right paren ] (0,5);
  \end{tikzpicture}} = 
   { \begin{tikzpicture}[ baseline={([yshift=-.5ex]current bounding box.center)},vertex/.style={anchor=base,circle,fill=black!25,minimum size=18pt,inner sep=2pt}]
  \draw[name path=zigzag1] (0,10) to (3,9) to (0,8) to (3,7) to (0,6);
   
   \draw[opacity=0,name path=t1] (0, 10) to (0,5);
   \fill [name intersections={of=t1 and zigzag1, by={a, i,q}}]
        (a) circle (2pt)
        (i) circle (2pt)
        (q) circle (2pt);
\node[label=left:{$\nabla(a)$}] at (a) {};
\node[label=left:{$i$}] at (i) {};
\node[label=left:{$q$}] at (q) {};

\draw[opacity=0,name path=t2] (0.5, 10) to (0.5,6);
   \fill [name intersections={of=t2 and zigzag1, by={b,h, j, p}}]
        \foreach \x in {b, h, p}
        {(\x) circle (2pt)}
        (j) circle (2pt); 
    \foreach \x in {b, h, p}
        {\node[above=2pt] at (\x) {$\x$};}
        \node[below] at (j) {$j$};

\draw[opacity=0,name path=t3] (1.5, 10) to (1.5,6);
   \fill [name intersections={of=t3 and zigzag1, by={c,g, k, o}}]
        \foreach \x in {c,g, k, o}
        {(\x) circle (2pt)};
        \foreach \x in {c,g, k, o}
        {\node[label=above:{$\x$}] at (\x) {};}
        
\draw[opacity=0,name path=t4] (2, 10) to (2,6);
   \fill [name intersections={of=t4 and zigzag1, by={d,f, l, n}}]
        \foreach \x in {d,f, l, n}
        {(\x) circle (2pt)};
        \foreach \x in {d,f, l, n}
        {\node[label=above:{$\x$}] at (\x) {};}
        
\draw[opacity=0,name path=t5] (3, 10) to (3,6);
   \fill [name intersections={of=t5 and zigzag1, by={e, m}}]
        \foreach \x in {e,m}
        {(\x) circle (2pt)};
        \foreach \x in {e,m}
        {\node[label=above:{$\x$}] at (\x) {};}
\end{tikzpicture}}
$\pm$  { \begin{tikzpicture}[ baseline={([yshift=-.5ex]current bounding box.center)},vertex/.style={anchor=base,circle,fill=black!25,minimum size=18pt,inner sep=2pt}]
  \draw[name path=zigzag1] (0,10) to (3,9) to (0,8) to (3,7) to (0,6);
   
   \draw[opacity=0,name path=t1] (0, 10) to (0,5);
   \fill [name intersections={of=t1 and zigzag1, by={a, i,q}}]
        (a) circle (2pt)
        (i) circle (2pt)
        (q) circle (2pt);
\node[label=left:{$a$}] at (a) {};
\node[label=left:{$i$}] at (i) {};
\node[label=left:{$q$}] at (q) {};

\draw[opacity=0,name path=t2] (0.5, 10) to (0.5,6);
   \fill [name intersections={of=t2 and zigzag1, by={b,h, j, p}}]
        \foreach \x in {b, h, p}
        {(\x) circle (2pt)}
        (j) circle (2pt); 
    \foreach \x in {h, p}
        {\node[above=2pt] at (\x) {$\x$};}
        \node[above=2pt] at (b) {$\nabla(b)$};
        \node[below] at (j) {$j$};

\draw[opacity=0,name path=t3] (1.5, 10) to (1.5,6);
   \fill [name intersections={of=t3 and zigzag1, by={c,g, k, o}}]
        \foreach \x in {c,g, k, o}
        {(\x) circle (2pt)};
        \foreach \x in {c,g, k, o}
        {\node[label=above:{$\x$}] at (\x) {};}
        
\draw[opacity=0,name path=t4] (2, 10) to (2,6);
   \fill [name intersections={of=t4 and zigzag1, by={d,f, l, n}}]
        \foreach \x in {d,f, l, n}
        {(\x) circle (2pt)};
        \foreach \x in {d,f, l, n}
        {\node[label=above:{$\x$}] at (\x) {};}
        
\draw[opacity=0,name path=t5] (3, 10) to (3,6);
   \fill [name intersections={of=t5 and zigzag1, by={e, m}}]
        \foreach \x in {e,m}
        {(\x) circle (2pt)};
        \foreach \x in {e,m}
        {\node[label=above:{$\x$}] at (\x) {};}
\end{tikzpicture}} $\pm \ldots \pm$ 
 { \begin{tikzpicture}[ baseline={([yshift=-.5ex]current bounding box.center)},vertex/.style={anchor=base,circle,fill=black!25,minimum size=18pt,inner sep=2pt}]
  \draw[name path=zigzag1] (0,10) to (3,9) to (0,8) to (3,7) to (0,6);
   
   \draw[opacity=0,name path=t1] (0, 10) to (0,5);
   \fill [name intersections={of=t1 and zigzag1, by={a, i,q}}]
        (a) circle (2pt)
        (i) circle (2pt)
        (q) circle (2pt);
\node[label=left:{$a$}] at (a) {};
\node[label=left:{$i$}] at (i) {};
\node[label=left:{$\nabla(q)$}] at (q) {};

\draw[opacity=0,name path=t2] (0.5, 10) to (0.5,6);
   \fill [name intersections={of=t2 and zigzag1, by={b,h, j, p}}]
        \foreach \x in {b, h, p}
        {(\x) circle (2pt)}
        (j) circle (2pt); 
    \foreach \x in {b, h, p}
        {\node[above=2pt] at (\x) {$\x$};}
        \node[below] at (j) {$j$};

\draw[opacity=0,name path=t3] (1.5, 10) to (1.5,6);
   \fill [name intersections={of=t3 and zigzag1, by={c,g, k, o}}]
        \foreach \x in {c,g, k, o}
        {(\x) circle (2pt)};
        \foreach \x in {c,g, k, o}
        {\node[label=above:{$\x$}] at (\x) {};}
        
\draw[opacity=0,name path=t4] (2, 10) to (2,6);
   \fill [name intersections={of=t4 and zigzag1, by={d,f, l, n}}]
        \foreach \x in {d,f, l, n}
        {(\x) circle (2pt)};
        \foreach \x in {d,f, l, n}
        {\node[label=above:{$\x$}] at (\x) {};}
        
\draw[opacity=0,name path=t5] (3, 10) to (3,6);
   \fill [name intersections={of=t5 and zigzag1, by={e, m}}]
        \foreach \x in {e,m}
        {(\x) circle (2pt)};
        \foreach \x in {e,m}
        {\node[label=above:{$\x$}] at (\x) {};}
\end{tikzpicture}}}
  \end{equation*}
and defined explicitly as
\begin{align}
\nabla_z( \underline{x}_{k,n})\= &(-1)^n \nabla(x_{(0,0)}) \otimes (x_{(1,1)} \otimes \ldots \otimes x_{(1,n+1)})\otimes \ldots \otimes (x_{(k,1)}  \otimes \ldots \otimes x_{(k,n+1)}) \label{EQ: nabla_z}\\
+ & \sum\limits_{p=1}^{n+1} \sum\limits_{i=1}^k (-1)^{n+ \beta_{(i,p)}}  x_{(0,0)} \otimes \ldots \otimes (x_{(i,1)} \otimes \ldots \otimes \nabla(x_{i,p}) \otimes \ldots \otimes x_{(i,n+1)})\nonumber \\
&\otimes \ldots \otimes (x_{(k,1)}  \otimes \ldots \otimes x_{(k,n+1)}),\nonumber
\end{align}
where $\beta_{(i,p)}$ is equal to the sum of the degrees of the elements in $\A$ appearing before the element that $\nabla$ is being applied to.  In particular $\beta_{(i,p)}\= |x_{(0,0)}| + \ldots + |x_{(i,1)}| + \ldots +  |x_{(i,p-1)}|$ for $p>1$ and $\beta_{(i,1)}\= |x_{(0,0)}| + \ldots + |x_{(i-1,1)}| + \ldots |x_{(i-1,n+1)}|$.

\item The $b_z$ is the ``Hochschild-esque'' component of the differential and is motivated by the geometric boundary of an $n$-simplex.  Applying $b_z$ to a monomial is represented visually by identifying two neighboring columns and multiplying their entries.
   \begin{equation*}\resizebox{12cm}{!}{
 $b_z$
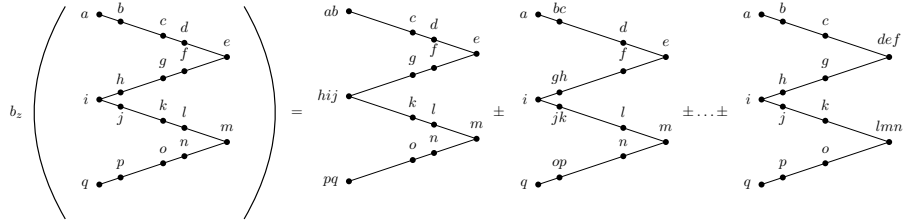
{\begin{tikzpicture}[ baseline={([yshift=-.5ex]current bounding box.center)},vertex/.style={anchor=base,circle,fill=black!25,minimum size=18pt,inner sep=2pt}]
  \draw [thick] (0,0) to [round left paren ] (0,5);
  \end{tikzpicture}}
{ \begin{tikzpicture}[ baseline={([yshift=-.5ex]current bounding box.center)},vertex/.style={anchor=base,circle,fill=black!25,minimum size=18pt,inner sep=2pt}]
  \draw[name path=zigzag1] (0,10) to (3,9) to (0,8) to (3,7) to (0,6);
   
   \draw[opacity=0,name path=t1] (0, 10) to (0,5);
   \fill [name intersections={of=t1 and zigzag1, by={a, i,q}}]
        (a) circle (2pt)
        (i) circle (2pt)
        (q) circle (2pt);
\node[label=left:{$a$}] at (a) {};
\node[label=left:{$i$}] at (i) {};
\node[label=left:{$q$}] at (q) {};

\draw[opacity=0,name path=t2] (0.5, 10) to (0.5,6);
   \fill [name intersections={of=t2 and zigzag1, by={b,h, j, p}}]
        \foreach \x in {b, h, p}
        {(\x) circle (2pt)}
        (j) circle (2pt); 
    \foreach \x in {b, h, p}
        {\node[above=2pt] at (\x) {$\x$};}
        \node[below] at (j) {$j$};

\draw[opacity=0,name path=t3] (1.5, 10) to (1.5,6);
   \fill [name intersections={of=t3 and zigzag1, by={c,g, k, o}}]
        \foreach \x in {c,g, k, o}
        {(\x) circle (2pt)};
        \foreach \x in {c,g, k, o}
        {\node[label=above:{$\x$}] at (\x) {};}
        
\draw[opacity=0,name path=t4] (2, 10) to (2,6);
   \fill [name intersections={of=t4 and zigzag1, by={d,f, l, n}}]
        \foreach \x in {d,f, l, n}
        {(\x) circle (2pt)};
        \foreach \x in {d,f, l, n}
        {\node[label=above:{$\x$}] at (\x) {};}
        
\draw[opacity=0,name path=t5] (3, 10) to (3,6);
   \fill [name intersections={of=t5 and zigzag1, by={e,  m}}]
        \foreach \x in {e,m}
        {(\x) circle (2pt)};
        \foreach \x in {e,m}
        {\node[label=above:{$\x$}] at (\x) {};}
\end{tikzpicture}}
{\begin{tikzpicture}[ baseline={([yshift=-.5ex]current bounding box.center)},vertex/.style={anchor=base,circle,fill=black!25,minimum size=18pt,inner sep=2pt}]
  \draw [thick] (0,0) to [round right paren ] (0,5);
  \end{tikzpicture}} = 
   { \begin{tikzpicture}[ baseline={([yshift=-.5ex]current bounding box.center)},vertex/.style={anchor=base,circle,fill=black!25,minimum size=18pt,inner sep=2pt}]
  \draw[name path=zigzag1] (0,10) to (3,9) to (0,8) to (3,7) to (0,6);
   
   \draw[opacity=0,name path=t1] (0, 10) to (0,5);
   \fill [name intersections={of=t1 and zigzag1, by={a,  i,q}}]
        (a) circle (2pt)
        (i) circle (2pt)
        (q) circle (2pt);
\node[label=left:{$ab$}] at (a) {};
\node[label=left:{$hij$}] at (i) {};
\node[label=left:{$pq$}] at (q) {};


\draw[opacity=0,name path=t3] (1.5, 10) to (1.5,6);
   \fill [name intersections={of=t3 and zigzag1, by={c,g, k, o}}]
        \foreach \x in {c,g, k, o}
        {(\x) circle (2pt)};
        \foreach \x in {c,g, k, o}
        {\node[label=above:{$\x$}] at (\x) {};}
        
\draw[opacity=0,name path=t4] (2, 10) to (2,6);
   \fill [name intersections={of=t4 and zigzag1, by={d,f, l, n}}]
        \foreach \x in {d,f, l, n}
        {(\x) circle (2pt)};
        \foreach \x in {d,f, l, n}
        {\node[label=above:{$\x$}] at (\x) {};}
        
\draw[opacity=0,name path=t5] (3, 10) to (3,6);
   \fill [name intersections={of=t5 and zigzag1, by={e,  m}}]
        \foreach \x in {e,m}
        {(\x) circle (2pt)};
        \foreach \x in {e,m}
        {\node[label=above:{$\x$}] at (\x) {};}
\end{tikzpicture}}
 $\pm$  { \begin{tikzpicture}[ baseline={([yshift=-.5ex]current bounding box.center)},vertex/.style={anchor=base,circle,fill=black!25,minimum size=18pt,inner sep=2pt}]
  \draw[name path=zigzag1] (0,10) to (3,9) to (0,8) to (3,7) to (0,6);
   
   \draw[opacity=0,name path=t1] (0, 10) to (0,5);
   \fill [name intersections={of=t1 and zigzag1, by={a,  i,q}}]
        (a) circle (2pt)
        (i) circle (2pt)
        (q) circle (2pt);
\node[label=left:{$a$}] at (a) {};
\node[label=left:{$i$}] at (i) {};
\node[label=left:{$q$}] at (q) {};

\draw[opacity=0,name path=t2] (0.5, 10) to (0.5,6);
   \fill [name intersections={of=t2 and zigzag1, by={b,h, j, p}}]
        \foreach \x in {b, h, p}
        {(\x) circle (2pt)}
        (j) circle (2pt); 
      \node[above=2pt] at (b) {$bc$};
      \node[above=2pt] at (h) {$gh$};
      \node[above=2pt] at (p) {$op$};
        \node[below] at (j) {$jk$};

        
\draw[opacity=0,name path=t4] (2, 10) to (2,6);
   \fill [name intersections={of=t4 and zigzag1, by={d,f, l, n}}]
        \foreach \x in {d,f, l, n}
        {(\x) circle (2pt)};
        \foreach \x in {d,f, l, n}
        {\node[label=above:{$\x$}] at (\x) {};}
        
\draw[opacity=0,name path=t5] (3, 10) to (3,6);
   \fill [name intersections={of=t5 and zigzag1, by={e,  m}}]
        \foreach \x in {e,m}
        {(\x) circle (2pt)};
        \foreach \x in {e,m}
        {\node[label=above:{$\x$}] at (\x) {};}
\end{tikzpicture}} $\pm \ldots \pm$ 
 { \begin{tikzpicture}[ baseline={([yshift=-.5ex]current bounding box.center)},vertex/.style={anchor=base,circle,fill=black!25,minimum size=18pt,inner sep=2pt}]
  \draw[name path=zigzag1] (0,10) to (3,9) to (0,8) to (3,7) to (0,6);
   
   \draw[opacity=0,name path=t1] (0, 10) to (0,5);
   \fill [name intersections={of=t1 and zigzag1, by={a,  i,q}}]
        (a) circle (2pt)
        (i) circle (2pt)
        (q) circle (2pt);
\node[label=left:{$a$}] at (a) {};
\node[label=left:{$i$}] at (i) {};
\node[label=left:{$q$}] at (q) {};

\draw[opacity=0,name path=t2] (0.5, 10) to (0.5,6);
   \fill [name intersections={of=t2 and zigzag1, by={b,h, j, p}}]
        \foreach \x in {b, h, p}
        {(\x) circle (2pt)}
        (j) circle (2pt); 
    \foreach \x in {b, h, p}
        {\node[above=2pt] at (\x) {$\x$};}
        \node[below] at (j) {$j$};

\draw[opacity=0,name path=t3] (1.5, 10) to (1.5,6);
   \fill [name intersections={of=t3 and zigzag1, by={c,g, k, o}}]
        \foreach \x in {c,g, k, o}
        {(\x) circle (2pt)};
        \foreach \x in {c,g, k, o}
        {\node[label=above:{$\x$}] at (\x) {};}
        
        
\draw[opacity=0,name path=t5] (3, 10) to (3,6);
   \fill [name intersections={of=t5 and zigzag1, by={e,  m}}]
        \foreach \x in {e,m}
        {(\x) circle (2pt)};      
        \node[label=above:{$def$}] at (e) {};
        \node[label=above:{$lmn$}] at (m) {};
\end{tikzpicture}}}
  \end{equation*}
The differential $b_z = \sum\limits_{0\le \ell  \le n} b_{\ell}$, and we define each of these smaller components $b_{\ell}$ by the following three cases.  For $0 < \ell < n$, in the middle of the zigzags, we have
\begin{align}
b_{\ell}( \underline{x}_{k,n})\=  (-1)^{n+ \ell} &x_{(0,0)} \otimes \left( x_{(1,1)} \otimes \ldots \otimes  x_{(1, \ell)} \cdot x_{(1, \ell + 1)} \otimes \ldots \otimes  x_{(1,n+1)} \right) \label{EQ b_z}\\
  & \otimes \ldots \otimes \left( x_{(2j+1,1)}\otimes \ldots \otimes  x_{(2j+1, \ell)} \cdot x_{(2j+1, \ell + 1)} \otimes \ldots \otimes  x_{(2j+1,n+1)} \right) \nonumber \\
 & \otimes \ldots \otimes \left( x_{(2i,1)} \otimes \ldots \otimes x_{(2i,n- \ell)} \cdot x_{(2i, n- \ell+1)} \otimes \ldots  \otimes x_{(2i, n+1)} \right) \nonumber \\ 
 & \otimes  \ldots \otimes  \left(x_{(k,1)} \otimes \ldots \otimes x_{(k,n-\ell)} \cdot x_{(k, n-\ell+1)} \otimes \ldots  \otimes  x_{(k,n+1)} \right).\nonumber
\end{align}
For $\ell = 0$, at the left end-point, we have
\begin{align*}
b_0( \underline{x}_{k,n})\=  (-1)^n &(x_{(0,0)} \cdot x_{(1,1)}) \otimes \left( x_{(1,2)} \otimes \ldots \otimes x_{(1,n)} \otimes x_{(1,n+1)} \right)\\
 & \otimes \ldots \otimes \left( x_{(2i,1)} \otimes \ldots \otimes x_{(2i,n-1)} \otimes (x_{(2i,n)} \cdot x_{(2i, n+1)} \cdot x_{(2i+1, 1)}) \right)\\ 
  & \otimes \ldots \otimes \left( x_{(2j+1,2)} \otimes \ldots \otimes x_{(2j+1,n+1)} \right)\\ 
 & \otimes  \ldots \otimes  \left(x_{(k,1)} \otimes \ldots \otimes x_{(k,n-1)} \otimes x_{(k, n)} \cdot x_{(k,n+1)} \right).
\end{align*}
For the $\ell = n$, at the right end-point, we have
\begin{align*}
b_n( \underline{x}_{k,n})\= &x_{(0,0)} \otimes \left( x_{(1,1)} \otimes \ldots \otimes x_{(1,n-1)} \otimes x_{(1,n)} \cdot x_{(1,n+1)} \cdot_{(2,1)} \right)\\
  & \otimes \ldots \otimes \left( x_{(2j-1,1)} \otimes \ldots \otimes  x_{(2j-1,n)} \cdot x_{(2j-1,n+1)} \cdot_{(2j,1)} \right)\\ 
 & \otimes \ldots \otimes \left( x_{(2i,2)} \otimes \ldots \otimes x_{(2i,n+1)} \right)\\ 
 & \otimes  \ldots \otimes  \left(x_{(k,2)} \otimes \ldots \otimes x_{(k,n+1)} \right).
\end{align*}
Note that $b_z$ is a degree +1 operator, because while the sum of the degrees of tensor factors in $\underline{x}_{k,n}$ and $b_{\ell}(\underline{x}_{k,n})$ are equal,  $\underline{x}_{k,n}$ is shifted by $-n$ whereas $b_{\ell}(\underline{x}_{k,n})$ is shifted by $-(n-1)$. 

\item The $c_z$ inserts the curvature $R$ in between the different tensor factors of $\underline{x}_{k,n}$.  This is accomplished with components $c_{j, \ell}$ which insert a new column in $\underline{x}_{k,n}$ between columns $(\ell-1)$ and $\ell$, consisting of 1's except for exactly one $R$ at the $j$-th zig/zag (or ``row''):
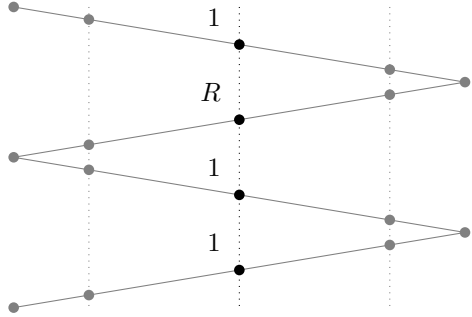
\begin{equation*}
\begin{tikzpicture}
\draw[name path=zigzagtop, gray] (0,10) to (6,9) to (0,8) to (6,7) to (0,6); 
\fill[gray] (0,10) circle (2pt) ;
\fill[gray] (0,8) circle (2pt) ;
\fill[gray] (0,6) circle (2pt) ;
\fill[gray] (6,9) circle (2pt) ;
\fill[gray] (6,7) circle (2pt) ;
\draw[gray, dotted,name path=t1] (1, 10) to (1,6);
\fill [gray, name intersections={of=t1 and zigzagtop, by={a1,b1, c1, d1}}]
        (a1) circle (2pt)
        (b1) circle (2pt)
        (c1) circle (2pt)
        (d1) circle (2pt);

\draw[dotted, name path=t2] (3, 10) to (3,6);
\fill [name intersections={of=t2 and zigzagtop, by={a2,b2, c2, d2}}]
        (a2) circle (2pt)
        (b2) circle (2pt)
        (c2) circle (2pt)
        (d2) circle (2pt);
\node[label=above left:{$1$}] at (a2) {};
\node[label=above left:$R$] at (b2) {};
\node[label=above left:$1$] at (c2) {};
\node[label=above left:$1$] at (d2) {};



\draw[dotted, gray, name path=t3] (5, 10) to (5,6);
\fill [gray, name intersections={of=t3 and zigzagtop, by={a3,b3, c3, d3}}]
        (a3) circle (2pt)
        (b3) circle (2pt)
        (c3) circle (2pt)
        (d3) circle (2pt);

\end{tikzpicture}
\end{equation*} In particular, we have
\begin{align}
c_{z}(\underline{x}_{k,n}) &= \sum\limits_{j=1}^k \sum\limits_{\ell=1}^{n+1} c_{j,\ell} \left(x_{(0,0)} \otimes \left(\bigotimes\limits_{r=1}^k \bigotimes\limits_{p=1}^{n+1} x_{(r,p)} \right)  \right)  \label{EQ c_z} \\
&=  \sum\limits_{j, \ell} (-1)^{n + j + \epsilon(\underline{x},j,\ell)} x_{(0,0)} \otimes  \left(\bigotimes\limits_{r=1}^k  \left( \bigotimes\limits_{p=1}^{p'_{j,\ell}-1} x_{(r,p)} \otimes (1 - (1-R) \cdot \delta_{r,j}) \otimes \bigotimes\limits_{p=p'_{j,\ell}}^{n+1} x_{(r,p)}  \right)  \right).  \nonumber
\end{align}
Here, $(1 - (1-R) \cdot \delta_{r,j})$ uses the Kronecker delta function to insert a curvature term if we are on zigzag row $j$ and a unit otherwise, $\epsilon(\underline{x},j,\ell)$  equals the sum of the degrees of the monomials in $\underline{x}_{k,n}$ which the curvature term had to commute past, and $p'_{j, \ell}:= \frac{1}{2}(n + (-1)^j \cdot n) + (-1)^{j+1}\cdot \ell$ is the quantity that equals $\ell$ when $j$ is odd and equals $(n- \ell)$ when $j$ is even. We may sometimes write $c_{\bullet, \ell}$ or $c_{j, \bullet}$ when we are fixing one index and summing over the others.
\end{enumerate}

Lastly, before we can define our zigzag algebra, we will be modding out our zigzags by two relations:

\begin{enumerate}[(a)]\label{page zigzag relations}
\item Firstly, by a sub-space $Ins_{zz}$, which is the sum over all operations $ins_j$, each of which is inserting a single zig-and-zag of units at any (even) height $j$. For example, we would have
   \begin{equation*}\resizebox{8cm}{!}
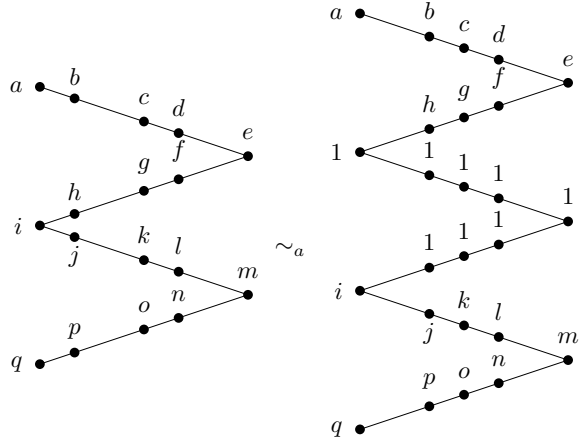
{
{ \begin{tikzpicture}[ baseline={([yshift=-.5ex]current bounding box.center)},vertex/.style={anchor=base,circle,fill=black!25,minimum size=18pt,inner sep=2pt}]
  \draw[name path=zigzag1] (0,10) to (3,9) to (0,8) to (3,7) to (0,6);
   
   \draw[opacity=0,name path=t1] (0, 10) to (0,5);
   \fill [name intersections={of=t1 and zigzag1, by={a,  i,q}}]
        (a) circle (2pt)
        (i) circle (2pt)
        (q) circle (2pt);
\node[label=left:{$a$}] at (a) {};
\node[label=left:{$i$}] at (i) {};
\node[label=left:{$q$}] at (q) {};

\draw[opacity=0,name path=t2] (0.5, 10) to (0.5,6);
   \fill [name intersections={of=t2 and zigzag1, by={b,h, j, p}}]
        \foreach \x in {b, h, p}
        {(\x) circle (2pt)}
        (j) circle (2pt); 
    \foreach \x in {b, h, p}
        {\node[above=2pt] at (\x) {$\x$};}
        \node[below] at (j) {$j$};

\draw[opacity=0,name path=t3] (1.5, 10) to (1.5,6);
   \fill [name intersections={of=t3 and zigzag1, by={c,g, k, o}}]
        \foreach \x in {c,g, k, o}
        {(\x) circle (2pt)};
        \foreach \x in {c,g, k, o}
        {\node[label=above:{$\x$}] at (\x) {};}
        
\draw[opacity=0,name path=t4] (2, 10) to (2,6);
   \fill [name intersections={of=t4 and zigzag1, by={d,f, l, n}}]
        \foreach \x in {d,f, l, n}
        {(\x) circle (2pt)};
        \foreach \x in {d,f, l, n}
        {\node[label=above:{$\x$}] at (\x) {};}
        
\draw[opacity=0,name path=t5] (3, 10) to (3,6);
   \fill [name intersections={of=t5 and zigzag1, by={e,  m}}]
        \foreach \x in {e,m}
        {(\x) circle (2pt)};
        \foreach \x in {e,m}
        {\node[label=above:{$\x$}] at (\x) {};}
\end{tikzpicture}}
$\sim_a$ 
   { \begin{tikzpicture}[ baseline={([yshift=-.5ex]current bounding box.center)},vertex/.style={anchor=base,circle,fill=black!25,minimum size=18pt,inner sep=2pt}]
  \draw[name path=zigzag1] (0,10) to (3,9) to (0,8) to (3,7) to (0,6) to (3,5) to (0,4);

   \draw[opacity=0,name path=t1] (0, 10) to (0,3);
   \fill [name intersections={of=t1 and zigzag1, by={a,  insL,  i,q}}]
        (a) circle (2pt)
        (insL) circle (2pt)
        (i) circle (2pt)
        (q) circle (2pt);
\node[label=left:{$a$}] at (a) {};
\node[label=left:{$1$}] at (insL) {};
\node[label=left:{$i$}] at (i) {};
\node[label=left:{$q$}] at (q) {};

\draw[opacity=0,name path=t2] (1, 10) to (1,3);
   \fill [name intersections={of=t2 and zigzag1, by={b, h, ins1t, ins1b, j, p}}]
        \foreach \x in {b, h, ins1t, ins1b, j, p}
        {(\x) circle (2pt)};
        \foreach \x in {b, h, p}
        {\node[label=above:{$\x$}] at (\x) {};}
        \node[below] at (j) {$j$};
        \foreach \x in { ins1t, ins1b}
        {\node[label=above:{$1$}] at (\x) {};}

\draw[opacity=0,name path=t3] (1.5, 10) to (1.5,3);
   \fill [name intersections={of=t3 and zigzag1, by={c,g, ins2t, ins2b, k, o}}]
        \foreach \x in {c,g, ins2t, ins2b, k, o}
        {(\x) circle (2pt)};
        \foreach \x in {c,g, k, o}
        {\node[label=above:{$\x$}] at (\x) {};}
         \foreach \x in { ins2t, ins2b}
        {\node[label=above:{$1$}] at (\x) {};}    
        
\draw[opacity=0,name path=t4] (2, 10) to (2,3);
   \fill [name intersections={of=t4 and zigzag1, by={d,f,ins3t, ins3b, l, n}}]
        \foreach \x in {d,f,ins3t, ins3b, l, n}
        {(\x) circle (2pt)};
        \foreach \x in {d,f, l, n}
        {\node[label=above:{$\x$}] at (\x) {};}
\foreach \x in { ins3t, ins3b}
        {\node[label=above:{$1$}] at (\x) {};} 
        
\draw[opacity=0,name path=t5] (3, 10) to (3,3);
   \fill [name intersections={of=t5 and zigzag1, by={e,  insR,   m}}]
        \foreach \x in {e,insR,m}
        {(\x) circle (2pt)};
        \foreach \x in {e,m}
        {\node[label=above:{$\x$}] at (\x) {};}
      \node[label=above:{$1$}] at (insR) {};

\end{tikzpicture}}}
  \end{equation*}
\item Secondly, elements can commute across units, moving to an adjacent zig or zag, but remaining in the same column.

   \begin{equation*}\resizebox{8cm}{!}
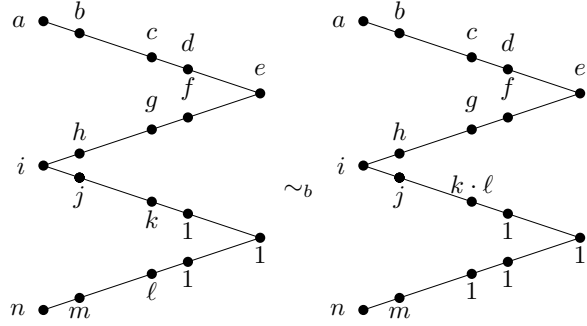
{
 \begin{tikzpicture}[ baseline={([yshift=-.5ex]current bounding box.center)},vertex/.style={anchor=base,circle,fill=black!25,minimum size=18pt,inner sep=2pt}]
  \draw[name path=zigzag1] (0,10) to (3,9) to (0,8) to (3,7) to (0,6);
   
   \draw[opacity=0,name path=t1] (0, 10) to (0,5);
   \fill [name intersections={of=t1 and zigzag1, by={a,  i,n}}]
        (a) circle (2pt)
        (i) circle (2pt)
        (n) circle (2pt);
\node[label=left:{$a$}] at (a) {};
\node[label=left:{$i$}] at (i) {};
\node[label=left:{$n$}] at (n) {};

\draw[opacity=0,name path=t2] (0.5, 10) to (0.5,6);
   \fill [name intersections={of=t2 and zigzag1, by={b,h, j, m}}]
        \foreach \x in {b, h, j, m}
        {(\x) circle (2pt)}
        (j) circle (2pt); 
    \foreach \x in {b, h}
        {\node[above=2pt] at (\x) {$\x$};}
        \node[below] at (j) {$j$};
           \node[below] at (m) {$m$};

\draw[opacity=0,name path=t3] (1.5, 10) to (1.5,6);
   \fill [name intersections={of=t3 and zigzag1, by={c,g, k, l}}]
        \foreach \x in {c,g, k, l}
        {(\x) circle (2pt)};
        \foreach \x in {c,g}
        {\node[label=above:{$\x$}] at (\x) {};}
                \node[below] at (k) {$k$};
                \node[below] at (l) {$\ell$};

\draw[opacity=0,name path=t4] (2, 10) to (2,6);
   \fill [name intersections={of=t4 and zigzag1, by={d,f, l, m}}]
        \foreach \x in {d,f, l, m}
        {(\x) circle (2pt)};
        \foreach \x in {d,f}
        {\node[label=above:{$\x$}] at (\x) {};}
                \node[below] at (l) {$1$};
                \node[below] at (m) {$1$};
                        
\draw[opacity=0,name path=t5] (3, 10) to (3,6);
   \fill [name intersections={of=t5 and zigzag1, by={e,  m}}]
        \foreach \x in {e,m}
        {(\x) circle (2pt)};
        \foreach \x in {e}
        {\node[label=above:{$\x$}] at (\x) {};}
                        \node[below] at (m) {$1$};

\end{tikzpicture} $\sim_b$
 \begin{tikzpicture}[ baseline={([yshift=-.5ex]current bounding box.center)},vertex/.style={anchor=base,circle,fill=black!25,minimum size=18pt,inner sep=2pt}]
  \draw[name path=zigzag1] (0,10) to (3,9) to (0,8) to (3,7) to (0,6);
   
   \draw[opacity=0,name path=t1] (0, 10) to (0,5);
   \fill [name intersections={of=t1 and zigzag1, by={a,  i,n}}]
        (a) circle (2pt)
        (i) circle (2pt)
        (n) circle (2pt);
\node[label=left:{$a$}] at (a) {};
\node[label=left:{$i$}] at (i) {};
\node[label=left:{$n$}] at (n) {};

\draw[opacity=0,name path=t2] (0.5, 10) to (0.5,6);
   \fill [name intersections={of=t2 and zigzag1, by={b,h, j, m}}]
        \foreach \x in {b, h, j, m}
        {(\x) circle (2pt)}
        (j) circle (2pt); 
    \foreach \x in {b, h}
        {\node[above=2pt] at (\x) {$\x$};}
        \node[below] at (j) {$j$};
           \node[below] at (m) {$m$};

\draw[opacity=0,name path=t3] (1.5, 10) to (1.5,6);
   \fill [name intersections={of=t3 and zigzag1, by={c,g, k, l}}]
        \foreach \x in {c,g, k, l}
        {(\x) circle (2pt)};
        \foreach \x in {c,g}
        {\node[label=above:{$\x$}] at (\x) {};}
                \node[above] at (k) {$k \cdot \ell$};
                \node[below] at (l) {$1$};

\draw[opacity=0,name path=t4] (2, 10) to (2,6);
   \fill [name intersections={of=t4 and zigzag1, by={d,f, l, m}}]
        \foreach \x in {d,f, l, m}
        {(\x) circle (2pt)};
        \foreach \x in {d,f}
        {\node[label=above:{$\x$}] at (\x) {};}
                \node[below] at (l) {$1$};
                \node[below] at (m) {$1$};
                        
\draw[opacity=0,name path=t5] (3, 10) to (3,6);
   \fill [name intersections={of=t5 and zigzag1, by={e,  m}}]
        \foreach \x in {e,m}
        {(\x) circle (2pt)};
        \foreach \x in {e}
        {\node[label=above:{$\x$}] at (\x) {};}
                        \node[below] at (m) {$1$};

\end{tikzpicture}}
  \end{equation*}

\end{enumerate}

It is straightforward to see, albeit tedious to check, that the components of our differentials described above are well defined with respect to the relation generated by both\footnote{The relation $\sim_b$ is required to have $c_z$ be well defined with respect to the relation $\sim_a$.} $\sim_a$ and $\sim_b$. In the definition below, we will describe the minimal relation generated by these two as $\sim_{a,b}$.\label{Page Equiv a,b}

\begin{definition}\label{def CH^ZZ}
Let $(\A, \nabla, R)$ be a curved differential graded algebra. We define the \emph{\underline{zigzag algebra}}
 \[ZZ(\A) := \bigoplus\limits_{\substack{n, k-1 \ge 0,\\ k \text{ even}}}   (\A \otimes(\A^{\otimes n}\otimes \A)^{\otimes k})[n] / \sim_{a,b},\]
  where $[n]$ denotes a total shift down by $n$ and $\sim_{a,b}$ is described above this definition. A monomial $\underline{x}_{k,n} \in {ZZ}(\A)_{k,n} \subset {ZZ}(\A)$ will be written as
\abovedisplayskip=5pt
 \belowdisplayskip=5pt
 \begin{align*}
 \underline{x}_{k,n} = x_{(0,0)}&\otimes (x_{(1,1)} \otimes \ldots \otimes x_{(1,n)} \otimes x_{(1,n+1)}) \otimes\\
  \ldots &\otimes (x_{(i,1)} \otimes \ldots \otimes x_{(i,n)} \otimes x_{(i,n+1)})\otimes \\
  \ldots &\otimes (x_{(k,1)} \otimes \ldots \otimes x_{(k,n)} \otimes x_{(k,n+1)}).
 \end{align*}

The differential $D_z: ZZ(\A)  \to ZZ(\A) $ is given by $D_z( \underline{x}_{k,n})\= (\nabla_z + b_z+ c_z)(  \underline{x}_{k,n})$ as described above (see page \pageref{describe Dz}: specifically \eqref{EQ: nabla_z}, \eqref{EQ b_z}, and \eqref{EQ c_z}), and the product $\odot$ is defined immediately below. 
\end{definition}

Next we define a shuffle product $\odot: ZZ(\A)  \otimes ZZ(\A)  \to ZZ(\A) $.  Recall that for $CH^I(\A)$, we could define a shuffle product, but it would not be compatible with the usual Hochschild differential, $D$, unless $\A$ was a commutative DGA.  For $ZZ(\A)$, the idea is to concatenate the two zigzags while shuffling in 1's. The unit for this product will be represented by a zigzag of units in $\A$. This shuffle product will be well-defined on the equivalence classes of $ZZ(\A)$ generated by $\sim_{a,b}$ and so we can merely define the product on representative monomials.

Before we introduce the shuffle product below, we first set some conventions and notation in the following remark on shuffles.
\begin{remark}[Some comments on shuffles and their signatures] \label{shuffle remarks}
Consider the set $S_{n,m}$ of $(n,m)$ shuffles:  $S_{n,m} = \{\sigma \in S_{n+m} \, | \, \sigma(1) < \ldots < \sigma(n) \ \text{ and } \ \sigma(n+1) < \ldots < \sigma(n+m)\}$.  We can use these shuffles to get maps $\sigma: \A^{\otimes n} \times \A^{\otimes m} \to \A^{\otimes n+m}$ via $$\sigma((a_1 \otimes \ldots \otimes a_n),( a_{n+1} \otimes \ldots \otimes a_{n+m})) : = a_{\sigma^{-1}(1)}   \otimes \ldots \otimes a_{\sigma^{-1}(n+m)}.$$Next, since we will be shuffling in $1$'s on zigs and zags in the opposite order, we offer a formal convention to describe that process.  For an element $E = e_1 \otimes \ldots \otimes e_p \in \A^{\otimes p}$, we define $\overline{E}\= e_{p} \otimes \ldots \otimes e_1 \in \A^{\otimes p}$.  Now, for a shuffle $\sigma \in S_{n,m}$ interpreted as $\sigma: \A^{\otimes n} \times \A^{\otimes m} \to \A^{\otimes n+m}$, we can define $\sigma^{\Sh} \in S_{n,m}$ giving the induced map $\sigma^{\Sh}: \A^{\otimes n} \times \A^{\otimes m} \to \A^{\otimes n+m}$ to be $$\sigma^{\Sh}(L,R)\= \overline{\sigma( \overline{L}, \overline{R} )}.$$  For a given $i\in \mathbb{N}$ we will define $\sigma^{\Sh_i}$ to be $\sigma$ if $i$ is odd (i.e. on a ``zig") and $\sigma^{\Sh}$ if $i$ is even (i.e. on a ``zag").  Finally, recall that for a shuffle (or permutation of any type) we can define the signature of that shuffle by $sgn(\sigma) = (-1)^{|\sigma|}$ where $|\sigma|$ is the parity if $\sigma$, i.e. the number of transpositions (mod 2) used when writing $\sigma$ as a product of transpositions.  By our convention, for reasons having nothing to do with zigs and zags, we will need to consider $sgn(\sigma^{\Sh})$.  It is straightforward to prove that $sgn(\sigma^{\Sh}) = (-1)^{\abs{\sigma^{\Sh}}}= (-1)^{nm+ \abs{\sigma}}$ where $\sigma$ shuffles an $n$-tuple with an $m$-tuple.
\end{remark}

\begin{definition}\label{zz shuffle}The shuffle product $\odot$ for $ZZ(\A)$ is defined on representative monomials by 
\[
\underline{x}_{k, n} \odot \underline{y}_{\ell, m} =  \sum\limits_{\sigma \in S_{n,m}} (-1)^{\abs{\sigma} +(\abs{\underline{x}} - n)\cdot m}  \underline{x}_{k, n} \odot_{\sigma} \underline{y}_{\ell, m}, \]
where $\odot_{\sigma}$ is defined as
\begin{align*}
 \underline{x}_{k, n} \odot_{\sigma} \underline{y}_{\ell, m} = &\left( x_{(0,0)} \otimes \bigotimes\limits_{q=1}^k  \bigotimes\limits_{i=1}^{n+1} x_{(q,i)} \right) \odot_{\sigma} \left( y_{(0,0)} \otimes \bigotimes\limits_{r=1}^{\ell}  \bigotimes\limits_{j=1}^{m+1} y_{(r,j)} \right) \\
=  &\, x_{(0,0)} \otimes \left(\bigotimes\limits_{q=1}^k   \sigma^{\Sh_i} \left(   \bigotimes\limits_{i=1}^{n} x_{(q,i)},  \bigotimes\limits_{j=1}^{m} 1 \right) \otimes x_{(q,n+1)} \right)\\
& \otimes y_{(0,0)} \otimes \left(\bigotimes\limits_{r=1}^{\ell}   \sigma^{\Sh_i} \left(   \bigotimes\limits_{i=1}^{n} 1,  \bigotimes\limits_{j=1}^{m} y_{(r, j)} \right) \otimes y_{(r,m+1)} \right) 
\end{align*}
and extended linearly to all of $ZZ(\A)$.  Here, $\abs{\sigma}$ is the sign of the shuffle, we use $|\underline{x_{k,n}}|\= |x_{(0,0)}| + |x_{(1,1)}| + \ldots + |x_{(k,n+1)}|$, and the shuffles are happening at all $i,j$ simultaneously for a given $\sigma$.  
\end{definition}

Note in the above definition one could consider the fact that $\abs{\sigma} +(\abs{\underline{x}} - n)\cdot m = \abs{\underline{x}}\cdot m + \abs{\sigma^{\Sh}}$ in connection with Remark \ref{shuffle remarks}.

\begin{prop}\label{propzzDderivation}
The shuffle product $\odot$ is associative, and $D_z$ is a derivation with respect to $\odot$.
\begin{proof}
Since the shuffle product is a sum over all shuffles, associativity of this shuffle product essentially amounts to the fact that concatenating zigzags on top of one another is an associative operation.  With respect to $D_z$ being a derivation, we wish to show that $D_z(\underline{x} \odot \underline{y}) = D_z(\underline{x}) \odot \underline{y} + (-1)^{|\underline{x}| -n}\underline{x} \odot D_z(\underline{y})$ for any two monomials  $\underline{x}_{k,n}, \underline{y}_{j,m} \in ZZ(\A)$.  We note that $D_z(\underline{x}) \odot \underline{y}$ amounts to applying the differential only to the top ($k$-mayn) zigzags where as $\underline{x} \odot D_z(\underline{y})$ applies the differential only to the bottom ($j$-many) zigzags.  Recall that $D_z = \nabla_z + b_z + c_z$ has three components and so $D_z(\underline{x} \odot \underline{y})$ applies first $\nabla$ to each term, in which case it is either being applied to the top or the bottom zigzags and so recovering those terms, $\nabla_z(\underline{x}) \odot \underline{y}$ and $\underline{x} \odot \nabla_z(\underline{y})$, is straightforward; of course, the left end-point at which the two monomials are joined requires application of the Leibniz rule for $\nabla$.  Thus, we have $\nabla_z(\underline{x} \odot \underline{y}) = \nabla_z(\underline{x}) \odot \underline{y} + (-1)^{|\underline{x}| -n}\underline{x} \odot \nabla_z(\underline{y})$.  

The proof that $b_z(\underline{x} \odot \underline{y}) = b_z(\underline{x}) \odot \underline{y} + (-1)^{|\underline{x}| -n}\underline{x} \odot b_z(\underline{y})$  holds is similar to the proof for the shuffle product on the Hochschild complex, but we offer the general idea here.  Recall that $b_z$ has components which we can refer to as \emph{columns} and $\odot$ has components which are $(n,m)$-\emph{shuffles}.  In observing a particular term of $b_z(\underline{x} \odot \underline{y})$,  say $b_{\ell}(\underline{x} \odot_{\sigma} \underline{y})$, we see that there are essentially two cases (outside of the endpoints).  In the first case, $b_{\ell}$ is joining the $\ell$ and $\ell+1$ columns where $\odot_{\sigma}$ has not alternated drawing from both $\underline{x}$ and $\underline{y}$ in these two columns.  If, for example, $\odot_{\sigma}$  has drawn from $\underline{x}$ on columns $\ell$ and $\ell +1$, then we have  $b_{\ell}(\underline{x} \odot_{\sigma} \underline{y}) =  b_{\ell'}(\underline{x}) \odot_{\sigma'} \underline{y}$ for the appropriate $\ell'$ and $\sigma'$.  Similarly, if for example, $\odot_{\sigma}$  has drawn from $\underline{y}$ on columns $\ell$ and $\ell +1$, then we have  $b_{\ell}(\underline{x} \odot_{\sigma} \underline{y}) =  (-1)^{|\underline{x}| -n}\underline{x} \odot_{\sigma'} b_{\ell'}(\underline{y})$ for the appropriate $\ell'$ and $\sigma'$. The second case is when $\odot_{\sigma}$ has alternated drawing from $\underline{x}$ and $\underline{y}$ in the column $\ell$ and column $\ell+1$.  In this case, such a term will cancel, meaning we will observe $b_{\ell}(\underline{x} \odot_{\sigma} \underline{y}) = -b_{\ell}(\underline{x} \odot_{\sigma'} \underline{y})$, where $\sigma'$ is a transposition of $\sigma$, reversing the choice of drawing from $\underline{x}$ and $\underline{y}$ on these two columns.  While the endpoints are quite different in this construction than that of the Hochschild complex, the above ideas carry over in observing the components $b_0$ and $b_{n+m}$ applied to the shuffle product.  

Finally, we observe that $c_z(\underline{x} \odot \underline{y}) = c_z(\underline{x}) \odot \underline{y} + (-1)^{|\underline{x}| -n}\underline{x} \odot c_z(\underline{y})$.  Recall that $c_z$ had components $c_{i, \ell}$ which essentially choose a new column ($\ell$) to insert a curvature element, and then choose which zigzag ($i$) to place it along.  If the curvature term is placed along one of the top zigzags, i.e. when $i \le k$, then we have 
$c_{i, \bullet}(\underline{x} \odot \underline{y}) = c_{i, \bullet}(\underline{x}) \odot \underline{y} $ whereas if the $i$-th zigzag is one of the bottom zigzags, i.e. when $i > k$, then $c_{i, \bullet}(\underline{x} \odot \underline{y}) = (-1)^{|\underline{x}| -n}\underline{x} \odot c_{i-k, \bullet}(\underline{y})$.
\end{proof} 
\end{prop}

\begin{theorem}\label{theorem zigzag is NGA}
Given a curved dga $(\A, \nabla, R)$, the zigzag algebra $(ZZ(\A), D_z, \underline{R_z}, \odot)$ forms a curved differential graded algebra, where $\underline{R_z}\= R \otimes 1 \otimes 1$ is the following monomial with one complete zigzag and no columns.
\begin{equation*} \underline{R_z} \quad = \quad
{\resizebox{!}{1cm}{\begin{tikzpicture}[baseline={([yshift=-.5ex]current bounding box.center)},vertex/.style={anchor=base,
    circle,fill=black!25,minimum size=18pt,inner sep=2pt}]
\clip(4.5,4.5) rectangle (8.5,8.5);    
%
%
%
%
%
   \draw[name path=zigzag2] (5,7) to (8,6) to (5,5);
   
   \draw[opacity=0,name path=s1] (5, 8) to (5,4);
   \fill [name intersections={of=s1 and zigzag2, by={a, c}}]
        (a) circle (2pt)
        (c) circle (2pt);
\node[label=left:{$R$}] at (a) {};
\node[label=left:{$1$}] at (c) {};

\draw[opacity=0,name path=s2] (8, 7) to (8,5);
   \fill [name intersections={of=s2 and zigzag2, by={b, blank}}]
        (b) circle (2pt);
\node[label=right:{$1$}] at (b) {};      
 \end{tikzpicture}}}
\end{equation*}
 \begin{proof}
We first note that $ZZ(\A)$ is in fact a graded algebra and that $\underline{R_z}$ is in fact a degree $2$ element of $ZZ(\A)$.  After having proven that $D_z$ is a derivation over $\odot$ in Proposition \ref{propzzDderivation}, it remains to be shown that $(D_z)^2 = [\underline{R_z}, - ]$, as degree $2$ operators on $ZZ(\A)$.  Since $\underline{R_z}$ has no interior columns, one can quickly calculate the relevant bracket with $\underline{x}_{k,n} = x_{(0,0)} \otimes \cdots \otimes x_{(k, n+1)}$ to be 
\[ [\underline{R_z}, \underline{x}_{k,n}] = \underline{R_z} \odot \underline{x}_{k,n} \pm \underline{x}_{k,n} \odot \underline{R_z} = (R \cdot x_{(0,0)}) \otimes \cdots \otimes x_{(k, n+1)} \pm x_{(0,0)} \otimes \cdots \otimes  (x_{(k, n+1)}\cdot R),\]
where we have applied the normalization relation $\sim_a$.

With $D_z \= \nabla_z + b + c_z$, we will analyze the six terms of $(D_z)^2 = \frac{1}{2} [D_z, D_z]$.  In attempt to communicate the important ideas, we will ignore the various $\pm$ signs in the following proof; the relevant terms will always appear in pairs with opposite signs and cancel. To start with we show that the term $[ b_z, \nabla_z]$ is equal to zero.  Recall that $b_z$ collapses the zigzag column-wise whereas $\nabla_z$ is a sum of applying the differential to each tensor factor, which, for the sake of this proof we will write as a sum of $\nabla_j$ where $j$ is a column and this term is being summed (with appropriate sign) over all zigzag (i.e. row) levels.  If the columns being collapsed by $b_i$ are not ``near'' the column where the derivative is being applied by $\nabla_j$ then cancellation happens immediately: for $i < j-1$, $b_i \circ \nabla_j = \nabla_{j-1} \circ b_i$; and for $j < i$, $b_i \circ \nabla_j = \nabla_j \circ b_i$.  If the column that $\nabla_i$ is being applied to is exactly that which $b$ has collapsed two columns into, then we use the fact that $\nabla$ is a derivation and so $\nabla_j \circ b_i$ in this case is recovered from applying first $\nabla$ on the left and right columns separately and then applying $b$ to collapse them accordingly.  In particular, for $i$ not ``at an endpoint'' of the zigzag we observe $\nabla_i \circ b_i = b_i \circ \nabla_i + b_i \circ \nabla_{i+1}$; and at the endpoints, we also have $\nabla_0 \circ b_0 = b_0 \circ \nabla_1 + b_0 \circ \nabla_0$ and $\nabla_n \circ b_n = b_n \circ \nabla_{n+1} + b_n \circ \nabla_n$.  Thus $[ b_z, \nabla_z] = 0$.

The proof that $(b_z)^2 = 0$ is almost identical to the analog in the 2-sided bar complex.  Next, to see that $c_z^2 = 0$ as well, and again using recalling the components $c_{i, \ell}$ of $c_z$ which place curvature terms at row $i$ in the new column $\ell$, it is first important to note that if columns $\ell$ and $\ell'$ are not ``near'' one another, then $[c_{\bullet, \ell}, c_{\bullet, \ell'}] = 0$.  To check what happens when the columns are near one another, we see that $c_{\bullet, \ell} \circ c_{\bullet, \ell} = c_{\bullet, \ell+1} \circ c_{\bullet, \ell}$.  Since the index $\ell$ in $c_{\bullet, \ell}$ is never pointing to an endpoint in the resulting zigzag, the above two considerations are enough to show that $c_z^2 = 0$.  

Using the Bianchi identity, $\nabla(R) = 0$, we can show that $[ \nabla_z, c_z] = 0$ as well.  The term which uses the Bianchi identity is $\nabla_i \circ c_i$, whereby a curvature term is placed in the newly-formed $i$ column, and then $\nabla$ is applied as a sum along the terms in that column.  Since $\nabla$ will only be applied to either a $1$ or a curvature term, $R$, in either case the result is zero.  The remaining terms of $[ \nabla_z, c_z] $ all cancel since the $\nabla$ is not being applied ``near'' the curvature column and so cancellation occurs for degree reasons.

Finally, we show that $\nabla_z^2 + [ b, c_z] = [ \underline{R_z}, -]$.  Recall that in our curved dga, we have $(\nabla_z)^2(\omega) = [ R, \omega]$.  Since the right hand side of this equation is reflected in the application of $(D_z)^2$ on zigzags by first inserting a curvature column to the left and the right of column $i$, (i.e. $c_{\bullet, i}$ and $c_{\bullet,i+1}$) and then collapsing that column from the left and the right to obtain $\nabla_i^2 = b_{i} \circ c_{\bullet,i} + b_i \circ c_{\bullet, i+1}$.  Of course, this is only possible when $i$ is not at an endpoint.  For the right endpoints, we have $(\nabla_n)^2 = b_{n} \circ c_{\bullet ,n+1}$, since $c_{\bullet,n+1}$ will end up placing a curvature term both before and after the endpoint (as a sum) and then $b_{n}$ collapses these appropriately.  Similarly, for left endpoints, but not at the initial or terminal zigzag (or row), we have $\nabla_0^2 = b_{0} \circ c_{\bullet, 1}$.  At the initial column and zig (i.e. at $x_{(0,0)}$), the operator $b_0 \circ c_{\bullet, 1}$ only places a curvature term on the right of $x_{(0,0)}$ before collapsing so we are off by a zigzag with initial point being $R \cdot x_{(0,0)}$.  At the terminal column and zag (i.e. at $x_{(k, n+1)}$) the operator $b_{0} \circ c_{\bullet, 1}$ only places a curvature term on the right of $x_{(0,0)}$ before collapsing (order-wise) so we are off by a zigzag with terminal point being $x_{(k,n+1)} \cdot R$.  These two terms are recovered by the operator $[ \underline{R_z}, - ]$, up to the equivalence relation imposed by $\sim_{a,b}$ from page \pageref{Page Equiv a,b}.  Note that we have analyzed the terms of $ [ b, c_z] $ where the columns being operated on are near one another but for most $i$ and $j$, we have $[ b_i , c_{\bullet,j} ] = 0$.  

Thus we have proven that $D_z^2 = \nabla_z^2 + [ b, c_z] = [ \underline{R_z}, - ]$.
\end{proof}
\end{theorem}

It is straightforward to apply the definitions and check that our zigzag algebra construction is functorial, as we record below.
\begin{prop}
The zigzag algebra is functorial in the sense that if $f: \mathcal{A} \to \mathcal{B}$ is a map of curved differential graded algebras, then the induced map $ZZ(f): ZZ(\mathcal{A}) \to ZZ(\mathcal{B})$ is a map of curved differential graded algebras.
\end{prop}

\subsection{The Zigzag algebra retracts to the original algebra}
This section adapts ideas presented in \cite[Lemma 3.3]{GJP},  where they show that the dga $\mathcal{A} \= \Omega(M, \mathbb{R})$ is chain-homotopic to the normalized Hochschild complex  $N(\mathcal{A})$. We briefly summarize their argument here for the reader's conveneince. Monomials in $N(\mathcal{A})$ could be thought of as tensored forms, denoted $\omega_0 \lbrack \omega_1 \vert \omega_2  \vert \ldots \vert \omega_j \rbrack \omega_{j+1}$.  There is a map $\alpha: \mathcal{A} \to N(\mathcal{A})$ given on monomials by $\omega \mapsto \omega \lbrack \ \rbrack 1$ and a map the other way, $\eta: N(\mathcal{A}) \to \mathcal{A}$ given on monomials by $\omega_0 \lbrack \ \rbrack \omega_1 \mapsto \omega_0  \omega_1$ and is $0$ for $j>0$.  These two maps are quasi-inverses since $\eta \circ \alpha = id_{\mathcal{A}}$ and $id_{N(\mathcal{A})} - \alpha \circ \eta = [D, s]$ for a chain-homotopy, $s: N(\mathcal{A}) \to N(\mathcal{A})$, that contracts each monomial, 
\[ \omega_0 \lbrack \omega_1 \vert \omega_2  \vert \ldots \vert \omega_j \rbrack \omega_{j+1} \mapsto \omega_0 \lbrack \omega_1 \vert \omega_2  \vert \ldots \vert \omega_j  \vert \omega_{j+1} \rbrack 1.\]

There are various reasons for why this exact same contraction and bar-Hochschild model doesn't work in the curved case.  However, we can adopt this same idea to our setting.  Below, we introduce the zigzag-analogs of the maps $\alpha$ and $\eta$ from above, and then finally we prove that they are homotopy inverse to one another.  

The map, $\alpha: \mathcal{A} \to ZZ(\mathcal{A})$, will take a monomial, $\omega \in \mathcal{A}$, and send it to a zigzag where that monomial is placed at the start and all other slots are filled with units:
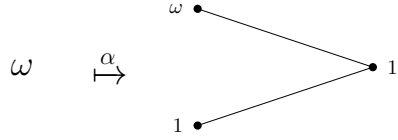
\begin{equation*} 
{\resizebox{!}{2cm}{\begin{tikzpicture}[baseline={([yshift=-.5ex]current bounding box.center)},vertex/.style={anchor=base,
    circle,fill=black!25,minimum size=18pt,inner sep=2pt}]
\clip(1.5,4) rectangle (9,9);
       \node at (2,6) {\huge $\omega$};      
%
%
%
%
%
   \draw[name path=zigzag2] (5,7) to (8,6) to (5,5);
   
   \draw[opacity=0,name path=s1] (5, 8) to (5,4);
   \fill [name intersections={of=s1 and zigzag2, by={a, c}}]
        (a) circle (2pt)
        (c) circle (2pt);
\node[label=left:{$\omega$}] at (a) {};
\node[label=left:{$1$}] at (c) {};

\draw[opacity=0,name path=s2] (8, 10) to (8,4);
   \fill [name intersections={of=s2 and zigzag2, by={b, blank}}]
        (b) circle (2pt);
\node[label=right:{$1$}] at (b) {};      
         \node at (3.5,6) {\huge $\xmapsto{\alpha}$};
 \end{tikzpicture}}}
\end{equation*}

The map, $\eta: ZZ(\mathcal{A}) \to \mathcal{A}$, will take most zigzags to zero, except for when $k=2$ and $n=0$ in which case it will simply multiply the endpoints: 
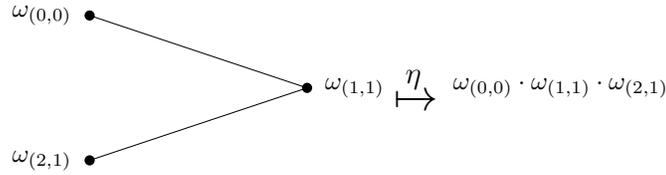
\begin{equation*} 
{\resizebox{!}{2cm}{\begin{tikzpicture}[baseline={([yshift=-.5ex]current bounding box.center)},vertex/.style={anchor=base,
    circle,fill=black!25,minimum size=18pt,inner sep=2pt}]
\clip(-0.5,4) rectangle (9,8);
       \node at (7.5,6) { $\omega_{(0,0)} \cdot \omega_{(1,1)} \cdot \omega_{(2,1)} $};      
%
%
%
%
%
   \draw[name path=zigzag2] (1,7) to (4,6) to (1,5);
   
   \draw[opacity=0,name path=s1] (1, 8) to (1,4);
   \fill [name intersections={of=s1 and zigzag2, by={a, c}}]
        (a) circle (2pt)
        (c) circle (2pt);
\node[label=left:{$\omega_{(0,0)}$}] at (a) {};
\node[label=left:{$\omega_{(2,1)}$}] at (c) {};

\draw[opacity=0,name path=s2] (4, 10) to (4,4);
   \fill [name intersections={of=s2 and zigzag2, by={b, blank}}]
        (b) circle (2pt);
\node[label=right:{$\omega_{(1,1)}$}] at (b) {};      
         \node at (5.5,6) {\huge $\xmapsto{\eta}$};
 \end{tikzpicture}}}
\end{equation*}

These two maps will be proven to be quasi-inverse to one another by introducing a chain homotopy on $ZZ(\mathcal{A})$.  Similar to the chain homotopy of \cite{GJP}, the map, $s$, will place a unit at the right end-point of all zigzags and pull the data which was on those endpoints into the interior of the zigzags.  The visual representation of what $s: ZZ(\mathcal{A}) \to ZZ(\mathcal{A})$ will do on zigzags is the following:
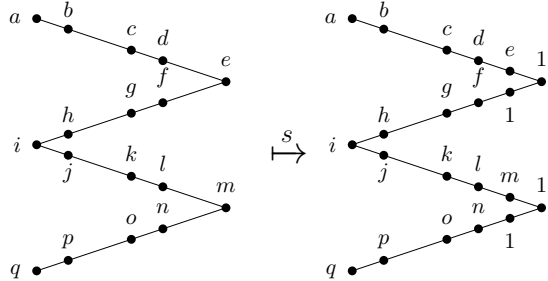
\begin{equation*} 
{\resizebox{!}{3cm}{\begin{tikzpicture}[baseline={([yshift=-.5ex]current bounding box.center)},vertex/.style={anchor=base,
    circle,fill=black!25,minimum size=18pt,inner sep=2pt}]
\clip(-1,4) rectangle (9,11);
%
%
%
%
%
   \draw[name path=zigzag1] (0,10) to (3,9) to (0,8) to (3,7) to (0,6);
   \draw[opacity=0] (-0.2, -0.2) rectangle (18, 12.2);
   
   \draw[opacity=0,name path=t1] (0, 10) to (0,5);
   \fill [name intersections={of=t1 and zigzag1, by={a,  i,q}}]
        (a) circle (2pt)
        (i) circle (2pt)
        (q) circle (2pt);
\node[label=left:{$a$}] at (a) {};
\node[label=left:{$i$}] at (i) {};
\node[label=left:{$q$}] at (q) {};

\draw[opacity=0,name path=t2] (0.5, 10) to (0.5,6);
   \fill [name intersections={of=t2 and zigzag1, by={b,h, j, p}}]
        \foreach \x in {b, h, p}
        {(\x) circle (2pt)}
        (j) circle (2pt); 
    \foreach \x in {b, h, p}
        {\node[above=2pt] at (\x) {$\x$};}
        \node[below] at (j) {$j$};

\draw[opacity=0,name path=t3] (1.5, 10) to (1.5,6);
   \fill [name intersections={of=t3 and zigzag1, by={c,g, k, o}}]
        \foreach \x in {c,g, k, o}
        {(\x) circle (2pt)};
        \foreach \x in {c,g, k, o}
        {\node[label=above:{$\x$}] at (\x) {};}
        
\draw[opacity=0,name path=t4] (2, 10) to (2,6);
   \fill [name intersections={of=t4 and zigzag1, by={d,f, l, n}}]
        \foreach \x in {d,f, l, n}
        {(\x) circle (2pt)};
        \foreach \x in {d,f, l, n}
        {\node[label=above:{$\x$}] at (\x) {};}
        
\draw[opacity=0,name path=t5] (3, 10) to (3,6);
   \fill [name intersections={of=t5 and zigzag1, by={e,  m}}]
        \foreach \x in {e,m}
        {(\x) circle (2pt)};
        \foreach \x in {e,m}
        {\node[label=above:{$\x$}] at (\x) {};}
         
%
%
%
%
%
   \draw[name path=zigzag2] (5,10) to (8,9) to (5,8) to (8,7) to (5,6);
   \draw[opacity=0] (-5.2, -5.2) rectangle (10, 12.2);
   
   \draw[opacity=0,name path=s1] (5, 10) to (5,5);
   \fill [name intersections={of=s1 and zigzag2, by={a,  i,q}}]
        (a) circle (2pt)
        (i) circle (2pt)
        (q) circle (2pt);
\node[label=left:{$a$}] at (a) {};
\node[label=left:{$i$}] at (i) {};
\node[label=left:{$q$}] at (q) {};

\draw[opacity=0,name path=s2] (5.5, 10) to (5.5,6);
   \fill [name intersections={of=s2 and zigzag2, by={b,h, j, p}}]
        \foreach \x in {b, h, p}
        {(\x) circle (2pt)}
        (j) circle (2pt); 
    \foreach \x in {b, h, p}
        {\node[above=2pt] at (\x) {$\x$};}
        \node[below] at (j) {$j$};

\draw[opacity=0,name path=s3] (6.5, 10) to (6.5,6);
   \fill [name intersections={of=s3 and zigzag2, by={c,g, k, o}}]
        \foreach \x in {c,g, k, o}
        {(\x) circle (2pt)};
        \foreach \x in {c,g, k, o}
        {\node[label=above:{$\x$}] at (\x) {};}
        
\draw[opacity=0,name path=s4] (7, 10) to (7,6);
   \fill [name intersections={of=s4 and zigzag2, by={d,f, l, n}}]
        \foreach \x in {d,f, l, n}
        {(\x) circle (2pt)};
        \foreach \x in {d,f, l, n}
        {\node[label=above:{$\x$}] at (\x) {};}
        
\draw[opacity=0,name path=s5] (7.5, 10) to (7.5,6);
   \fill [name intersections={of=s5 and zigzag2, by={e, z, m, w}}]
        \foreach \x in {e,z,m,w}
        {(\x) circle (2pt)};
        \foreach \x in {e,m}
        {\node[label=above:{$\x$}] at (\x) {};}
        \foreach \x in {z,w}
        {\node[label=below:{$1$}] at (\x) {};}

 \draw[opacity=0,name path=s6] (8, 10) to (8,6);
   \fill [name intersections={of=s6 and zigzag2, by={x,  y}}]
        \foreach \x in {x,y}
        {(\x) circle (2pt)};
        \foreach \x in {x,y}
        {\node[label=above:{$1$}] at (\x) {};}
        
         \node at (4,8) {\huge $\xmapsto{s}$};
\end{tikzpicture}}}
\end{equation*}

\begin{prop}\label{algebraic quasi-iso}
The maps of curved dgas $\alpha: {ZZ}(\mathcal{A}) \to \mathcal{A}$ and $\eta: \mathcal{A} \to {ZZ}(\mathcal{A})$ are homotopy equivalences.  In particular, there is an isomorphism between their curved cohomology algebras $H_{cur}\left( \mathcal{A} \right) \iso H_{cur}\left({ZZ}(\mathcal{A})\right)$.
\begin{proof}
Define the map $\eta: \mathcal{A} \to  {ZZ}(\mathcal{A})$ to be the map given on monomials by
\[ \omega \xmapsto{\eta} \omega \otimes  1\otimes  1, \]
so that, in the notation of Definition \ref{def CH^ZZ}, $\eta(\omega)$ has $n=0$ and $k=2$ with $1 = s(\omega)_{(1,1)} = s(\omega)_{(2,1)}.$  It is straightforward to check that $\eta$ is a morphism of curved dga's (Definition \ref{Def cdga morphism}) with respect to the differentials, and it is easily seen to be a map of algebras once the normalization relation (i.e. $\sim_{a,b}$ from page \pageref{Page Equiv a,b}) on the zigzags are used.  Next, we define the map $\alpha:  {ZZ}(\mathcal{A}) \to  \mathcal{A}$ on a monomial $\underline{x}_{k,n}$ (see Definition \ref{def CH^ZZ} for notation)  by 
\begin{equation}
\alpha( \underline{x}_{k,n}) = \begin{cases} x_{(0,0)} \cdot x_{(1,1)} x_{(2,1)} \cdot \ldots \cdot  x_{(k,1)}& \text{for }  n=0 \\
0 & \text{otherwise}.
\end{cases}
\end{equation}
Again, once one uses the normalization relation, checking that $\alpha$ is a morphism of curved dgas is straightforward.  A simpler observation is that we have $\alpha \circ \eta = id_{\mathcal{A}}$, since the image of $\eta$ is a zigzag with $n=0$ and so $\alpha \circ \eta (\omega) = \omega \cdot 1 \cdot 1$.  

Next, we define the contracting homotopy, $s: ZZ^p(\mathcal{A}) \to ZZ^{p-1}(\mathcal{A})$, that will satisfy our desired homotopy equation $id + \eta \o \alpha = [D_z, s]$.  (Just as in the proof of Theorem \ref{theorem zigzag is NGA}, the intermediate arguments in this proof will be modulo $\pm$ sign.)  This map $s$ will take a monomial with $n$-columns to a monomial with $(n+1)$-columns, without changing the number of zigzags, by moving the data at the right endpoints in towards the newly formed $(n+1)$st-column and then place a unit at the right endpoints, now labeled by $(n+2)$.  To be precise, for a monomial $\underline{x}_{k,n}$ we define the monomial $s( \underline{x})_{k,n+1}$ by 

 \begin{align*}
 s(\underline{x}_{k,n})  = x_{(0,0)}&\otimes (x_{(1,1)} \otimes \ldots \otimes x_{(1,n)} \otimes  x_{(1,n+1)} \otimes 1) \otimes \\
  \ldots &\otimes ( 1  \otimes x_{(2i,1)} \otimes \ldots \otimes x_{(2i,n+1)}  )\otimes \\
    \ldots &\otimes ( x_{(2j+1,1)} \otimes \ldots \otimes x_{(2j+1,n+1)} \otimes 1 )\otimes \\
  \ldots &\otimes (1 \otimes x_{(k,1)} \otimes \ldots \otimes x_{(k,n+1)} )
 \end{align*}

We now prove that $\eta \circ \alpha = D_z \circ s + s \circ D_z$.  In particular, we first show that $b_z \circ s  + s \circ b_z = \eta \circ \alpha$.  Consider a general monomial $\underline{x}_{k,n}$, with $n>0$, and recall that the operator $b_z = \sum\limits_{0\le \ell  \le n} b_{\ell}$ combines the $\ell$-th with  the $(\ell +1)$-th column.  Note that $(b_z \circ s)(\underline{x}_{k,n})$ is a sum of $(n+2)$-many terms, while $(s \circ b_z)(\underline{x}_{k,n})$ is a sum of $(n+1)$-many terms.   For each $\ell < n$, we obtain identical terms for $b_{\ell} \circ s$ and $s \circ b_{\ell}$ of the form,
 \begin{align*}
 x_{(0,0)}&\otimes (x_{(1,1)} \otimes \ldots \otimes x_{(1, p)} \cdot x_{(1, p+1)} \ldots \otimes x_{(1,n)} \otimes  x_{(1,n+1)} \otimes 1) \otimes \\
  \ldots &\otimes ( 1  \otimes x_{(2i,1)} \otimes  \ldots \otimes x_{(2i, n-p)} \cdot x_{(2i, n-p+1)} \ldots \otimes x_{(2i,n+1)}  )\otimes \\
    \ldots &\otimes ( x_{(2j+1,1)} \otimes \ldots \otimes x_{(2j+1, p)} \cdot x_{(2j+1, p+1)} \ldots \otimes x_{(2j+1,n+1)} \otimes 1 )\otimes \\
  \ldots &\otimes (1 \otimes x_{(k,1)} \otimes \ldots \otimes x_{(k, n-p)} \cdot x_{(k, n-p+1)} \ldots \otimes x_{(k,n+1)} )
 \end{align*}
and so they cancel.  For $\ell = n$, we obtain a term for each $b_{\ell} \circ s$ and $s \circ b_{\ell}$ which are not immediately equal, but under the normalization relations are equal to a monomial of the form, 
 \begin{align*}
 x_{(0,0)}&\otimes (x_{(1,1)} \otimes \ldots \otimes x_{(1,n)} \cdot  x_{(1,n+1)} \otimes 1) \otimes \\
  \ldots &\otimes (  x_{(2i,1)} \otimes  \ldots \otimes x_{(2i,n+1)}  )\otimes \\
    \ldots &\otimes ( x_{(2j+1,1)} \otimes \ldots  \otimes x_{(2j+1,n)} \cdot x_{(2j+1,n+1)} \otimes 1 )\otimes \\
  \ldots &\otimes ( x_{(k,1)} \otimes \ldots  \otimes  x_{(k,n+1)} )
 \end{align*}
and these two terms cancel as well.  Finally note that when $\ell = n+1$, $b_{\ell} \circ s = id$ where there is no such term for $s \circ b_{\ell}$.  In summary, for $n>0$, since $(\eta \circ \alpha) (\underline{x}_{k,n}) = 0$, we have the desired 
\[( id + \eta \circ \alpha) (\underline{x}_{k,n}) = (b_z \circ s + s\circ b_z) (\underline{x}_{k,n}).\]  
To conclude our analysis of the $b_z$-component of $D_z$, we now consider the case where $n = 0$, so that for such a monomial, $\underline{x}_{k,0}$, $b_z$ is defined to be zero on this element.  Thus one can check that,
\begin{align*} (b_z \circ s + s\circ b_z) (\underline{x}_{k,0}) &= (b_z \circ s) (\underline{x}_{k,0}) + (s\circ b_z) (\underline{x}_{k,0}) = (b_z \circ s) (\underline{x}_{k,0}) \\
 &= (b_0 \circ s) (\underline{x}_{k,0})   +  (b_1 \circ s) (\underline{x}_{k,0})   =(\eta \circ \alpha) ( \underline{x}_{k,0})+  \underline{x}_{k,0},
\end{align*}
where the last equality involving $\eta \circ \alpha$ makes use of the normalization relations.  To offer some additional detail, the calculation is the following:
\begin{align*}
(b_0 \circ s) (\underline{x}_{k,0}) &= (b_0 \circ s) \left( x_{(0,0)} \otimes (x_{(1,1)}) \otimes (x_{(2,1)}) \ldots \otimes (x_{(k,1)}) \right) \\
&= b_0 \left( x_{(0,0)} \otimes (x_{(1,1)} \otimes 1) \otimes (1 \otimes x_{(2,1)}) \ldots \otimes (1 \otimes x_{(k,1)}) \right)\\
&=  x_{(0,0)}\cdot  x_{(1,1)} \otimes 1\otimes ( x_{(2,1)} \cdot x_{(3,1)}) \ldots \otimes (x_{(k-1,1)}\cdot x_{(k,1)})
\end{align*}
where the last term is equivalent to $(\eta \circ \alpha) (\underline{x}_{k,0}) = x_{(0,0)} \cdot x_{(1,1)} \cdot \ldots \cdot x_{(k,1)} \otimes 1 \otimes 1$, which is a monomial with $n=0$ and $k=2$, by way of $\sim_{a,b}$.

For the component, $\nabla_z$, of $D_z$, we note that $\nabla_z \circ s + s \circ \nabla_z = 0$ as a straightforward calculation since $\nabla(1) = 0$.  We also have $c_z \circ s - s \circ c_z = 0$ as well, however, there is one application of $\sim_{a,b}$ where a curvature term is placed just before the right endpoint, both on a zig, and then on a zag.  Since these two curvature terms are placed at the same column and are separated by a unit at the right endpoint, they are equivalent after applying the relation and so we obtain the desired result: $b \circ s + s \circ b = id + \eta \circ \alpha.$  
\end{proof}
\end{prop}

\section{A Chen Map out of the ZigZag Hochschild Complex}\label{section iterated integral}
Our ultimate goal for $ZZ(\A)$ is to model $\Omega_{\wt \nabla}(PM, ev_0^* \End (E))$, the non-abelian curved dga of endormorphism-valued differential forms on the path space.  To recall from Examples \ref{EX: Omega nabla} and  \ref{EX: Omega(PM)}, let $(E, \nabla^E)\to M$ be an arbitrary vector bundle with connection.  Then $(\E,\nabla)\to M$ denotes the induced endomorphism bundle $\E \= E\otimes E^\vee = \End(E)$ with the naturally induced connection $\nabla = \nabla^E \otimes \nabla^{E \vee} =  ad_{\nabla^E}  = [\nabla^E, -]$.  We view the path space $PM = C^\infty([0,1], M)$ as an infinite-dimensional smooth manifold; the tangent space  $T(PM)_\gamma$ is $\Omega^0([0,1], \gamma^* TM)$, the collection of smooth vector fields along $\gamma$.  

For any $t \in [0,1]$ let $(\E_t, \nabla^t) \to PM$ denote the pullback of the endomorphism bundle along the evaluation map $ev_t$.
\[ \begin{tikzcd}
(\E_t, \nabla^t) \= (ev_t^* \E, ev_t^*\nabla) \ar[r] \ar[d] & (\E, \nabla) \ar[d] \\
PM \ar[r, "ev_t"] & M
\end{tikzcd}  \]
Similarly, let $E_t$ and $E_t^\vee$ denote the pullbacks of $E$ and $E^\vee$, respectively.  When $t=0$, which is the case we use most, we use the abbreviation $(\E_0, \wt{\nabla}) \= (ev_0^* \End (E), ev_0^* \nabla) \to PM$ and denote the curved dga by $\Omega_{\wt{\nabla}}(PM, \E_0)$.

The goal of this section is to construct a homotopy equivalence in $\NGA$ 
\[ It \colon ZZ\big(\Omega_\nabla(M,\E)\big) \overset{\simeq}{\longrightarrow} \Omega_{\wt{\nabla}}(PM,\E_0)  \]
that is analogous to Chen's iterated integral map.  To do so, we first recall some basic facts about integration (Subsection \ref{subsec:Integration}) and parallel transport (Subsection \ref{subsec:ParallelTransport}).  In Subsection \ref{subsec: define It}, we define the iterated integral map and prove it is a morphism of curved dgas.  It will then easily follow that the map a homotopy equivalence.

\subsection{Integration}\label{subsec:Integration}
Suppose $Y \xrightarrow{p} X$ is a smooth fiber bundle whose fibers are diffeomorphic to $F$, a compact oriented $n$-dimensional manifold.  The classical integration over fibers map
\begin{equation} \label{eq:Integration} \int\limits_F \colon \Omega^\bullet (Y) \longrightarrow \Omega^{\bullet - n} (X) \end{equation}
exists and satisfies the Stokes formula
\begin{equation}\label{eq:ClassicalStokes}
(-1)^n \,\,\, d \int\limits_F \omega \quad  = \quad  \int\limits_F d \omega \quad - \quad \int\limits_{\partial F} \omega.\end{equation}

If $E \to Y$ is an arbitrary vector bundle, the integration map \eqref{eq:Integration} need not extend to the domain $\Omega(Y,E)$.  For $\omega \in \Omega(Y,E)$, defining  $(\int_F \omega)$ at the point $x\in X$ would essentially involve integrating $\omega_y$ over all $y\in p^{-1}(x)$.  Since the $\omega_y$ are vectors in different vector spaces $E_y$, it does not make sense to sum or integrate them.  However, this is no longer a problem if $(p^*E, p^*\nabla ) \to Y$ is a bundle (with connection) that is pulled back from a bundle on the base $(E,\nabla) \to X$.  Note that while the connection $\nabla$ is essential to the Stokes formula, it is not used to define the integration.

\begin{prop}\label{prop:Integration}Let $F \hookrightarrow Y \xrightarrow{p}X$ be a smooth fiber bundle, with fiber $F$ a compact oriented $n$-manifold, and $(E,\nabla)\to X$ a vector bundle with connection.  Then, the usual integration over fibers  \eqref{eq:Integration} naturally extends to a map 
\begin{equation}\label{eq:BundleIntegration} \int\limits_F \colon \Omega_{(p^*\nabla)}^\bullet(Y, p^*E) \to \Omega_{\nabla}^{\bullet-n}(X,E) \end{equation} that satisfies the (generalized) Stokes formula
\begin{equation}\label{eq:BundleStokes} (-1)^n \,\,\, \nabla \int\limits_F \omega \quad = \quad \int\limits_F (p^*
\nabla) \omega \quad -  \quad \int\limits_{\partial F} \omega.  \end{equation}
\end{prop}
\begin{proof}
Since differential forms are sections of a bundle, it suffices to define $(\int_F\omega)_x$ for general $x\in X$.  Let $\{ e_\alpha \}$ be a local basis of sections of $E$ near $x \in X$.  At all points $y$ such that $p(y)=x$, this induces a natural basis $\{ p^*e_\alpha \}$ of $p^*E$.  For $\omega = \omega^\alpha \otimes (p^*e_\alpha) \in  \Omega^\bullet(Y, p^*E)$, the integral is defined $\int_F (\omega^\alpha \otimes (p^*e_\alpha)) \= (\int_F \omega^\alpha) \otimes e_\alpha \in \Omega^{\bullet -n}(X,E)$.  

The generalized Stokes formula \eqref{eq:BundleStokes}  follows from a direct calculation.
\begin{align*}
(-1)^n \, \nabla\int\limits_F \omega &=  (-1)^n \, \nabla \! \int\limits_F \omega^\alpha \otimes p^*e_\alpha = (-1)^n \nabla \left[ (\int\limits_F \omega^\alpha)\otimes e_\alpha \right] \\
&=  (-1)^n \left[  (d\int\limits_F \omega^\alpha) \otimes e_\alpha + (-1)^{|\omega|-n} (\int\limits_F \omega^\alpha) \otimes \nabla e_\alpha \right] \\
&= \left( \int\limits_F d\omega^\alpha \right) \otimes e_\alpha  -  \int\limits_{\partial F} \omega^\alpha \otimes e_\alpha + (-1)^{|\omega|} \int\limits_F \omega \otimes (p^*\nabla)p^*e_\alpha \\
&= \int\limits_F (p^*\nabla) (\omega^\alpha\otimes p^*e_\alpha) - \int\limits_{\partial F} \omega^\alpha\otimes e_\alpha = \int\limits_F (p^*\nabla) \omega - \int\limits_{\partial F} \omega 
\end{align*}
The first and fourth lines rewrite using definitions, the second line uses the Leibniz property of $\nabla$, and the third line uses the classical Stokes formula \eqref{eq:ClassicalStokes}.
\end{proof}

\begin{remark}
The $\pm$ signs in \eqref{eq:BundleStokes} are obtained directly from the $\pm$ signs in the original Stokes formula \eqref{eq:ClassicalStokes}.  If a different $\pm$ sign convention is used, the bundle-valued Stokes formula will still be the classical formula with $d$ replaced by $\nabla$.
\end{remark}

The Stokes formula is crucial to verifying that our iterated integral map in Definition \ref{zigzag Chen} is a morphism in $\NGA$,  but it also implies the homotopy equivalence $\Omega_{\wt{\nabla}}(PM, \E_0) \simeq \Omega_{\nabla}(M, \E)$.

\begin{prop}\label{prop:PMhtpyM}The curved dgas $\Omega_{\wt{\nabla}}(PM, \E_0)$ and $\Omega_{\nabla}(M, \E)$ are naturally homotopy equivalent.  More concretely, inclusion along constant maps $i\colon M\to PM$ and the evaluation map $ev_0\colon PM\to M$ induce morphisms in $\NGA$ 
\[ \begin{tikzcd}
\Omega_{\wt\nabla} (PM, \E_0) \ar[r, bend left=25, "{i^*}"]  & \Omega_{\nabla}(M, \E) \ar[l, bend left=25, "{ev_0^*}"]
\end{tikzcd} \]
such that  $i^* \o ev_0^*$ is the identity and $ev_0^* \o i^*$ is homotopic to the identity.
\end{prop}
\begin{proof}
By Proposition \ref{prop:BundleDGAMaps}, $ev_0^*\colon \Omega_{\nabla}(M, \E) \to  \Omega_{\wt{\nabla}}(PM, \E_0)$ is a well-defined morphism of curved dgas, since $\E_0 = ev_0^* \E$.  Similarly, $ev_0 \o\, i = id_M$ implies $i^*\E_0 = i^*(ev_0^*\E) = \E$, and thus  $i^*\colon \Omega_{\wt{\nabla}}(PM, \E_0) \to \Omega_{\nabla}(M, \E)$ is a curved dga morphism.  It immediately follows that the composition $i^* \o ev_0^* = id_{\Omega_{\nabla}(M, \E)}$.

We now show that the usual argument for $PM \simeq M$ implies that the composition $ev_0^* \o\, i^*$ is chain homotopic to the identity.  Let $F\colon [0,1]\times PM \to PM$ be the standard homotopy that shrinks a path to its starting point. 
\[ F(s, \gamma)(t) =\begin{cases} \gamma(t) & t \leq s \\ \gamma(s) & s \leq t. \end{cases} 
\quad  \quad \quad
 \begin{tikzcd}
\{0\} \times PM \ar[d, hook] \ar[dr, "i \o ev"]\\
{[0,1]} \times PM \ar[r, "F"]& PM\\
\{1\} \times PM \ar[u, hook'] \ar[ur, "id_{PM}" '] 
\end{tikzcd} \]
 Proposition \ref{prop:Integration} gives the integration map $\int_{[0,1]} \colon \Omega^\bullet_{F^*\wt{\nabla}}([0,1]\times PM, F^* \E_0) \to \Omega_{\wt{\nabla}}^{\bullet-1}(PM, \E_0)$, and composing with $F^*$ gives the homotopy $h = \int_{[0,1]}F^*\colon \Omega_{\wt{\nabla}}^{\bullet}(PM, \E_0) \to \Omega_{\wt{\nabla}}^{\bullet-1}(PM, \E_0)$.  The Stokes formula \eqref{eq:BundleStokes}  implies that $h$ is a chain homotopy of curved dgas, since
\[  (-1) \wt{\nabla}\int\limits_{[0,1]} F^* \omega = \int\limits_{[0,1]}F^*(\wt{\nabla}\omega) - \int\limits_{\partial[0,1]}  F^* \omega  \]
and therefore $id^*\omega - (i \o ev)^* \omega =  h( \wt{\nabla} \omega) + \wt{\nabla}(h\omega).$
\end{proof}

\subsection{Parallel transport of bundle-valued forms}\label{subsec:ParallelTransport}  
For a fixed $a, b \in [0,1]$ we denote the parallel transport section by $\ptrans{}{a}{b} \in \Omega^0(PM, E_b \otimes E_a^\vee)$; its value at a path $\gamma \colon [0,1]\to M$ is the parallel transport along $\gamma$
\[ \ptrans{\gamma}{a}{b}\colon E_{\gamma(a)} \overset{\cong}\longrightarrow E_{\gamma(b)}.\]
We will use the notation $\ptrans{\gamma}{}{}$ when the endpoints of $\gamma$ are clear from context.  Composition gives natural identifications 
\begin{equation}\label{eq:ptComposition}  \ptrans{}{b}{c} \cdot \ptrans{}{a}{b} = \ptrans{}{a}{c}.\end{equation}  

We now proceed to calculate the covariant derivative of $\ptrans{}{a}{b}$.  The following proposition is certainly not new, but we could not find it in the literature.  When the path $\gamma$ is a closed loop, this formula appears as Proposition 2.1 of \cite{TWZ} as well as in \cite[Corollary 4.6]{SW}. Let $\bt$ denote the canonical vector field $\frac{\partial}{\partial t}$ on the interval $[0,1]$, and let $\nabla^{b\bar{a}} \= \nabla^b \otimes \nabla^{a \vee}$ be the natural connection on $E_b\otimes E_a^\vee \to PM$ induced by $\nabla^E$.

\begin{prop}\label{prop d hol}The covariant derivative of $\ptrans{}{a}{b}$ is given by the formula
\[ \nabla^{b\bar{a}} \left( \ptrans{}{a}{b} \right) = \int_{t=a}^{t=b} \ptrans{}{t}{b}\cdot \iota_{\bt}R(t) \cdot \ptrans{}{a}{t} \, dt \quad \in \quad \Omega^1(PM, E_b \otimes E_a^\vee). \]
In other words, for a vector  $X \in T(PM)_\gamma$, which is a vector field along the image of $\gamma$ in $M$, 
\[ \nabla^{b\bar{a}}_X \ptrans{\gamma}{a}{b} =  \int_{t=a}^{t=b} \ptrans{\gamma}{t}{b} \cdot R\left(\dot{\gamma}(t), X(t)\right) \cdot \ptrans{\gamma}{a}{t}\,\, dt \quad \in \quad  E_b \otimes E_a^\vee.\]
\end{prop}
\begin{proof}
There is a classical relationship between curvature and holonomy given by the following construction.  Suppose that $(E,\nabla)\to M$ is a vector bundle with connection, and we have a rectangle mapped to $M$, by which we mean some map $\Gamma: \mathbb{R}^2 \to M$ defined near $(0,0)$.  For $t_1, t_2 >0$, let $C(t_1,t_2)$ be the closed loop in $M$, with basepoint $\Gamma(0,0)=x_0 \in M$, defined by traveling counterclockwise along $\Gamma$ from $(0,0)$ to $(t_1,0)$, then to $(t_1,t_2)$, then $(0, t_2)$, and finally back to $(0,0)$.  The parallel transport along $C(t_1, t_2)$ gives an element $\ptrans{C(t_1,t_2)}{}{} \in \End(E_{x_0})$.  As $t_1, t_2 \to 0$, the curvature naturally appears as the quadratic term in the Taylor series expansion
\begin{equation}\label{eq:CurvatorTaylorFormula} \ptrans{C(t_1, t_2)}{}{} = 1 - t_1 t_2 R\left(\frac{\partial \Gamma}{\partial t_1},\frac{\partial \Gamma}{\partial t_2}\right) + O(3) . \end{equation}
See \cite[Proposition 3.3.14]{Nicolaescu} for full details in the case where $\Gamma$ is obtained from a coordinate system in $M$.  The slightly more general case described here follows by pulling back $(E,\nabla)\to M$ via $\Gamma$.  

Calculating the covariant derivative $\nabla^{b \bar{a}}(\ptrans{}{a}{b})$ involves the parallel transport induced by the connection $\nabla^b \otimes \nabla^{\bar{a}}$ on $E_b\otimes E_a^\vee \to PM$.  Consider a path in $PM$ as 1-parameter family of paths $\gamma_s$.  Let $\phi_a$ and $\phi_b$ denote the paths in $M$ obtained by restricting $\gamma_s(t)$ to the endpoints $t=a,b$.  Parallel transporting backwards along $\gamma_s$, 
\[ \Hom(E_{\gamma_s(a)}, E_{\gamma_s(b)}) =  E_{\gamma_s(b)}\otimes E_{\gamma_s(a)}^\vee \xrightarrow{\iso} E_{\gamma_0(b)}\otimes E_{\gamma_0(a)}^\vee = \Hom(E_{\gamma_0(a)}, E_{\gamma_0(b)}), \]
is simply pre-composition by $\ptrans{\gamma_s(a)}{}{}$ and post-composition by $\ptrans{\gamma_s(b)}{}{}^{-1}$.  For $\ptrans{\gamma}{a}{b}$, it's value at $\gamma_s$ translates back to 
\[ \ptrans{\phi_b}{s}{0} \cdot \ptrans{\gamma_s}{a}{b} \cdot \ptrans{\phi_a}{0}{s} \quad \in \quad  E_{\gamma(b)} \otimes E_{\gamma(a)}^\vee. \]
We let $\eta_s$ denote the path given by the concatenation of $\phi_a * \gamma_s * {\phi_b}^{-1}$, for a given parameter $s\geq 0$.  If $X \in TPM_\gamma$ is the vector field along $\gamma$ generated by $\gamma_s$, then $\nabla_X^{b \bar{a}} \ptrans{\gamma}{a}{b}$ is simply the derivative of the parallel transport of $\eta_s$ at $s=0$.  

We now decompose the path $\eta_s$ into a concatenated sequence of smaller paths.  This is illustrated in the following diagram, where $\eta_s$ is the path traveling the upper and outer two edges of the rectangle.
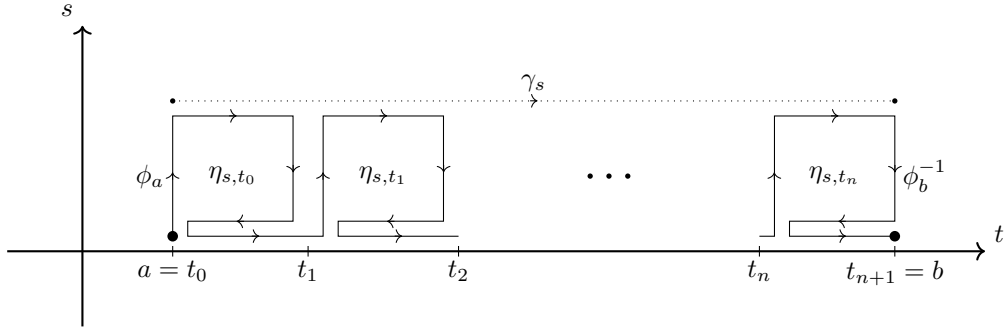
\begin{equation*}
\begin{tikzpicture}
\draw (1.2,0.2) to  coordinate[midway] (b)node{} (1.2,1.8) to
     coordinate[midway] (c)node{} (2.8,1.8) to
       coordinate[midway] (d)node{} (2.8,0.4) to 
        coordinate[midway] (e)node{} (1.4,0.4)  to  (1.4,0.2) to 
         coordinate[midway] (f)node{} (3.2,0.2) to 
          coordinate[midway] (g)node{} (3.2,1.8) to
           coordinate[midway] (h)node{} (4.8,1.8) to
             coordinate[midway] (i)node{} (4.8,0.4) to
               coordinate[midway] (j)node{} (3.4,0.4) to  (3.4,0.2) to
                 coordinate[midway] (k)node{} (5.0,0.2);
\draw (9.0,0.2) to (9.2, 0.2) to  coordinate[midway] (l)node{}  (9.2,1.8) to
  coordinate[midway] (m)node{} (10.8,1.8) to  coordinate[midway] (n)node{} (10.8,0.4) to
     coordinate[midway] (o)node{} (9.4,0.4) to (9.4,0.2) to 
       coordinate[midway] (p)node{} (10.8,0.2); 
\draw[dotted](1.2,2) to coordinate[midway](r)node{} (10.8,2);
\node[above] at (r) {$\gamma_s$};
\foreach \x in {d, i, n} {\node[rotate=-90] at (\x) {\midarrow};}
\foreach \x in {b, g, l} {\node[rotate=90] at (\x) {\midarrow};}
\foreach \x in {c, h, m, f, k, p, r} {\node at (\x) {\midarrow};}
\foreach \x in {e, j, o} {\node[rotate=180] at (\x) {\midarrow};}
\fill (1.2, .2)circle (2pt);
\fill (10.8,.2)circle (2pt);
\fill (6.75,1)circle (1pt);
\fill (7,1)circle (1pt);
\fill (7.25,1)circle (1pt);
\fill (1.2,2)circle(1pt);
\fill (10.8,2)circle(1pt);
\draw[thick,<-, ->] (0,-1) to (0,3);
\draw[thick, <-, ->] (-1,0) to (12,0);
\node[above left] at (0,3) {$s$};
\node[above right] at (12,0) {$t$};
\draw ($(1.2,0)+(0,2pt)$) -- ($(1.2,0)-(0,2pt)$);
\node[below] at (1.2,0) {$a=t_0$};
\draw ($(3,0)+(0,2pt)$) -- ($(3,0)-(0,2pt)$);
\node[below] at (3,0) {$t_1$};
\draw ($(5,0)+(0,2pt)$) -- ($(5,0)-(0,2pt)$);
\node[below] at (5,0) {$t_2$};
\draw ($(9,0)+(0,2pt)$) -- ($(9,0)-(0,2pt)$);
\node[below] at (9,0) {$t_n$};
\draw ($(10.8,0)+(0,2pt)$) -- ($(10.8,0)-(0,2pt)$);
\node[below] at (10.8,0) {$t_{n+1}=b$};
\node at (2,1) {$\eta_{s, t_0}$};
\node at (4,1) {$\eta_{s, t_1}$};
\node at (10,1) {$\eta_{s, t_n}$};
\node at (.9,1) {$\phi_a$};
\node at (11.2,1) {$\phi_b^{-1}$};
\end{tikzpicture} 
\end{equation*}
Partition the interval $[a,b]$ into smaller sub-intervals, 
\[ a=t_0 < t_1 < \cdots < t_n < t_{n+1}=b, \quad \Delta t = \Delta t_i = t_{i+1} - t_i.\]
(It is unimportant whether one uses constant $\Delta t$ or a varying mesh, so we will keep the notation as simple as possible.)  Let $\eta_{s, t_i}$ be the closed loop starting at $\gamma(t_i)$, going up in the $X(t_i)$ direction, along $\gamma_s$ from $t_i$ to $t_{i+1}$, back down along $X(t_{i+1})$, and then back along $\gamma$.  Then, using the composition identities \eqref{eq:ptComposition}, 
\begin{align*} \ptrans{\eta_s}{}{} &= \ptrans{\gamma}{t_n}{b} \cdot \ptrans{\eta_{s, t_n}}{}{} \cdot \ptrans{\gamma}{t_{n-1}}{t_n} \cdot  \ptrans{\eta_{s, t_{n-1}}}{}{} \cdots \ptrans{\gamma}{t_1}{t_2} \cdot \ptrans{\eta_{s, t_1}}{}{} \cdot \ptrans{\gamma}{a}{t_1} \cdot  \ptrans{\eta_{s,t_0}}{}{} \\
&=  \prod_{i=0}^n \ptrans{\gamma}{t_i}{t_{i+1}} \cdot \ptrans{\eta_{s, t_i}}{}{}.  \end{align*}

We can now use the Leibniz rule to calculate the derivative, and remember that the factors $\ptrans{\gamma}{t_i}{t_{i+1}}$ are constant with respect to $s$.  The parallel transport $\ptrans{\eta_{0,t_i}}{}{}=1$ since it is given by transporting back and forth along $\gamma$.  Also, note the curves $\eta_{s, t_i}$ are clockwise, hence the $\pm$ sign reversal in the following calculation.  We now proceed to calculate the desired derivative.
\begin{align*}
\nabla_X \ptrans{\gamma}{a}{b} &= \restr{\frac{d}{ds}}{s=0} \ptrans{\eta_s}{}{} =\restr{\frac{d}{ds}}{s=0} \left( \prod_{i=0}^n \ptrans{\gamma}{t_i}{t_{i+1}} \cdot \ptrans{\eta_{s, t_i}}{}{}  \right)  \\
&= \sum_{i=0}^n  \ptrans{\gamma}{t_n}{b} \cdot \ptrans{\eta_{0, t_n}}{}{} \cdot \ptrans{\gamma}{t_{n-1}}{t_n} \cdot \ptrans{\eta_{0, t_{n-1}}}{}{} \cdots \restr{\frac{d}{ds}}{s=0} \Big( \ptrans{\eta_{s,t_i}}{}{}  \Big) \cdots  \ptrans{\gamma}{a}{t_1} \cdot  \ptrans{\eta_{0,t_0}}{}{} \\
&= \sum_{i=0}^n  \ptrans{\gamma}{t_n}{b} \cdot 1 \cdot \ptrans{\gamma}{t_{n-1}}{t_n} \cdot 1  \cdots \restr{\frac{d}{ds}}{s=0} \Big( \ptrans{\eta_{s,t_i}}{}{}  \Big) \cdots  \ptrans{\gamma}{a}{t_1} \cdot  1 \\
&= \sum_{i=0}^n  \ptrans{\gamma}{t_i}{b} \cdot \restr{\frac{d}{ds}}{s=0} \Big( \frac{1}{\Delta t} \ptrans{\eta_{s,t_i}}{}{}  \Big) \cdot \ptrans{\gamma}{a}{t_i}  \,\, \Delta t \\
 &= \sum_{i=0}^n  \ptrans{\gamma}{t_i}{b} \cdot\restr{\frac{d}{ds}}{s=0} \left( \frac{1}{\Delta t}(1 + s \,\Delta t R(\dot \gamma, X) + \cdots) \right) \cdot \ptrans{\gamma}{a}{t_i} \,\, \Delta t  \\
&= \sum_{i=0}^n \ptrans{\gamma}{t_i}{b} \cdot  \Big( R\big(\dot \gamma(t_i), X(t_i)\big) + O(\Delta t) \Big) \cdot \ptrans{\gamma}{a}{t_i}  \, \, \Delta t 
\end{align*}
This calculation holds for all mesh, and taking the limit as $\Delta t \to 0$ gives the desired result
\[ \nabla_X \ptrans{\gamma}{a}{b} = \int_a^b \ptrans{\gamma}{t}{b} \cdot R\left(\dot \gamma(t), X(t)\right) \cdot \ptrans{\gamma}{a}{t}  \, \, dt .\]
Note this integral is defined since, for a given $\gamma$ and $X$, the $\ptrans{\gamma}{t}{b} \cdot R\left(\dot \gamma(t), X(t)\right) \cdot \ptrans{\gamma}{a}{t}$ are all elements of the same vector space $E_{\gamma(b)}\otimes E_{\gamma(a)}^{\vee} = \Hom(E_{\gamma(a)}, E_{\gamma(b)})$.
\end{proof}

\subsection{Defining the curved iterated integral}\label{subsec: define It}

First let us recall the usual iterated integral for the interval Hochschild complex, i.e. the two-sided bar complex.  Let $\underline{\omega }_n \in CH^I(\Omega(M))$ and parametrize the geometric $n$-simplex by 
\[ \Delta^n \= 
\{ (t_0,  \ldots, t_{n+1}) \in \mathbb{R}^{n+2} \ \vert \ 0 = t_0 \leq t_1 \leq \cdots \leq t_n \leq t_{n+1}=1\}.\]
We have an evaluation map 
\[ \Delta^n \times PM \xrightarrow{ev} M^{n+2} \]
given by $ev\big( (t_i)_{i=0}^{n+1},  \gamma \big) = \big(\gamma(t_i) {\big)}_{i=0}^{n+1}$, and we use the composition
 \begin{equation} \label{EQ: tensor forms map real}
 \Omega(M)^{\otimes n+2} \xhookrightarrow{} \Omega(M^{n+2}) \xrightarrow{ev^*} \Omega(\Delta^n \times PM) \xrightarrow{\int_{\Delta^n}} \Omega(PM),  \end{equation}
linearly extended, to define $It(\underline{\omega}) \= \int_{\Delta^n} ev^*(\underline{\omega})$.  Note that this map lowers the total degree of differential forms by $n$, so the induced map $CH^I(\Omega(M)) \xrightarrow{It} \Omega(PM)$ preserves degree by virtue of the degree shift in the definition of the 2-sided bar complex.  We will use a similar evaluation map to define the iterated integral map out of $ZZ(\A)$.  The construction places differential forms along a zigzag, where each column corresponds to a time-slot.  See the figure on page \pageref{general zigzag} for a visual guide to the below evaluation map.

\begin{definition}\label{DEF: ev_zz}
Given a manifold $M$ and a choice of $k$ and $n$, we define the \underline{zigzag evaluation map} 
\[ \Delta^n \times PM \xrightarrow{ev_{k,n}} M^{nk + k + 1}\]
by its components $ev_{(\ell, m)}$, where $(\ell, m) \in \left(  \{1, \ldots, k\} \times \{1, \ldots, n+1\}  \right)\cup \{ (0,0)\}$.  For each choice of $\underline{t}\= (t_0=0, t_1, \ldots, t_n, t_{n+1}=1) \in \Delta^n$ and $\gamma \in PM$, we define those components in cases: 
\begin{itemize}
\item (Initial point) When $m=\ell =0$, define $ev_{(0,0)}(\underline{t}, \gamma)\= \gamma(0)$.
\item (Zigs) When $\ell$ is odd, define $ev_{(\ell, m)}(\underline{t}, \gamma)\= \gamma(t_m).$  
\item (Zags) When $\ell>0$ is even, define $ev_{(\ell, m)}(\underline{t}, \gamma)\= \gamma(t_{n+1 - m}).$ 
\end{itemize}
\end{definition}

Next we provide the analog of the map above, $\Omega(M)^{\otimes n+2} \xhookrightarrow{} \Omega(M^{n+2})$, in the context of endomorphism bundle valued forms. Recall that if we have  bundles with connection $\left(\E_i, \nabla_i\right) \to M_i$ indexed over a finite set, there is a natural map of forms
\begin{equation}\label{eq:imap}
\bigotimes\limits_i \Omega(M_i , \E_i) \xhookrightarrow{\otimes_i \pi_i^*} \Omega \left( \prod_i M_i , \bigotimes\limits_i \pi_i^* \E_i \right),
\end{equation}
where $\bigotimes\limits_i \pi_i^* \E_i \to \prod\limits_i M_i $ is the tensor-product bundle with induced connection $\otimes_i \pi_i^*\nabla_i$.  At a point $\underline{x} = (x_i)_i \in \prod_i M_i$, the image of an element under \eqref{eq:imap} is given by the natural vector space isomorphisms
\[ \Big( \bigwedge T( \Pi_i M_i)^*  \otimes \bigotimes_i \pi_i^* \E_i \Big)_{\underline{x}} \iso \Big( \bigwedge \bigoplus_i TM^*_{i, x_i} \Big) \otimes \bigotimes_i \pi_i^* \E_{i,x_i}   \iso \bigotimes_i \left(\bigwedge TM_{i, x_i}^* \otimes \E_{i, x_i} \right). \]
Note  that we don't have to worry about $\pm$ signs when swapping the ordering of $\E_i$ with $\bigwedge TM_j^*$ because elements of $\E_i$ have degree 0.  Also note that this inclusion encodes the product on $\Omega(M,\E)$.  In particular, if $d: M \to M \times M$ is the diagonal map and $m:\E \otimes \E \to \E$ is the product structure on $\E$, viewed as a morphism of bundles over $M$, then the composition 
 \begin{equation}\label{eq:wedgeproduct} \Omega(M,\E) \otimes \Omega(M, \E) \xhookrightarrow{\pi_1^*\otimes \pi_2^*} \Omega(M\times M, \pi_1^*\E \otimes \pi_2^* \E) \xrightarrow{d^*} \Omega(M, \E \otimes \E) \xrightarrow{m} \Omega(M,\E)  \end{equation}
is the algebra structure on the curved dga $\Omega(M,\E)$.  When $\E=\mathbb{R}$ is the trivial bundle, this is composition is simply the usual wedge product.

Our desired evaluation-pullback map is given by composing the inclusion \eqref{eq:imap} with our zigzag evaluation map from Definition \ref{DEF: ev_zz}.  We obtain
\begin{equation}\label{EQ:bundformstensor}
EV_{k, n}^*: \Omega(M, \E)^{\otimes (nk+k+1)} \xhookrightarrow{\otimes_i \pi_i^*} \Omega\Big(M^{(nk+k+1)},\otimes_i \pi_i^* \E \Big) \xrightarrow{ev_{k,n}^*} \Omega\Big(\Delta^n \times PM, EV_{k,n}^*\E  \Big),
\end{equation}
and we use the abbreviation $EV_{k,n}^* \= ev_{k,n}^* \o \otimes_i \pi_i^*$ for simplicity.

When working with bundle-valued forms $\Omega(M, \E)$, we must be careful about multiplying forms that have been evaluated at different time-values, as is happening in \eqref{EQ:bundformstensor}. At a point $(\underline{t}, \gamma) \in \Delta^n \times PM$, the fiber of $EV_{k,n}^* \E$ is the tensor product of fibers
\begin{align*} \left(EV_{k,n}^* \E \right)_{(\underline{t}, \gamma)} = \E_{\gamma(0)} 
& \otimes \E_{\gamma(t_1)}\otimes \E_{\gamma(t_2)} \otimes \dots \otimes \E_{\gamma(t_n)}   \otimes \E_{\gamma(1)}\\
& \otimes \E_{\gamma(t_n)}  \otimes \E_{\gamma(t_{n-1})}  \otimes \dots\otimes \E_{\gamma(t_1)}  \otimes \E_{\gamma(0)}\\
&\otimes \dots  \otimes \E_{\gamma(0)}.
\end{align*}
Though each fiber $\E_x$ has an algebra structure, different fibers are distinct algebras, so there is no way a priori to multiply the tensor factors in $EV_{k,n}^*\E$.  Fortunately we have a connection; the induced parallel transport along $\gamma$ provides natural isomorphisms between these different fibers.  The connection $\nabla$ on $\E$ is induced by a connection $\nabla^E$ on $E$, and the parallel transport identifications $\E_0 \iso \E_t$ from $\nabla$ are obtained by conjugating with the parallel transport $\ptrans{}{0}{t}$ described in Subsection \ref{subsec:ParallelTransport}.  We denote the resulting isomorphisms by $Ad_P$, as indicated in the following diagram.
\begin{equation}\label{EQ: Ad_P_s}
  \begin{tikzcd}[row sep=-.5em, column sep=small]
 (Ad_P)_t\colon &
\Omega(PM, \E_t) \ar[r] &  \Omega\Big(PM, (E_0\otimes E_t^\vee) \otimes (E_t\otimes E_t^\vee) \otimes (E_t \otimes E_0^\vee)\Big) \ar[r]& \Omega(PM, \E_0) \\
& \omega \ar[rr, mapsto] && \ptrans{}{t}{0} \cdot \omega \cdot \ptrans{}{0}{t}
\end{tikzcd} \end{equation}
Applying the parallel transport to each factor and then taking the product (using the algebra structure on $\E$)  gives a map of bundle-valued forms that we denote $(Ad_P)_*: EV_{k,n}^*\E \to \E_0$.  When combined with \eqref{EQ:bundformstensor}, we obtain
\begin{equation}\label{EQ: Ad_P_*}
\Omega(M, \E)^{\otimes (nk+k+1)} \xrightarrow{EV_{k,n}^*} \Omega\left(\Delta^n \times PM, EV_{k,n}^*\E \right)   \xrightarrow{(Ad_P)_*}
\Omega\left(\Delta^n \times PM, \E_0 \right).
\end{equation} 
Note that when we write $\E_0$ as a bundle over $\Delta^n \times PM$, we really mean the pullback of $\E_0\to PM$ along the projection map $\Delta^n\times PM \to PM$.  (We believe that not introducing a separate projection symbol here is the option least likely to cause  confusion.)   By Proposition \ref{prop:Integration}, the usual integration over fibers map extends to one with $\E_0$ coefficients,
\[ \int\limits_{\Delta^n} \colon \Omega^{\bullet}(\Delta^n \times PM, \E_0) \longrightarrow \Omega^{\bullet-n}(PM, \E_0).\]
Finally, our curved zigzag Chen iterated integral map is obtained by composing these maps.

\begin{definition}\label{zigzag Chen}Let $\A = \Omega_{\nabla}(M,\E)$, where $(E,\nabla^E)\to M$ is a vector bundle with connection.  
The \underline{curved zigzag Chen map}, $It: ZZ\left( \Omega_{\nabla}(M,\E) \right) \to \Omega_{\wt{\nabla}}(PM, \E_0)$, is defined on each monomial  $\underline{\omega }_{(k,n)} \in (\A \otimes(\A^{\otimes n}\otimes \A)^{\otimes k})[n]$ by the composition
 \[\Omega^{\bullet}(M, \E)^{\otimes nk + k+1} \xrightarrow{(Ad_P)_* \circ EV_{k,n}^*} \Omega^{\bullet}(\Delta^n \times PM, \E_0) \xrightarrow{\int_{\Delta^n}} \Omega^{\bullet-n}(PM, \E_0),\]
and then extended linearly to all of $ZZ(\A)$. 
 \end{definition}
 
Note that this  integral is degree-preserving because of the grading convention for $ZZ(\A)$ in Definition \ref{def CH^ZZ}: a monomial in the domain $\Omega^{\bullet}(M, \E)^{\otimes nk + k+1}$ has it's degree shifted down by $n$.  We will now proceed to further unpack this definition.  Definition \ref{zigzag Chen} defines the curved Chen map by 
 \begin{equation}\label{EQ IT^A sh}
It( \underline{\omega}) =  \int\limits_{\Delta^n} \left((Ad_P)_* \circ EV_{k,n}^*\right)(\underline{\omega}_{k,n}) .
 \end{equation}
When we use the abbreviation $\omega(t) = ev_t^*\omega$ to denote the form $\omega$ being evaluated at time $t$ on $PM$, equation \eqref{EQ IT^A sh} can be rewritten
\begin{align*}
It( \underline{\omega})=\int\limits_{\Delta^n} &\omega_{(0,0)}(t_0) \cdot \left((Ad_P)_{t_1}\omega_{(1,1)}(t_1) \right) \cdot   \left((Ad_P)_{t_2} \omega_{(1,2)}(t_2)\right)    \cdots \left( (Ad_P)_{t_{n+1}} \omega_{(1,n+1)}(t_{n+1}) \right) \\
&\cdot \left( (Ad_P)_{t_n}  \omega_{(2,1)}(t_{n}) \right)  \cdots  \left( (Ad_P)_{t_1}  \omega_{(2,n)}(t_{1}) \right) \cdot \omega_{(2, n+1)} (t_0) \cdots  \omega_{(k,n+1)}(t_0).
\end{align*}
As a reminder, the endpoint time values $t_0=0$ and $t_{n+1}=1$ cannot vary. Each of the forms is conjugated by parallel transport between some time value $t$ and $t=0$, and much of this conjugation cancels.  For example, 
\begin{align*}  \left((Ad_P)_{t_1}\omega_1 (t_1) \right) \cdot  \left(  (Ad_P)_{t_2} \omega_2(t_2) \right) &=  \left( \ptrans{}{t_1}{0} \cdot \omega_1(t_1) \cdot \ptrans{}{0}{t_1} \right) \cdot \left( \ptrans{}{t_2}{0} \cdot \omega_2(t_2) \cdot  \ptrans{}{0}{t_2} \right) \\
&=  \ptrans{}{t_1}{0} \cdot \omega_1(t_1) \cdot   \ptrans{}{t_2}{t_1} \cdot  \omega_2(t_2)   \cdot \ptrans{}{0}{t_2}.  \end{align*}
Therefore, the curved iterated integral formula can be rewritten 
\begin{align}
It( \underline{\omega})=  \int\limits_{\Delta^n} &\omega_{(0,0)}(0) \cdot \ptrans{}{t_1}{0} \cdot \omega_{(1,1)}(t_1) \cdot  \ptrans{}{t_2}{t_1} \cdot \omega_{(1,2)}(t_2)  \cdot \ptrans{}{t_3}{t_2} \cdots  \ptrans{}{1}{t_n}\cdot \omega_{(1,n+1)}(1) \label{EQ IT^A}\\
&\cdot  \ptrans{}{t_n}{1} \cdot \omega_{(2,1)}(t_{n}) \cdot \ptrans{}{t_{n-1}}{t_n}   \cdots \ptrans{}{t_1}{t_2} \cdot \omega_{(2,n)} (t_1)  \cdot \ptrans{}{0}{t_1} \cdot \omega_{(2, n+1)} (0) \cdots  \ptrans{}{0}{t_1} \cdot \omega_{(k,n+1)}(0).\nonumber
\end{align}

Proving that $It$ is compatible with the respective differentials requires a few intermediate propositions, but the overall idea is as follows.  We use the Stokes formula \eqref{eq:BundleStokes} to take the covariant derivative of \eqref{EQ IT^A}.  Up to a $\pm$ sign, this is given by the following two terms
\[ \wt{\nabla} \left( It(\underline{\omega}) \right) = \int\limits_{\Delta^n} \nabla \Big( \omega_{(0,0)}(t_0) \cdot \ptrans{}{t_1}{0} \cdots  \omega_{(k,n+1)}(0) \Big)  \quad - \quad \int\limits_{\partial \Delta^n}  \omega_{(0,0)}(t_0) \cdot \ptrans{}{t_1}{0} \cdots\omega_{(k,n+1)}(0).  \]
The second term, the boundary term, is given by composing $It$ with the $b_z$ component of the zigzag differential $D_z$.  The first term decomposes into a large sum via the Leibniz rule.  The $\nabla_z$ component of $D_z$ accounts for when $\nabla$ is applied to the $\omega$ factors, and the $c_z$ component accounts for the terms with $\nabla( \ptrans{}{t_\ell}{t_{\ell\pm 1}})$. We now begin the process of proving the curved iterated integral is a morphism in $\NGA$.

\begin{prop}The curved iterated integral $It$, as defined on monomials in \eqref{EQ IT^A sh}, is well-defined on the quotient $ZZ(\Omega(M,\E))$.
\begin{proof}Using the abbreviation $\A = \Omega(M,\E)$, we defined the iterated integral for monomials $\underline{\omega}_{k,n} \in (\A \otimes(\A^{\otimes n}\otimes \A)^{\otimes k})[n]$.  However, $ZZ(\A)$ was defined as a quotient in Definition \ref{def CH^ZZ}.  We quickly show that monomials related by the equivalences $\sim_{a}$ and $\sim_b$ (see page \pageref{page zigzag relations}) have the same image under $It$.  

The relation $\sim_a$ relates a monomial $\underline{\omega}$ to the monomial $ins_{\ell}\left(\underline{\omega}\right)$, which is obtained by inserting an entire zigzag of units $1$ (identity-valued $0$-forms) at some level $\ell$.  It easily follows that $((Ad_P)_* \o EV^*)(\underline{\omega}) = ((Ad_P)_* \o EV^*)(ins_{\ell}\left(\underline{\omega}\right))$.  More explicitly, the portion of the integrands that differ are merely a composition of parallel transports and units, and the composition of these parallel transports is the identity:
\begin{align*}
It(ins_{\ell}\left(\underline{\omega}\right)) &= \int\limits_{\Delta^n} \cdots \omega_{(\ell-1, n)}(t_1)\cdot \ptrans{}{0}{t_1} \cdot 1 \cdot \ptrans{}{t_1}{0} \cdot 1 \cdot \ptrans{}{t_2}{t_1} \cdot 1 \cdots 1\cdot \ptrans{}{0}{t_1} \cdot \omega{(\ell-1,n+1)}(0) \cdots  \\
&= \int\limits_{\Delta^n} \cdots \omega_{(\ell-1, n)}(t_1)\cdot \ptrans{}{0}{t_1} \cdot \omega_{(\ell-1,n+1)}(0) \cdots  = It(\underline{\omega}).
\end{align*}

Similarly, the relation $\sim_b$ relates the zigzag $\underline{\omega}$ to $\underline{\omega'}$ by commuting forms across units, but not changing columns. Since the columns are unchanged, the relevant forms are still evaluated at the same time values and multiplied in the same order.  If one writes the two relevant integrands as we did for the case of $\sim_a$, one easily sees that the portion where they could potentially differ is a sequence of parallel transports that composes to the identity.
\end{proof}
\end{prop}

An important part of showing that the curved zizgzag Chen map of Definition \ref{zigzag Chen} commutes with the differentials is to prove that the boundary term in our (generalized) Stoke's formula from \eqref{eq:BundleStokes} is produced by the $b_z$ operator in the zigzag algebra.  This holds for essentially the same reasons as in Chen's original iterated integral, only we must be a little more careful about where various items live.

\begin{prop}\label{Prop:b and d commute}For $\underline{\omega} = \underline{\omega}_{k,n}$,
\[ It\left( b_z (\underline{\omega}) \right) =  (-1)^{n-1} \int\limits_{\partial \Delta^n} \left((Ad_P)_* \circ EV_{k,n}^*\right)(\underline{\omega}) . \]
\begin{proof}
Writing the operator as $b_z = \sum_{\ell=0}^n b_{\ell}$, we will show that each component $b_{\ell}$ corresponds (up to a $\pm$ sign) to the $\ell$-th face map $d_\ell$.  More precisely, let $d_{\ell}: \Delta^{n-1} \to \Delta^n$ be the face map on the geometric simplex given by $(t_1, \ldots , t_{\ell}, \ldots t_{n-1}) \mapsto (t_1, \ldots , t_{\ell}, t_{\ell}, \ldots t_{n-1})$.  In a slight abuse of notation, let $d_{\ell}$ also denote the induced inclusion $d_{\ell}: M^{(n-1)k + k + 1} \to M^{nk+k+1}$ that makes the following diagram commute.
\begin{equation}\label{eq:comm diag for dl}
\begin{gathered}
\xymatrix{\Delta^{n-1} \times PM \ar[d]_{d_{\ell} \times id} \ar[r]^{ev_{k, n-1}} & M^{(n-1)k + k+1 } \ar[d]^{d_{\ell}} \\
 \Delta^n \times PM  \ar[r]^{ev_{k,n}} & M^{nk+k+1}   }
 \end{gathered}\end{equation}
Essentially, the right-vertical map $d_{\ell}$ behaves as a product of identity and diagonal inclusion maps on $M$, where the diagonal inclusions occur in time slots $t_{\ell}$.  

When the bundle $\otimes_{i=1}^{nk+k+1}\E \to M^{nk+k+1}$ is pulled back to $M^{(n-1)k+k+1}$ via $d_{\ell}^*$ it will again be a large tensor product, where some factors come from pulling back the tensor product bundle $\E \otimes\E \to M$.  This is easiest to describe after pulling back along $ev_{k,n-1}$, where the tensor product factors appear in time slots $t_{\ell}$:
\begin{align} \left(ev_{k,n-1}^* d_{\ell}^* (\otimes_i \pi_i^* \E) \right)_{(\underline{t}, \gamma)} = \E_{\gamma(0)} 
& \otimes \E_{\gamma(t_1)}\otimes \cdots \otimes ( \E\otimes \E)_{\gamma(t_{\ell})} \otimes \dots \otimes \E_{\gamma(t_{n-1})}   \otimes \E_{\gamma(1)} \nonumber \\
& \otimes \E_{\gamma(t_{n-1})}  \otimes \cdots \otimes (\E \otimes \E)_{\gamma(t_{\ell})}  \otimes \dots\otimes \E_{\gamma(t_1)}  \otimes \E_{\gamma(0)} \label{eq:domain for dl}\\
&\otimes \dots  \otimes \E_{\gamma(0)} .\nonumber
\end{align}
Letting $m: \E \otimes \E \to \E$ denote the map of bundles over $M$ that encodes the algebra structure on $\E$,  define $m_{\ell}: d_{\ell}^*(\bigotimes_{i=1}^{nk + k + 1} \pi_i^* \E) \to \bigotimes_{i=1}^{(n-1)k + k + 1} \pi_i^* \E$ to be the bundle map over $M^{(n-1)k + k + 1}$ given by the multiplication $m$ on each of the tensor factors  in \eqref{eq:domain for dl} that is itself a tensor product.  As described in \eqref{eq:wedgeproduct}, the composition $m \o d^* \o (\pi_1^*\otimes \pi_2^*)$ is simply the product $\Omega(M,\E) \otimes \Omega(M,\E) \to \Omega(M,\E)$.  Therefore, the composition $m_\ell \o d_\ell^* \o (\otimes_i \pi_i^* \E)$ is equivalent to the product on the forms in $\Omega(M,\E)^{nk+k+1}$ on the ``time $t_\ell$ columns.'' 

To state this more precisely, the left-hand square in the following diagram commutes.
 \begin{equation}\label{eq:multi-diagonal}
\begin{tikzcd}
{\Omega (M, \E)^{\otimes nk +k + 1}} \arrow[r, hook] \arrow[d, "(-1)^{n+\ell}b_\ell"] & {\Omega(M^{ nk +k + 1}, \otimes_i\pi_i^* \E)} \arrow[r, "{(Ad_P)_* \circ ev_{k,n}^*}"] \arrow[d, "m_\ell \o d_\ell^*"] &[5em]  {\Omega(\Delta^n \times PM, \E_0)} \arrow[d, "(d_\ell \times id)^*"] \\
{\Omega(M, \E)^{\otimes (n-1)k +k + 1}} \arrow[r, hook] & {\Omega(M^{ (n-1)k +k + 1}, \otimes_i \pi_i^* \E)} \arrow[r, "{(Ad_P)_* \circ ev_{k,n-1}^*}"] &{\Omega(\Delta^{n-1} \times PM, \E_0)}
\end{tikzcd}
\end{equation}
The factor $(-1)^{n+\ell}$ is due to its appearance in the definition of $b_\ell$  in \eqref{EQ b_z}.  The right-hand square of \eqref{eq:multi-diagonal} also commutes; the commutative diagram \eqref{eq:multi-diagonal} implies this for $\Omega(M)$, and the result extends to $\Omega(M,\E)$ after using $(Ad_P)_*$ to take the total product on all factors of $\E$.  In summary, our commutative diagram \eqref{eq:multi-diagonal} may be expressed by the equation
\begin{equation} \Big( (Ad_P)_* \o EV_{k,n-1}^* \Big) \big(b_\ell (\underline{\omega})\big) =(-1)^{n+\ell} \Big(  (d_\ell \times id_{PM})^* \o (Ad_P)_* \o EV^*_{k,n} \Big)(\underline{\omega}). \end{equation}

We may now proceed to the calculation of our curved iterated integral when it is restricted to the boundary of an $n$-simplex.  Note that the boundary $\partial \Delta^{n}$ is spanned by the image of the face maps $d_{\ell}: \Delta^{n-1} \hookrightarrow \partial \Delta^n$.  This embedding is orientation-preserving when $\ell$ is odd, so we pick up a factor of $(-1)^{\ell +1}$ in the following calculation. 
\begin{align*}
 (-1)^{n-1} \int\limits_{\partial \Delta^n} &((Ad_P)_* \circ EV_{k,n}^*)(\underline{\omega})  =   (-1)^{n-1}  \sum\limits_\ell (-1)^{\ell+1} \int\limits_{\Delta^{n-1}} (d_\ell \times id)^*((Ad_P)_* \circ EV_{k,n}^*)(\underline{\omega})    \\
 = & \sum\limits_\ell   \int\limits_{\Delta^{n-1}} ((Ad_P)_* \circ EV_{k, n-1}^*)(b_\ell (\underline{\omega}) ) 
 =  \int\limits_{\Delta^{n-1}} \big((Ad_P)_* \circ EV_{k,n-1}^*)(b_z (\underline{\omega}) \big) = It\left( b_z (\underline{\omega}) \right).\end{align*}
\end{proof}\end{prop}

Now that we've established that the $b_z$ operator produces the boundary term in $\wt{\nabla}(It(\underline{\omega}))$, we show that the $c_z$ operator produces the terms inside $\int_{\Delta}\nabla ( (Ad_P)_* \o EV_{k,n}^*)(\underline{\omega})$ that arise from covariant derivatives of  parallel transports.
 
\begin{prop}\label{prop:R terms in It}Under the curved iterated integral $It$, the $c_{j,\ell}$ component of the $c_z$ operator gives the covariant derivative of the parallel transport between $t_{\ell-1}$ and $t_\ell$ on the $j$th zig/zag, with the precise formula stated in \eqref{eq:It of c}.
\begin{proof}
We perform the calculation for $c_{1, \ell}$ and note that the calculations for other $c_{j, \ell}$ follow in the exact same manner.  As a reminder to the reader, up to a $\pm$ sign the operator $c_{1, \ell}$ inserts a new column in between the original $(\ell-1)$ and $\ell$ columns, where $1 \leq \ell \leq n + 1$.  In this new column, the curvature $R$ is placed on the first zig, and the identity $1$ is placed on all other zigs/zags.  We will temporarily suppress the $\pm$ sign by using the following description. 
\begin{align*}
c_{1,\ell}( \underline{\omega}_{k,n}) = 
\pm \;   \omega_{(0,0)}&\otimes \cdots \otimes \omega_{(1,\ell-1)} \otimes R \otimes \omega_{(1,\ell)} \otimes \cdots \otimes \omega_{(1,n)} \otimes \omega_{(1,n+1)}  \otimes\\
   &\otimes \omega_{(2,1)} \otimes \ldots \otimes \omega_{(2, n-\ell+1)} \otimes 1 \otimes \omega_{(2, n-\ell+2)} \otimes \cdots \otimes \omega_{(2,n+1)} \otimes \cdots  
 \end{align*} 

We now proceed to our calculation.  
\begin{align*}
It&(c_{1,\ell}( \underline{\omega}_{k,n}))  \\
=&\pm  \int\limits_{\Delta^{n+1}}  \omega_{(0,0)}(t_0) \cdot \ptrans{}{t_1}{0} \cdots \omega_{(1,\ell-1)}(t_{\ell-1}) \cdot \ptrans{}{t_{\ell}}{t_{\ell-1}}\cdot R(t_{\ell}) \cdot \ptrans{}{t_{\ell+1}}{t_{\ell}}  \cdot \omega_{(1,\ell)}(t_{\ell+1})  \cdots \ptrans{}{1}{t_{n+1}}\cdot \omega_{(1,n+1)}(t_{n+2}) \\
&  \quad \quad \quad \cdot  \ptrans{}{t_{n+1}}{1} \cdot \omega_{(2,1)}(t_{n+1})   \cdots \omega_{(2,n-\ell+1)}(t_{\ell+1}) \cdot   \ptrans{}{t_{\ell}}{t_{\ell+1}} \cdot 1(t_{\ell}) \cdot \ptrans{}{t_{\ell-1}}{t_{\ell}} \cdot \omega_{(2,n-\ell+2)} (t_{\ell-1})  \cdots \\
=  &  \pm \int\limits_{\Delta^{n}} \omega_{(0,0)}(t_0) \cdot \ptrans{}{t_1}{0} \cdots \omega_{(1,\ell-1)}(t_{\ell-1}) \cdot   \left(  \int_{t=t_{\ell-1}}^{t=t_{\ell}} \ptrans{}{t}{t_{\ell-1}}\cdot \iota_{\bt} R(t) \cdot \ptrans{}{t_{\ell}}{t} dt \right)  \cdot \omega_{(1,\ell)}(t_{\ell})  \cdots \ptrans{}{1}{t_{n}}\cdot \omega_{(1,n+1)}(t_{n+1}) \\
& \quad \quad \quad  \cdot  \ptrans{}{t_{n}}{1} \cdot \omega_{(2,1)}(t_{n})   \cdots \omega_{(2,n-\ell+1)}(t_{\ell}) \cdot    \ptrans{}{t_{\ell-1}}{t_{\ell}}   \cdot \omega_{(2,n-\ell+2)} (t_{\ell-1})  \cdots \\
= & \mp  \int\limits_{\Delta^n}  \omega_{(0,0)}(t_0)\cdots  \omega_{(1,\ell-1)}(t_{\ell-1}) \cdot  \nabla \left( \ptrans{}{t_{\ell}}{t_{\ell-1}}\right)   \cdot \omega_{(1,\ell)}(t_{\ell}) \cdots  \omega_{(k,n+1)}(t_0) 
\end{align*}
The first equality follows directly by writing the iterated integral as we previously did in \eqref{EQ IT^A}.  To see the second equality, first note that $\ptrans{}{t_{\ell}}{t_{\ell-1}}\cdot R(t_{\ell}) \cdot \ptrans{}{t_{\ell+1}}{t_{\ell}}$ is the only part that is not constant with respect to $t_{\ell}$.  Therefore, it is the only place where a $dt_{\ell}$ can occur, making it easy to integrate out the original $t_{\ell}$ variable.  In doing this, we relabel our $t_i$ variables, so that the new $t_i$ are the prior $t_{i+1}$ for $i\geq \ell$.  Finally, we use the calculation of $\nabla(\ptrans{}{t_\ell}{t_{\ell -1}})$ from Proposition \ref{prop d hol} to obtain the third equality.  A factor of $(-1)$  appears in this step because we must swap the integration bounds $t=t_{\ell-1}$ and $t=t_\ell$.  

One can calculate $c_{j,\ell}(\underline{\omega})$ by the exact same method.  The only difference is that the $(-1)$ factor appearing in the last step only appears when $j$ is odd, because we do not need to swap the bounds to calculate the desired $\nabla(\ptrans{}{t_{\ell -1}}{t_\ell})$ appearing when $j$ is even.

Finally, if we include the correct $\pm$ signs used in defining $c_z$, we see that
\begin{equation}\label{eq:It of c} It\left( c_{j, \ell}(\underline{\omega}) \right) = (-1)^{n+\epsilon(\underline{\omega}, j, \ell)} \begin{cases}   \int\limits_{\Delta^n} \cdots \omega_{(j, \ell-1)}(t_{\ell-1}) \cdot \nabla\left(\ptrans{}{t_\ell}{t_{\ell-1}} \right) \cdot \omega_{(j, \ell)}(t_\ell) \cdots & (j \text{ odd}) \\
 \int\limits_{\Delta^n} \cdots \omega_{(j, n-\ell+1)}(t_{\ell}) \cdot \nabla\left(\ptrans{}{t_{\ell-1}}{t_{\ell}} \right) \cdot \omega_{(j, n- \ell+2)}(t_{\ell-1}) \cdots & (j \text{ even}),
\end{cases}  \end{equation}
where $\epsilon(\underline{\omega}, j, \ell)$ is simply the sum of the degrees of all forms appearing to the left of $\nabla P$ in the iterated integral.  
\end{proof}
\end{prop}

We now have all the pieces necessary to prove that our curved iterated integral is compatible with the respective differentials.  

\begin{prop}\label{prop zz chain map}
The Iterated Integral $It: ZZ(\Omega_\nabla(M,\E)) \to \Omega_{\wt{\nabla}}(PM, \E_0)$ satisfies $It \circ D_z = \widetilde{\nabla} \circ It$.
\begin{proof}
Recall that the curved differential graded algebras $ZZ(\Omega_{\nabla}(M,\E))$ and  $\Omega_{\wt{\nabla}}(PM,\E_0)$ have differentials $D_z\= (\nabla_z + b_z + c_z) $ (Definition \ref{def CH^ZZ}) and $\widetilde{\nabla} \= ev_0^*\nabla$ (Example \ref{EX: Omega(PM)}), respectively.  We wish to show that, for an arbitrary monomial $\underline{\omega}=\underline{\omega}_{k,n}$, we have $\widetilde{\nabla}\left( It( \underline{\omega}) \right) =  It( D_z(\underline{\omega}))$.  In the following, the Stokes formula \eqref{eq:BundleStokes} gives us first equality, and Proposition \ref{Prop:b and d commute} allows us to rewrite the boundary term using the $b_z$ differential.
\begin{equation} \label{EQ: zz chain 1}\begin{aligned}
\widetilde{\nabla}\big( It( \underline{\omega}) \big) &=   (-1)^n \int\limits_{\Delta^n} \widetilde{\nabla} \Big( (Ad_P)_* \circ EV_{k,n}^*)(\underline{\omega}) \Big)+ (-1)^{n-1} \int\limits_{\partial \Delta^n} ((Ad_P)_* \circ EV_{k,n}^*)(\underline{\omega}) \\
&= (-1)^n \int\limits_{\Delta^n} \widetilde{\nabla} \Big( (Ad_P)_* \circ EV_{k,n}^*)(\underline{\omega}) \Big) \quad +   \quad \int\limits_{\Delta^{n-1}} ((Ad_P)_* \circ EV_{k,n-1}^*)(b_z(\underline{\omega})) 
\end{aligned}\end{equation}

We now proceed to calculate the first term $\int_{\Delta^n} \widetilde{\nabla} ((Ad_P)_* \circ EV_{k,n}^*)(\underline{\omega})$ of \eqref{EQ: zz chain 1}.  Using \eqref{EQ IT^A}, this can be rewritten explicitly as
\begin{align*}
(-1)^n \int\limits_{\Delta^n} \wt{\nabla}\bigg(&\omega_{(0,0)}(0) \cdot \ptrans{}{t_1}{0} \cdot \omega_{(1,1)}(t_1) \cdot  \ptrans{}{t_2}{t_1} \cdot \omega_{(1,2)}(t_2)  \cdot \ptrans{}{t_3}{t_2} \cdots  \ptrans{}{1}{t_n}\cdot \omega_{(1,n+1)}(1) \\
&\cdot  \ptrans{}{t_n}{1} \cdot \omega_{(2,1)}(t_{n}) \cdot \ptrans{}{t_{n-1}}{t_n}   \cdots \ptrans{}{t_1}{t_2} \cdot \omega_{(2,n)} (t_1)  \cdot \ptrans{}{0}{t_1} \cdot \omega_{(2, n+1)} (0) \cdots  \ptrans{}{0}{t_1} \cdot \omega_{(k,n+1)}(0) \bigg). 
\end{align*}
This will become a sum of terms, since the operator $\wt{\nabla}$ is a derivation, where $\nabla$ is applied either to some $\omega_{(j, \ell)}$ or some $\ptrans{}{t_\ell}{t_{\ell \pm 1}}$.  The terms involving $\nabla \omega_{(j, \ell)}$ clearly sum to  $\int_{\Delta^n} ( (Ad_P)_* \circ EV_{k,n}^*) \left(\nabla_z (\underline{\omega}) \right)$.  Finally, Proposition \ref{prop d hol} allows us to calculate $\nabla(\ptrans{}{t_i}{t_{i\pm1}})$, and in Proposition \ref{prop:R terms in It} we showed how to rewrite this part of the integral.  The terms involving $\nabla(\ptrans{}{t_{\ell+1}}{t_{\ell}})$ in the zigs (or $\nabla(\ptrans{}{t_{\ell}}{t_{\ell+1}})$ in the zags) each equal the iterated integral of the monomial $\underline{\omega}$ with $R$ inserted as a new column in that same position.  The sum of these such terms is precisely the definition of $c_z(\omega)$.  Therefore, 
\begin{equation}\label{eq:zz chain final} \widetilde{\nabla}\big( It( \underline{\omega}) \big) =  \int\limits_{\Delta^n} \big( (Ad_P)_* \circ EV_{k,n}^* \big)\big((\nabla_z + c_z)(\underline{\omega}) \big)  +    \int\limits_{\Delta^{n-1}} ((Ad_P)_* \circ EV_{k,n-1}^*)(b_z(\underline{\omega})) = It\big( D_z(\underline{\omega}) \big).\end{equation}
Finally, note that the definition of $D_z$ included factors of $(-1)^n$, which is the reason the $\pm$ signs in  \eqref{EQ: zz chain 1} all become $+$ in \eqref{eq:zz chain final}.
\end{proof}
\end{prop}

\begin{remark}
One could consider an adjusted two-sided bar construction whose differential includes a curvature component, and use an iterated integral map, $It: CH^I(\Omega_{\nabla}(M,\E)) \to (\Omega_{\nabla}(PM,ev_0^*\E)$,  similar to the one in this paper.  By a calculation analogous to that in the proof of Proposition \ref{prop zz chain map}, this map would be a chain map largely due to the differential of $CH^I(\Omega(M))$ never requiring any forms to commute. However, as we have tried to make clear in this paper, and illustrated on page \pageref{intro calculation}, this map would not respect the product structures. Since, additionally, $CH^I(\Omega_{\nabla}(M,\E))$ does not form a curved dga, such a map would not be a map of algebras, let alone curved differential graded algebras.
\end{remark}

\begin{prop}\label{zz It algebra map}
The iterated integral $It: ZZ(\Omega_{\nabla}(M,\E)) \to \Omega_{\wt{\nabla}}(PM, \E_0)$ is a map of algebras, and hence it is a morphism in $\NGA$.
\begin{proof}
First, recall that the product of standard simplices is a union of inclusions
 \[\Delta^m \times \Delta^n = \bigcup\limits_{\sigma \in S_{m,n}} \beta^{\sigma}\left( \Delta^{m+n}\right),\]
 where the sum is over all $(m,n)$-shuffles $\sigma$ (see Remark \ref{shuffle remarks}) and $\beta^{\sigma} : \Delta^{m+n} \to \Delta^m \times \Delta^n$ is the unshuffle map \[(t_1, \ldots , t_{m+n}) \xmapsto{\beta^{\sigma}} ((t_{\sigma(1)}, \ldots , t_{\sigma(m)}), (t_{\sigma_{(m+1)}}, \ldots , t_{\sigma_{(m+n)}})).\]
 Since the intersection of these inclusions have measure zero in $\Delta^{m+n}$, then for any bundle $(E \to \Delta^m \times \Delta^n \times X)$ we have
 \begin{equation}\label{EQ:delta n delta m shuff}
  \int\limits_{\Delta^m \times \Delta^n} \alpha =  \int\limits_{\coprod\limits_{\sigma \in S_{m,n}} \beta^{\sigma} \left( \Delta^{m+n} \right)} (-1)^{\abs{\sigma}} \alpha = \sum\limits_{\sigma} (-1)^{\abs{\sigma}} \int\limits_{\Delta^{m+n}} \left( \beta^{\sigma}\right)^* (\alpha) 
  \end{equation}
for any form $\alpha \in \Omega(\Delta^m \times \Delta^n \times X, E)$, where $\abs{\sigma}$ is the parity of the shuffle.
Next, we note how these $\beta^{\sigma}$ fit into a commutative diagram with our zigzag evaluation maps, 
\begin{equation*}
\xymatrix{\Delta^{m+n} \times PM \ar[d]^{\beta^{\sigma} \times d_{PM}} \ar[rrr]^{ev_{k_m + k_n, \, m+n}} & &  & M^{L}\ar[d]^{\pi^{\sigma}} \\
\Delta^m \times PM \times \Delta^n \times PM \ar[rrr]^{(ev_{m,k_m}, \, ev_{n,k_n})} & & &M^{m'} \times M^{n' }}
\end{equation*}
where for the moment we use $m' = m k_m + k_m + 1$, $n' = n k_n + k_n + 1$, and $L = (m+n+1)(k_m + k_n)+1$.  We now apply $\Omega$ to the above diagram.  The main difference between this proof and a proof in the abelian setting (see Remark \ref{REM:It abelian algebra map}) is the extra care required when keeping track of the coefficients.  The details are similar to the details in the proof of Proposition \ref{Prop:b and d commute}, and we provide the following commutative diagram for the reader's convenience.
\begin{equation*}\label{eq:CurvedAlgebraMap}
\begin{tikzcd}
\Omega(M, \E)^{\otimes m'} \otimes \Omega(M, \E)^{\otimes n'}\arrow[d, "\left( \bigotimes \pi_i^* \right) \otimes \left( \bigotimes \pi_j^* \right)"]  \arrow[r, "\odot_{\sigma}"] &[10em]  \Omega(M, \E)^{\otimes L} \arrow[d, "\bigotimes \pi_{\ell}^*"] \\
\Omega(M^{m'}, \bigotimes\limits_i \pi_i^* \E) \otimes \Omega(M^{n'}, \bigotimes\limits_j \pi_j^*\E) \arrow[d, "ev^*_{k_m, m} \otimes ev^*_{k_n, n}"]  \arrow[r, "(\pi^{\sigma})^* \circ (\pi_{m'}^* \otimes \pi_{n'}^*)"] & \Omega(M^L, \bigotimes\limits_{\ell} \pi_{\ell}^* \E) \arrow[d, "ev^*_{k_m+k_n, m+n}"] \\
\bigotimes\limits_{m,n} \Omega(\Delta^{\bullet} \times PM, EV_{k_{\bullet}, \bullet}^* \E)  \arrow[d, "m \circ \left((Ad_P)_* \otimes (Ad_P)_*\right)"] \arrow[r, "(\beta^{\sigma} \times d_{PM})^* \circ (\pi_{\Delta^m \times PM}^* \otimes \pi_{\Delta^n \times PM}^*)"] & \Omega(\Delta^{m+n} \times PM, EV^*_{k_m +k_n, m+n}) \arrow[d, " (Ad_P)_*"] \\
\bigotimes\limits_{m,n} \Omega(\Delta^{\bullet} \times PM, \E_0)  \arrow[r, "m \circ (\beta^{\sigma} \times d_{PM})^* \circ (\pi_{\Delta^m \times PM}^* \otimes \pi_{\Delta^n \times PM}^*)"] & \Omega(\Delta^{m+n} \times PM, \E_0)\\
\end{tikzcd}
\end{equation*}
From here the computation, follows by summing over all $\sigma$ and taking the appropriate $(Ad_P)_*$ maps into consideration:
\begin{small}\begin{align*}
&It\left( \underline{x} \odot \underline{y} \right)\\ = &\int\limits_{\Delta^{m+n}} (Ad_P)_* \circ \left(EV^*_{k_m + k_n, m+n}\right) \left( \underline{x} \odot \underline{y} \right) =   \sum\limits_{\sigma}(-1)^{\abs{\sigma}+(\abs{\underline{x}} - m)n}  \int\limits_{\Delta^{m+n}}   (Ad_P)_* \circ \left(EV^*_{k_m + k_n, m+n}\right) \left( \underline{x} \odot_{\sigma} \underline{y} \right)\\
= &\sum\limits_{\sigma}  (-1)^{\abs{\sigma} + (\abs{\underline{x}} - m) n} \int\limits_{\Delta^{m+n}}  m \circ \left(\beta^{\sigma} \times d_{PM}\right)^* \circ \left(\bigotimes_{m,n} \pi^*_{\Delta^{\bullet} \times PM}\right) \circ m \circ \left(\bigotimes_{m,n}  (Ad_P)_* \right) \circ \left(\bigotimes_{m,n} EV^*_{k_{\bullet},\bullet} \right) \left( \underline{x} \otimes \underline{y} \right) \\
\intertext{and now we can apply \eqref{EQ:delta n delta m shuff} to the fiber we are integrating over,}
&=m \left( \ (-1)^{(\abs{\underline{x}} - m)\cdot n} \int\limits_{\Delta^m \times \Delta^n} \left(id \times d_{PM}\right)^* \circ \left(\bigotimes_{m,n} \pi^*_{\Delta^{\bullet} \times PM}\right) \left( (Ad_P)_* \left( EV^*_{k_m,m} (\underline{x}) \right) \otimes (Ad_P)_* \left( EV^*_{k_n,n} (\underline{y})\right) \right)  \right)\\
&= m \circ d_{PM}^* \circ \pi_1^* \otimes \pi_2^* \left( \ \int\limits_{\Delta^m}  (Ad_P)_* \left( EV^*_{k_m,m} (\underline{x}) \right)  \otimes  \int\limits_{\Delta^n} (Ad_P)_* \left( EV^*_{k_n,n} (\underline{y})\right) \right) = It(\underline{x}) \cdot It(\underline{y})
\end{align*}\end{small}
 
\noindent where again $m:\E_0 \otimes \E_0 \to \E$ and we recall from \eqref{eq:wedgeproduct} how $m \circ d^* \circ \pi_1^* \otimes \pi_2^*$ recovers the wedge product. 
\end{proof}\end{prop}

\begin{remark}\label{REM:It abelian algebra map}
The proof that $It: CH^{I}(\Omega(M)) \to \Omega(PM )$ is an algebra map is a bit simpler due to the lack of tracking coefficients along with the fact that we routinely apply $\mathbb{R} \otimes \mathbb{R} \simeq \mathbb{R}$, so there is no multiplication map needed. Additionally, the proof is different since we would instead rely on the commutative diagram 
\begin{equation*}
\xymatrix{\Delta^{m+n} \times PM \ar[d]^{\beta^{\sigma} \times d} \ar[rr]^{ev_{m+n}} &&  M^{m+n+2} \ar[d]^{\rho^{\sigma}} \\
\Delta^m \times \Delta^n \times PM \ar[rr]^{(ev_m, \, ev_n)} & &M^{m+2} \times M^{n+2}  }
\end{equation*}
which in fact ``mixes'' the tensor-ordering of our forms.  For this reason, the proof of Proposition \ref{zz It algebra map} does not simply reduce to the abelian/bar-complex case.
\end{remark}

\begin{theorem}\label{thm It is NGA equivalence}
The iterated integral map $It$ sits in the following commutative diagram, each of which is a homotopy equivalence in $\NGA$.
\[ \xymatrix{
\Omega_\nabla(M,\E) \ar[rr]_{\simeq}^{ev_0^*} \ar[dr]^{\simeq}_{\eta} &&\Omega_{\wt{\nabla}}(PM, \E_0)  \\
&ZZ\big( \Omega_{\nabla}(M,\E) \big) \ar[ur]^{\simeq}_{It}
}\]
\begin{proof}
First we show the diagram commutes, i.e. that $It \circ \eta = ev_0^*$. This follows from the fact that $\eta$ sends a form  $\omega \in \Omega(M,\E)$ to the single zig-and-zag
\[\eta(\omega) = { \begin{tikzpicture}[ baseline={([yshift=-.5ex]current bounding box.center)},vertex/.style={anchor=base,circle,fill=black!25,minimum size=18pt,inner sep=2pt}]
  \draw[name path=zigzag1] (0,10) to (2,9.5) to (0,9);
   
   \draw[opacity=0,name path=t1] (0, 10) to (0,9);
   \fill [name intersections={of=t1 and zigzag1, by={a, c,blank}}]
        (a) circle (2pt)
        (c) circle (2pt);
\node[label=left:{$\omega$}] at (a) {};
\node[label=left:{$1$}] at (c) {};

\draw[opacity=0,name path=t2] (2, 10) to (2,9);
   \fill [name intersections={of=t2 and zigzag1, by={b,blank}}]
        (b) circle (2pt);
\node[label=right:{$1$}] at (b) {};
\end{tikzpicture}} =  \omega \otimes 1 \otimes 1 \in    (\A \otimes(\A^{\otimes 0}\otimes \A)^{\otimes 2})[0] / \sim_{a,b} \; \subset ZZ( \A).\]
When we apply $It$ to $\omega \otimes 1 \otimes 1$, we integrate over the fiber $\Delta^0$, so we simply evaluate.  Following \eqref{EQ IT^A}, we obtain
\begin{align*}
It\left( \omega \otimes 1 \otimes 1\right) &= \int\limits_{\Delta^0} ((Ad_P)_* \circ EV_{2,0}^*)(\omega \otimes 1 \otimes 1) = \int\limits_{\Delta^0} \omega(t_0) \cdot \ptrans{}{t_1}{t_0} \cdot 1(t_1) \cdot \ptrans{}{t_0}{t_1} \cdot 1(t_0) \\
&= \omega(0) \cdot \ptrans{}{1}{0} \cdot \ptrans{}{0}{1} = \omega(0) = ev_0^*\omega.
\end{align*}
Thus, the diagram commutes.

By Proposition \ref{algebraic quasi-iso}, the map $\eta\colon \Omega_{\nabla}(M,\E) \to ZZ(\Omega_{\nabla}(M,\E))$ is a homotopy equivalence.  By Proposition \ref{prop:PMhtpyM}, $ev_0^*\colon \Omega_{\nabla}(M,\E) \to \Omega_{\wt{\nabla}}(PM, \E_0)$ is also a homotopy equivalence.  The usual 2/3 property of homotopy equivalences, shown in Proposition \ref{prop:HomotopyEquivalences} to still hold in $\NGA$, immediately implies that $It$ must also be a homotopy equivalence.
\end{proof}
\end{theorem}


\section{Relating the Zigzag algebra with the interval Hochschild complex}\label{section finale}
In this section we observe the compatibility between the curved story presented above and the well known flat story for commutative differential graded algebras. 
\begin{definition}\label{def zz to I}
If $\C$ is a commutative DGA, we have a column-collapse map $Col: ZZ(\C) \to CH^I(\C)$ defined as follows: Let $\underline{x}_{(k,n)} \in ZZ(\C)$.  Then 
$$Col(\underline{x}_{(k,n)} ) = (-1)^{\epsilon} Col^{\mathcal{L}}(\underline{x}_{(k,n)} ) \otimes Col^1(\underline{x}_{(k,n)} ) \otimes \ldots \otimes Col^n(\underline{x}_{(k,n)} ) \otimes Col^{\mathcal{R}}(\underline{x}_{(k,n)} )$$
where 
\begin{align*}
Col^{\mathcal{L}}(\underline{x}_{(k,n)} )\= & x_{(0,0)} \cdot x_{(2,n+1)}  \cdot x_{(4,n+1)} \dots \cdot x_{(k,n+1)}, \\
Col^p(\underline{x}_{(k,n)} )\=  &\prod\limits_{\substack{ i=1 \\ i \text{ odd}}}^{k-1} x_{(i,p)} \cdot \prod\limits_{\substack{ i=2 \\ i \text{ even}}}^{k} x_{(i,n-p+1)} \\
Col^{\mathcal{R}}(\underline{x}_{(k,n)} )\= &  x_{(1,0)} \cdot x_{(3,n+1)}  \cdot x_{(5,n+1)} \dots \cdot x_{(k-1,n+1)}. \\
\end{align*}
Here, $\prod$  refers to the ordered product induced by the algebra $\C$, and $\epsilon$ comes from the usual Koszul rule of changing the order of elements $x_{(i,p)}$.
\end{definition}

\begin{prop}\label{prop zz to I}
Let $\C$ be a commutative differential graded algebra.  Then $Col: ZZ(\C) \to CH^I(\C)$ is a morphism in $\NGA$.
\begin{proof}
It is straightforward to check that the differentials are compatible because of Leibniz and because we always collapse full rows.  Similarly, the shuffle products are compatible since the insertion of $1$'s in an entire column amounts to the usual shuffle after collapsing.
\end{proof}
\end{prop}

Note that the above collapse-map was not purely ad-hoc, as illustrated by the following lemma.
\begin{lemma}\label{Lemma column}
The collapse map $Col: ZZ(\Omega(M, \mathbb{R})) \to CH^I(\Omega(M, \mathbb{R}))$ from Definition \ref{def zz to I} is given, level-wise for each choice of $n \in \mathbb{N}$, by $zz^*$ where $zz:M^{n+2} \to M^{nk+k+1}$ is a particular diagonal map.  In other words, we have the commutative diagram: 
\begin{equation*}
\begin{tikzcd}
ZZ(\Omega(M))_{k,n} \ar[d, "Col"] \ar[r, hookrightarrow]& \Omega(M^{nk+k+1}) \ar[d, "zz^*"] \\
CH^I(\Omega(M))_n \ar[r, hookrightarrow]& \Omega(M^{n+2})
\end{tikzcd}
\end{equation*}
\begin{proof}
Similar to the proof of Proposition \ref{Prop:b and d commute}, this fact is another consequence of \eqref{eq:domain for dl}.
\end{proof}
\end{lemma}

\begin{prop}\label{prop Col It commute equiv}
If $\A= \Omega(M, \mathbb{R})$, we have the following commutative diagram of homotopy equivalences in $\NGA$.
\begin{equation}\label{Diagram column integral}
\begin{gathered}
\xymatrix{ZZ(\Omega(M)) \ar[d]^{Col} \ar[r]^{It} & \Omega(PM)\\
CH^I(\Omega(M)) \ar[ru]_{It} &  }
\end{gathered}\end{equation}
\begin{proof}
We observe the following commutative diagram for each choice of $n \in \mathbb{N}$.
\begin{equation}\label{Diagram column ev}\begin{gathered}
\xymatrix{M^{nk + k+1} & \Delta^n \times PM   \ar[l]_{ev_{k,n}} \ar[ld]^{ev_n} \\
M^{n+2} \ar[u]^{zz} &  }
\end{gathered}\end{equation}
When combined with the lemma above, we see that the diagram of \eqref{Diagram column integral} commutes, level-wise with respect to $n$, due to glueing together the pull-back diagram from \eqref{Diagram column ev} with the diagram from Lemma \ref{Lemma column}.
 \begin{equation*}
\xymatrix{ZZ(\Omega(M))_n \ar[d]_{Col} \ar[r] & \Omega(M^{nk+k+1})  \ar[d]^{zz^*} \ar[r]^{ev_{k,n}^*} &\Omega(\Delta^n \times PM) \ar[r]^{\int_{\Delta^n}} & \Omega(PM)\\
CH^I(\Omega(M))_n \ar[r]& \Omega(M^{n+2}) \ar[ru]_{ev_n^*} & & }
\end{equation*}
To see that the resulting triangular diagram is comprised of homotopy equivalences, we note that the map $It: CH^I(\Omega(M)) \to \Omega(PM)$ is well-known to be a homotopy equivalence in $\DGA$ (implicitly this fact is proved by \cite{C}, but a similar result is explicitly given in \cite[Lemma 3.3]{GJP}), while the map $It: ZZ(\Omega(M)) \to \Omega(PM)$ was proven to be a homotopy equivalence in Theorem \ref{thm It is NGA equivalence}. Since $\DGA  \hookrightarrow \NGA$ is a full-subcategory and the definition of homotopy equivalence is such that this inclusion preserves such equivalences, then we also have that $It$ is a homotopy equivalence in $\NGA$. By the two-out-of-three property, we conclude that the column-collapse map $ZZ(\Omega(M)) \xrightarrow{Col} CH^I(\Omega(M))$ is a homotopy equivalence.
\end{proof}
\end{prop}


\begin{thebibliography}{TWZ}

\bibitem[BD]{BD}
Jonathan Block and Calder Daenzer.
\newblock Mukai duality for gerbes with connection.
\newblock {\em J. Reine Angew. Math.}, 639:131--171, 2010.

\bibitem[Blo]{Block}
Jonathan Block.
\newblock Duality and equivalence of module categories in noncommutative
  geometry.
\newblock In {\em A celebration of the mathematical legacy of {R}aoul {B}ott},
  volume~50 of {\em CRM Proc. Lecture Notes}, pages 311--339. Amer. Math. Soc.,
  Providence, RI, 2010.

\bibitem[Che]{C}
Kuo~Tsai Chen.
\newblock Iterated path integrals.
\newblock {\em Bull. Amer. Math. Soc.}, 83(5):831--879, 1977.

\bibitem[GJP]{GJP}
Ezra Getzler, John D.~S. Jones, and Scott Petrack.
\newblock Differential forms on loop spaces and the cyclic bar complex.
\newblock {\em Topology}, 30(3):339--371, 1991.

\bibitem[Gla1]{1505.03192}
Cheyne~J.~Miller (Glass).
\newblock {T}he {Z}igzag {H}ochschild {C}omplex, 2015.
\newblock arXiv:1505.03192.

\bibitem[Gla2]{MillerThesis}
Cheyne~J.~Miller (Glass).
\newblock {\em On the {D}erivative of 2-{H}olonomy for a {N}on-{A}belian
  {G}erbe}.
\newblock ProQuest LLC, Ann Arbor, MI, 2016.
\newblock Thesis (Ph.D.)--City University of New York.

\bibitem[Gla3]{Glass}
Cheyne~J. Glass.
\newblock The derivative of global surface-holonomy for a non-abelian gerbe.
\newblock {\em Differential Geom. Appl.}, 75:101737, 2021.

\bibitem[GTZ]{MR2721877}
Gr{{\'e}}gory Ginot, Thomas Tradler, and Mahmoud Zeinalian.
\newblock A {C}hen model for mapping spaces and the surface product.
\newblock {\em Ann. Sci. {\'E}c. Norm. Sup{\'e}r. (4)}, 43(5):811--881, 2010.

\bibitem[Nic]{Nicolaescu}
Liviu~I. Nicolaescu.
\newblock {\em Lectures on the geometry of manifolds}.
\newblock World Scientific Publishing Co. Pte. Ltd., Hackensack, NJ, second
  edition, 2007.

\bibitem[Pol]{P}
Polishchuk.
\newblock Introduction to curved dg-algebras. \\
\newblock
\href{http://www.math.polytechnique.fr/SEDIGA/documents/polishchuk_ens100324.pdf}{http://www.math.polytechnique.fr/SEDIGA/documents/polishchuk\_ens100324.pdf}, 2010.

\bibitem[Qui]{QuickNotes}
Gereon Quick.
\newblock {A}lgebraic {T}opology {L}ecture {N}otes.\\
\newblock
  \href{https://folk.ntnu.no/gereonq/MA3403H2018/MA3403_Lecture_Notes.pdf}{https://folk.ntnu.no/gereonq/MA3403H2018/MA3403\_Lecture\_Notes.pdf},
  2018.

\bibitem[SW]{SW}
Urs Schreiber and Konrad Waldorf.
\newblock Smooth functors vs. differential forms.
\newblock {\em Homology Homotopy Appl.}, 13(1):143--203, 2011.

\bibitem[TWZ]{TWZ}
Thomas Tradler, Scott~O. Wilson, and Mahmoud Zeinalian.
\newblock Equivariant holonomy for bundles and abelian gerbes.
\newblock {\em Comm. Math. Phys.}, 315(1):39--108, 2012.

\end{thebibliography}

\end{document}